\documentclass[hidelinks, 12pt, oneside, reqno]{amsart}
\usepackage{dsliheader}
\usepackage{cite}
\newtheoremstyle{named}{}{}{\itshape}{}{\bfseries}{.}{.5em}{\thmnote{#3}}
\theoremstyle{named}
\newtheorem*{namedtheorem}{Theorem}

\usepackage{geometry}
\title[Uniqueness of AC $G_2$-shrinkers]{Uniqueness of Asymptotically Conical Gradient Shrinking Solitons in $G_2$-Laplacian Flow}
\author[Haskins, Khan, \& Payne]{Mark Haskins, Ilyas Khan, and Alec Payne}
\date{}
\begin{document}

\maketitle
\vspace{-.3in}

\begin{abstract}
    We prove a uniqueness result for asymptotically conical (AC) gradient shrinking solitons for the Laplacian flow of closed $G_2$-structures: If two gradient shrinking solitons to Laplacian flow are asymptotic to the same closed $G_2$-cone, then their $G_2$-structures are equivalent, and in particular, the two solitons are isometric. The proof extends Kotschwar--Wang's argument for uniqueness of AC gradient shrinking Ricci solitons. We additionally show that the symmetries of the $G_2$-structure of an AC shrinker end are inherited from its asymptotic cone; under a mild assumption on the fundamental group, the symmetries of the asymptotic cone extend to global symmetries. 
\end{abstract}

\section{Introduction}

This paper is aimed at improving the current understanding of a natural class of noncompact singularity models to Laplacian flow, a geometric flow of closed and nondegenerate 3-forms on a $7$-dimensional spin manifold. A major hope for Laplacian flow is that it will provide a parabolic approach to the construction of $G_2$-holonomy metrics on spin $7$-manifolds. The latter are interesting to Riemannian geometers because their  Ricci tensor vanishes. Currently, such $G_2$-holonomy metrics are the only known source of (nonflat) Ricci-flat metrics on compact, simply-connected manifolds of odd dimension. Existing constructions of $G_2$-holonomy metrics on compact manifolds are all based on the singular elliptic perturbation method pioneered by Joyce~\cite{Joyce96} and since developed by various authors~\cite{JoyceBook00,Kovalev03,JoyceKarigiannis,CortiHaskinsNordstromPacini,FoscoloHaskinsNordstrom}. 

In this parabolic approach, one might initially hope that the Laplacian flow exists for all time and converges in infinite time to a torsion-free $G_2$-structure, i.e.\ a nondegenerate 3-form which induces a metric with holonomy contained in $G_2$. In practice, one expects a Laplacian flow with generic initial data to encounter both finite-time and infinite-time singularities~\cite[\S 6.2]{Bryant03Remarks}. The simplest (and often most important) type of finite-time singularity model for a geometric flow is a shrinking soliton, a natural class of self-similar solution. Shrinking solitons for both Ricci flow and mean curvature flow have been extensively studied, and a variety of examples and classification results are now known: we refer the reader to~\cite{CaoSurvey09, AndrewsChowGuentherLangford20} for overviews of this extensive body of work. By contrast, the study of shrinkers in Laplacian flow (and of Laplacian flow solitons more generally) is still in its infancy; a number of examples have been constructed recently~\cite{Lauret17, Nicolini22, Fowdar:S1:inv:Laplacian:flow, HaskinsNordstrom21, KovalevLee11}, but almost no classification results are available.

A key property that distinguishes Laplacian flow from Ricci flow or mean curvature flow is the absence of compact shrinkers. This is a consequence of the fact that the volume element increases monotonically under Laplacian flow.  Therefore, for Laplacian flow an important task is to understand what asymptotic geometries can occur for complete noncompact Laplacian shrinkers and which of these geometries model finite-time singularities of the flow on a compact manifold. Once a certain type of asymptotic geometry is known to occur for shrinkers, it then becomes natural to ask what asymptotic data is needed to determine the shrinker completely. As the asymptotic geometry of a shrinker models the macroscopic structure of a singularity of the flow, unique determination of a shrinker from its asymptotics gives strong information about how the singularity forms and potentially how to continue the flow past this singularity.

 Very recently, Haskins--Nordstr\"om showed that shrinkers that are asymptotic to smooth cones arise in Laplacian flow~\cite{HaskinsNordstrom21}---as
 is well-known in the context of Ricci flow and mean curvature flow shrinkers~\cite{FeldmanIlmananKnopf03, KleeneMoller14, Leewang09, AngenentKnopf22, WangAsymptotic16}.
 Work by Lu Wang and later Bernstein on shrinking hypersurfaces in mean curvature flow~\cite{BernsteinUniquenessExpanders, Wang11}, by Kotschwar--Wang on gradient Ricci shrinkers~\cite{kotschwarwang2015}, and Khan for higher codimension mean curvature flow shrinkers~\cite{Khan20} has shown that such asymptotically conical (AC) shrinkers exhibit a perhaps surprising degree of rigidity: they are uniquely determined by their asymptotic cone. Related elliptic problems do not share such strong rigidity properties. For instance, it is well-known that AC K\"{a}hler Ricci-flat metrics with a given asymptotic cone often occur in non-trivial finite-dimensional moduli spaces. It turns out, however, that the strong rigidity properties of AC shrinkers have a parabolic origin: in strong backward uniqueness results for parabolic subsolutions. 

In upcoming work, Haskins--Juneman--Nordstr\"{o}m prove that the highly-symmetric AC Laplacian shrinkers constructed in~\cite{HaskinsNordstrom21} are determined by their asymptotic cones, within the class of shrinkers sharing the same symmetries~\cite{HaskinsJunemanNordstrom22}. The high degree of symmetry allowed these authors to establish rigidity with ODE methods. The rigidity results obtained in this special case suggested the further possibility that any AC Laplacian shrinker would be uniquely determined by its asymptotic cone, just as in Ricci flow and in mean curvature flow. Removing the symmetry assumptions puts the question beyond the scope of ODE methods. 
In this paper, we use PDE methods to prove that gradient AC Laplacian shrinkers are indeed determined by their asymptotic cones. 

We now proceed with some more precise definitions and the statements of our main results. 
Laplacian flow (introduced by Bryant in ~\cite{Bryant03Remarks}) is a one-parameter family of closed $3$-forms $t \mapsto \vp(t)$ which represent $G_2$-structures and solve
\begin{equation*}
     \frac{\partial \vp}{\partial t} = \Delta_{\vp} \vp,
\end{equation*}
for $t\in [0, T)$, where $\Delta_{\vp} \vp$ is the Hodge Laplacian of $\vp$. See~\cite{BryantXu11, LotayWeiLaplacianFlow, FineYaoReport19} for additional background on this flow. 

The gradient shrinking Laplacian soliton equation is
$$\De_\vp \vp = -\tfrac{3}{2}\vp + \cL_{\na f} \vp,$$
where $\vp$ is a closed $G_2$-structure on $M$ and $f$ is a smooth function on $M$. Any solution to the gradient shrinking Laplacian soliton equation determines a self-similarly shrinking solution to Laplacian flow. See~\cite[\S 2]{HaskinsNordstrom21} for a survey of what is known about Laplacian solitons.

We can now state the main theorem (see Definition \ref{definition asymptotically conical} for what it means for a shrinker to be asymptotic to a cone along an end):

\begin{thm}\label{theorem main}
    Let $(M_1, \vp_1, f_1)$ and $(M_2, \vp_2, f_2)$ be two gradient Laplacian shrinkers asymptotic to the closed $G_2$-cone $(\cC^\Sg, \vp_C)$ along the ends $V_1 \subset M_1$ and $V_2 \subset M_2$, respectively. Then there exist ends $W_1 \subset V_1$ and $W_2 \subset V_2$ and a diffeomorphism $\Psi: W_1 \rightarrow W_2$ such that $\Psi^*\vp_2 = \vp_1$.
\end{thm}

We briefly sketch the proof of this theorem. Broadly speaking, our proof extends the methodology employed by Kotschwar--Wang \cite{kotschwarwang2015} for Ricci shrinkers to the setting of Laplacian flow. Kotschwar--Wang reduce the question of uniqueness of AC shrinkers in Ricci flow to a parabolic backward uniqueness problem. Starting at any AC shrinking end, there is an associated Ricci flow which converges to the asymptotic cone in finite time. The goal then is to prove backward uniqueness for Ricci flows which converge to the same cone in finite time. However, the difference of two AC shrinking ends with the same asymptotic cone does not satisfy a strictly parabolic equation, and due to the inapplicability of DeTurck's trick (see~\cite[\S 1]{KotschwarBackwardUniqueness10}), there is no obvious way to break the gauge invariance of the problem. To 
overcome this problem Kotschwar--Wang construct an ODE-PDE system for the evolution of the differences of appropriate geometric quantities and prove that backward uniqueness holds for this ODE-PDE system~\cite[\S 3]{kotschwarwang2015}. They prove the parabolic backward uniqueness property for this system using Carleman estimates, as pioneered in the work of Escauriaza--Seregin--{\v{S}}ver{\'a}k on backward uniqueness results for the heat equation~\cite{EscauriazaSereginSverak}. 

However, the application of these methods in our setting is far from straightforward due to several significant differences between Laplacian flow and Ricci flow. The starting point for the proof is the fact that each $G_2$-structure $\vp$ induces a Riemannian metric $g_\vp$ on $M$, which then evolves along with $\vp$ under Laplacian flow. Indeed, we might first hope that the techniques of \cite{kotschwarwang2015} would apply to these underlying metrics after comparing the Ricci flow equation to the evolution of the metric under Laplacian flow:
\begin{equation}\label{equation evolution metrics}
         \frac{\dd g_{ij}}{\dd t} = -2 \Ric_{ij} - \frac{2}{3}|T|^2 g_{ij}- 4 T_i^{\;  l}T_{lj},
\end{equation}
where $T = T_{ij}$ is the torsion $2$-form of the closed $G_2$-structure (see \eqref{equation nabla phi and torsion} for its definition). The torsion terms in \eqref{equation evolution metrics} are second order in the metric, scale like curvature, and are not ``small" compared to the Ricci curvature. 

The additional torsion terms can have a dramatic influence on the geometry of Laplacian shrinkers. For instance, there exist examples of complete gradient shrinking Laplacian solitons due to Fowdar which have exponential volume growth and constant negative scalar curvature~\cite{Fowdar:S1:inv:Laplacian:flow}, starkly contrasting with Ricci shrinkers which have nonnegative scalar curvature and at most Euclidean volume growth~\cite{CaoZhu:JDG2010}. In fact, for a closed $G_2$-structure, the scalar curvature $R$ satisfies $R = -|T|^2$, meaning that we are always working in the setting of nonpositive scalar curvature. In the AC setting, the torsion and hence the scalar curvature are both asymptotically small. Intuitively, therefore, the AC condition ameliorates the negative scalar curvature of closed $G_2$-structures, thus somewhat constraining the geometry of AC 
Laplacian shrinkers. 

The first major departure from the approach of Kotschwar--Wang, then, is that even in the AC setting, the shrinker potential $f$ satisfies substantially weaker properties than in Ricci flow. More specifically, in the analysis of gradient Ricci shrinkers a key role is played by the following identity for the shrinker potential $f$
\[
R + \abs{\nabla f}^2 - f=0.
\]
In the AC setting, the Laplacian flow analogue  of this identity (see Lemma \ref{lemma fundamental shrinker identity}) contains an 
additional term with logarithmic growth; this turns out to change significantly the nature of the backward uniqueness problem for the ODE-PDE system that we must solve. The large error terms resulting from this phenomenon are ultimately consolidated in an error term growing logarithmically in the radial variable and lead to a necessarily much weaker Carleman estimate than found in Kotschwar--Wang (see Proposition~\ref{proposition second Carleman estimate PDE}). This requires a new approach on the PDE level to the proof of backward uniqueness, involving a careful subdivision of the domain and judicious absorption of ``bad terms" (see Proposition \ref{proposition exponential decay}). In Appendix \ref{section Appendix}, we showcase how these techniques can be used to prove the following extension of the main theorem of \cite{EscauriazaSereginSverak}:
\begin{namedtheorem}[Theorem A.1]
Let $Q_{R,T}$ be the parabolic domain $(\bR^n \setminus B_R) \times [0, T]$, and let $u: Q_{R,T} \rightarrow \bR$ satisfy
\begin{equation}
    |\De u + \dd_t u| \le M|\na u| + (N|x|^{\de} + M)|u|\;\; \text{ and }\;\; |u(x,t)| \le Me^{M|x|^2},
\end{equation}
for constants $M, T, R, C, D >0$ and $0< \de < 1/2$. Then, if $u(x,0) \equiv 0$ in $\bR^n \setminus B_R$, $u$ vanishes identically in $Q_{R,T}$. 
\end{namedtheorem}

The second major complication, beyond the differences between Ricci and Laplacian flow on the metric level, is the fact that the objects of Laplacian flow are closed nondegenerate $3$-forms and that the induced \textit{metric} structure contains strictly less information than the corresponding $3$-form. The point at which this becomes relevant is in the setup of the parabolic backward uniqueness problem in Section \ref{section setting up the parabolic problem}. In this section, we construct Laplacian flows that interpolate between an AC shrinker $\vp$ and its conical limit $\vp_C$. 

However, the flows as constructed converge \textit{a priori} only to \textit{some} closed $G_2$-cone $\vp_D$ and not necessarily to the common limit $\vp_C$ given in Theorem \ref{theorem main}. Thus, it is necessary to show that $\vp_C = \vp_D$, in order that the limit $\vp_1(t) - \vp_2(t)$ vanishes as $t \rightarrow 0$ and the backward uniqueness problem be well-defined. Kotschwar--Wang used Cheeger--Colding theory for the limits of metric spaces in order to identify the two metric cones associated to $g_D$ and $g_C$. As we are working at the level of the $3$-form, we cannot rely on this theory to make this identification. Ultimately, we overcome this obstacle by comparing the $C^k$-topologies of convergence generated by $g_D$ and $g_C$ for sections of $\La^3(TM)$ and showing that the rescaling limit of $\vp$ is the same in both of these topologies. Furthermore, we should point out that our construction is not limited to the case of 3-forms: it generalizes the reparametrization procedure in \cite{kotschwarwang2015} to flows of an $(r,s)$-tensor with the induced bundle metric.

After proving Theorem \ref{theorem main}, we also obtain a number of corollaries. Corollary \ref{corollary main} is a global statement which follows from Theorem \ref{theorem main} after using an argument for global extension of isometries~\cite[Theorem 5.2]{KotschwarWangIsometries}.

\begin{cor}\label{corollary main}
Let $(M_1, \vp_1, f_1)$ and $(M_2, \vp_2, f_2)$ be two complete gradient Laplacian shrinkers asymptotic to the closed $G_2$-cone $(\cC^\Sg, \vp_C)$ along the ends $V_1 \subset M_1$ and $V_2 \subset M_2$, respectively. Suppose that for some $x_0 \in V_1$ and $y_0  \in V_2$, the homomorphisms $\pi_1(V_1, x_0)\rightarrow \pi_1(M_1, x_0)$ and $\pi_1(V_2, y_0)\rightarrow \pi_1(M_2, y_0)$ induced by inclusion are surjective.

Then, there exists a diffeomorphism $\Psi: M_1 \to M_2$ such that $\Psi^* \vp_2 = \vp_1$.
\end{cor}

Theorem \ref{theorem main} is then used to prove Theorem \ref{theorem isometries}, which says that ends of AC shrinkers have the same symmetries as their asymptotic cones. For a closed $G_2$-structure $(M, \vp)$, we define $\Aut(M, \vp)$ to be all the diffeomorphisms $F: M \to M$ such that $F^* \vp = \vp$. The group $\Aut(M, \vp)$ describes the symmetries of the $G_2$-structure. 

\begin{thm}\label{theorem isometries}
Let $(M, \vp, f)$ be a gradient shrinking Laplacian soliton asymptotic to the closed $G_2$-cone $(\cC^\Sg, \vp_C)$ along the end $V \subset (M, \vp)$. Then, there is $r > 0$, an end $W$ of $(M, \vp)$, and a diffeomorphism $F: \cC_r^{\Sg} = \Sg \times (r,\infty) \to W$ such that $\ga \mapsto F \circ \ga \circ F^{-1}$ is an isomorphism from $\Aut(\cC^\Sg, \vp_C)$ to $\Aut(W, \vp)$.
\end{thm}
Here, $\Aut(\cC^\Sg, \vp_C)$ refers to the $G_2$-automorphisms fixing the vertex of $\cC^\Sg$. The Ricci flow analogue of Theorem \ref{theorem isometries}, which establishes an isomorphism between isometry groups, was proven by Kotschwar--Wang in~\cite{KotschwarWangIsometries} and relies on their uniqueness result for AC Ricci shrinkers. Given Theorem \ref{theorem main}, we may adapt the argument of~\cite{KotschwarWangIsometries} (see Section \ref{section automorphisms of shrinkers}). The argument involves extending known results on local real analyticity of Laplacian flow~\cite{LotayWeiRealAnalyticity} to show that the $\vp_C$ and $\vp_D$ are analytic with respect to a common real analytic structure. As a consequence, despite the fact that the asymptotic cone of a shrinker does not itself satisfy the shrinker equation, we deduce the following (see Corollary \ref{corollary common analytic structure shrinker},  \emph{cf.}~\cite[Cor.\ 3.2]{KotschwarWangIsometries}):
\textit{The asymptotic cone of an AC Laplacian shrinker must have a real analytic cross-section.} 

As a corollary of Theorem \ref{theorem isometries}, we find that under a mild topological assumption, the symmetries can be extended to the entire shrinker, not just an end.

\begin{cor}\label{corollary complete isometries}
Let $(M, \vp, f)$ be a connected, complete gradient shrinking Laplacian soliton asymptotic to the closed $G_2$-cone $(\cC^\Sg, \vp_C)$ along the end $V \subset (M, \vp)$. If for some $x_0 \in V$, the homomorphism $\pi_1(V, x_0) \rightarrow \pi_1(M, x_0)$ induced by inclusion is surjective, then $\Aut(\cC^\Sg, \vp_C)$ embeds into $\Aut(M, \vp)$.
\end{cor}

Finally, we note that the techniques developed to prove the results of this paper are sufficiently flexible that we expect our approach to be relevant beyond Laplacian flow. Further applications of these techniques may include proving similar results for other geometric flows or finding general criteria under which a flow's shrinkers satisfy an analogous uniqueness property.

\subsection*{Acknowledgments}
The authors would like to thank the Simons Foundation for its support of their research under the Simons Collaboration on Special Holonomy in Geometry, Analysis and Physics (grant \#488620, MH and AP; grant \#724071, IK). 

The authors would also like to thank Robert Bryant, Jason Lotay, and Johannes Nordstr\"{o}m for helpful comments and discussions.

\section{Preliminaries}

\subsection{Preliminaries on \texorpdfstring{$G_2$}{G2}-structures}

Throughout this paper, $M$ will denote a connected, smooth $7$-manifold. A $G_2$-structure on $M$ is a reduction of the structure group of the frame bundle to the group $G_2\subset GL(7)$. This is a purely topological condition, equivalent to $M$ being orientable and spinnable~\cite{Gray69}. It is well-known that any $G_2$-structure can be uniquely associated to a particular $3$-form $\vp \in \Omega^3(M)$ defining the $G_2$-structure, called a $G_2$-form. Such a $3$-form must satisfy the property that for each $x \in M$, there is an isomorphism $\io: T_x M \to \bR^7$ with $\io^* \vp_0 = \vp_x,$. Here, $\vp_0$ is the $3$-form on $\bR^7 = \Im(\bO)$ given by
$$\vp_0(u,v,w) := \langle u \times v, w\rangle = \langle uv, w\rangle,$$
where $\langle \cdot, \cdot \rangle$ is the Euclidean inner product and $uv$ is multiplication of $u$ and $v$ in the octonions.

Every $G_2$-structure we consider in this paper will be closed. This means that the associated $3$-form $\vp$ is a closed form, i.e.\ $d\vp=0$. If $\vp$ is also co-closed, i.e.\ $d^* \vp = 0$, then the $G_2$-structure is said to be torsion-free. This is equivalent to the $G_2$-structure having zero torsion, in the sense of the torsion of the $G$-structure~\cite{FernandezGray82}.    

For any $G_2$-structure $\vp$ on $M$, $g = g_{\vp}$ is a metric and an orientation canonically associated to $\vp$ by the relation
$$g_{\vp}(u,v) \Vol_{g_{\vp}} = \io_u \vp \wedge \io_v \vp \wedge \vp,$$
which holds for any vector fields $u,v$ on $M$. We will often suppress the subscript notation in $g_{\vp}$ unless we need to specify the dependence of the metric on the $G_2$-structure. An important fact is that if a $G_2$-structure is torsion-free, then this is equivalent to $\Hol(M, g_{\vp})\subseteq G_2$~\cite{FernandezGray82}.

There is a Hodge star $\star_{\vp}$, using the metric and orientation associated to the $G_2$-form $\vp$. The torsion of a $G_2$-structure is given by a two-tensor $T_{ij}$ satisfying
\begin{equation}\label{equation nabla phi and torsion}\na_i \vp_{jkl} = T_{i}^{\;m}(\star_{\vp} \vp)_{mjkl}.\end{equation}
If the $G_2$-structure is closed, then $T_{ij}$ is a two-form and hence is anti-symmetric. Also, in the closed case, $T_{ij}$ is divergence-free, $g^{ij} \na_i T_{jk} = 0$, and $T_{ij}$ satisfies the crucial relation
\begin{equation}\label{equation R = -|T|^2}
    R = -|T|^2,
\end{equation}
where $R$ is the scalar curvature of the metric $g_{\vp}$. See~\cite{Bryant03Remarks, KarigiannisFlows09, KarigiannisIntro20, LotayWeiLaplacianFlow} for additional details on these identities.

\subsection{Scaling of geometric quantities}

The scaling behavior of $G_2$-structures plays an important role throughout this paper, so for concreteness, we state the following basic result describing this behavior.
\begin{lem}\label{lemma scaling of metric}
If $\vp$ is a $G_2$-structure on $M$, $\Psi:M \rightarrow M$ is a diffeomorphism, and $\lambda > 0$ is a constant, consider the pulled-back and rescaled 3-form $\lambda \Psi^* \vp$, which is also a $G_2$-structure on $M$. Then, the corresponding metric $g_{\lambda \Psi^* \vp}$ is given by $\lambda^\frac{2}{3}\Psi^* g_{\vp}$. Moreover, 
the following identities are satisfied:
\begin{enumerate}
\item \textit{(Hodge Laplacian)}
\begin{equation}\label{equation scaling hodge laplacian}
    \Delta_{\la \Psi^* \vp}\la \Psi^* \vp = \la^{\frac{1}{3}}  \Psi^* (\Delta_{\vp} \vp)
\end{equation}
\item \textit{(Hodge Dual)}
\begin{equation}
    \star_{\la \Psi^* \vp} (\la \Psi^* \vp) = \la^{\frac{4}{3}} \Psi^* (\star_\vp \vp)
\end{equation}
\item \textit{(Torsion)}
\begin{equation}
    T_{\la \Psi^* \vp} = \la^{\frac{1}{3}} \Psi^* T_{\vp}
\end{equation}
where $T$ denotes the torsion tensor regarded as a $2$-tensor. 
\end{enumerate}
\end{lem}

\begin{proof}

The proofs follows from straightforward calculation in local coordinates using 
the relation between the associated metric $g_\vp$ and $\vp$. 
The result can also be seen as a special case of the behavior of a $G_2$-structure 
under a conformal rescaling, e.g as described in \cite[\S 3.1]{Karigiannis:deformations}.
We explain how the scaling of the torsion tensor $T$ (regarded as a 2-tensor) follows from the scaling behavior of the other geometric quantities. 
We can ignore the pullback $\Psi^*$ (since $\Psi$ is an isometry) and consider the torsion tensor $\tilde{T}$ corresponding to $\tilde \vp = \lambda \vp$.
Recall that, regarded as a $(1,1)$-tensor, $T$ satisfies
\[
\nabla_X \vp = T(X) \lrcorner \star_{\vp} \vp,
\]
for all smooth vector fields $X$.

Then using the fact that covariant differentiation is invariant under rescaling and the rescaling behavior for $\vp$ and $\star_{\vp} \vp$ we see that as a $(1,1)$-tensor 
the torsion is rescaled by $\lambda^{-1/3}$. So when we use the metric (which rescales by $\lambda^{2/3}$)
to regard $T$ as a 2-tensor instead it will rescale by a factor of $\lambda^{1/3}$
as claimed.

\end{proof}

\subsection{Estimates on the metric from estimates on the 3-form}

In this subsection, we get explicit estimates on the metric associated to a perturbed $G_2$-form.
\begin{prop}\label{proposition estimates on difference of metrics}
Let $\vp, \tvp$ be two $G_2$-forms on a smooth $7$-manifold $M$, and let $g, \tg$ be their corresponding induced metrics. There is an $\epsilon > 0$ such that if $|\tvp - \vp|_\vp \le \eps$, we have 
\begin{equation}\label{equation estimate on metrics}
  |\tg - g|_\vp \le  C |\tvp - \vp|_\vp,
\end{equation}
where the norm $|\cdot|_\vp$ is taken with respect to $g = g_\vp$ and $C > 0$ is an absolute constant. Alternatively, if the bound $|\tvp - \vp|_\vp \le \eps$ is not known and $|\tg - g|_{\vp} \le \ka$, then 
\begin{equation}\label{equation estimate on metrics dependent}
  |\tg - g|_\vp \le  C(1+\ka) |\tvp - \vp|_\vp,
\end{equation}
where $C>0$ is an absolute constant. 
\end{prop}
\begin{proof}
Define $\ga:= \tvp - \vp$. For any vector fields $u,v$ on $M$, 
\begin{align*}\tg(u,v) \Vol_{\tvp} &= \frac{1}{6}(\io_u \tvp)\we (\io_v \tvp) \we \tvp\\
&= \frac{1}{6}(\io_u (\vp + \ga))\we (\io_v (\vp + \ga)) \we (\vp + \ga)
\end{align*}
Then, there exists a $7$-form $\cG(u,v)$ on $M$ depending on the vector fields $u,v$ such that 
\begin{equation}\label{equation gVol difference}
\tg(u,v) \Vol_{\tvp} = g(u,v) \Vol_{\vp} + \cG(u,v)
\end{equation}
and
\begin{equation}\label{equation estimate on difference}|\cG(u,v)|_{\vp} \leq \frac{1}{6}  |\ga|_{\vp} \big(|\ga|_{\vp}^2 + 3|\ga|_{\vp} |\vp|_{\vp} + 3|\vp|_{\vp}^2\big)|u|_{\vp}|v|_{\vp},
\end{equation}
where the bars $|\cdot|_\vp$ denote the norm taken with respect to the metric $g$ induced by $\vp$.

The most direct way to estimate the difference $\tg - g$ is to expand $\Vol_{\tvp}$ around $\Vol_{\vp}$ using the expansion \cite[(6.7)]{Bryant03Remarks}. By the decomposition of $3$-forms on $M$ into irreducible representations of $G_2$, the form $\gamma$ can be expressed as
\begin{equation}\label{equation gamma decomposition} \gamma = 3b_0\vp + \star_\vp(b_1 \we \vp) + b_3, \;\;\; b_0 \in \Om^0(M), b_1 \in \Om^1(M), \text{ and  } b_3 \in \Om^3_{27}(M, \vp),
\end{equation}
where $3b_0\vp$, $\star_\vp(b_1 \we \vp)$, and $b_3$ lie in the mutually $g$-orthogonal spaces $\Om^3_{1}(M, \vp)$, $\Om^3_{7}(M, \vp)$, and $\Om^3_{27}(M, \vp)$, respectively. We now obtain bounds on the forms $b_0, b_1, b_3$. The following estimate on $b_3$ is immediate.
\begin{equation}\label{equation bound on 27 part of gamma}
    |b_3|_\vp = |\pi^3_{27}(\gamma)|_\vp \le |\gamma|_\vp,
\end{equation}
since $\pi^3_{27}:\Omega^3(M) \rightarrow \Om^3_{27}(M, \vp)$ is an orthogonal projection operator with respect to $g$. 

Now, notice that as in \cite[\S 2]{LotayWeiLaplacianFlow}, we have that $\vp_{ijk} \vp_{abl}g^{ia}g^{jb} = 6 g_{kl}$. Contracting once more, we find that $|\vp|_{\vp} = 7$, where we are taking the norm in $\Omega^3(M)$. Thus,
\begin{align}\label{equation bound on 1 part of gamma}
    |3b_0\vp|_\vp &= |\pi^3_{1}(\gamma)|_\vp \le |\gamma|_\vp \nonumber \\
    |b_0| & \le |\gamma|_\vp.
\end{align}
Finally, we estimate $b_1$. First, we note that $\star_\vp(b_1 \we \vp) = -b_1^{\#} \lrcorner \star_\vp \vp$. Now, calculating as on \cite[pg.\,226]{LotayWeiLaplacianFlow}, 
\begin{equation}\label{equation bound on 7 part of gamma}\ga^{ijk}(\star_{\vp} \vp)_{ijkl} = (\ga \lrcorner \star_{\vp} \vp)_l = 24(b_1)_l.
\end{equation}
Fix $l$. By Cauchy-Schwartz,
\begin{align*}
  |(b_1)_l| = \frac{1}{24} |\ga^{ijk} (\star_{\vp} \vp)_{ijkl}| \le 7|\ga|_\vp,
\end{align*}
and therefore $|b_1|_\vp = O(|\ga|_\vp)$. 

Finally, we appeal to the expansion \cite[(6.7)]{Bryant03Remarks}, which gives us a formula for the difference of the volume forms, in terms of $b_0, b_1,$ and $b_3$.

\begin{equation}\label{equation second order expansion of volume form}
    \star_{\tvp} 1 = \bigg(1 + 7b_0 + \bigg(14(b_0)^2 + \frac{2}{3}|b_1|^2_\vp - \frac{1}{6} |b_3|^2_\vp\bigg) + r(b_0, b_1, b_3) \bigg)\star_{\vp} 1,
\end{equation}
where, as we have shown, $b_0, b_1, b_3$ are all $O(|\ga|_{\vp})$ and $r$ has cubic decay in its arguments. Let $u,v$ be unit vectors with respect to $g$. Plugging \eqref{equation second order expansion of volume form} into \eqref{equation gVol difference}, we obtain 
\[
\tg(u,v) (1 + 7b_0 + Q(b_0, b_1, b_3))\Vol_\vp = g(u,v)\Vol_\vp  + \cG(u,v)
\]
\begin{multline*}
(\tg(u,v) - g(u,v))\star_\vp \Vol_\vp  =  (\tg(u,v)- g(u,v)) (-7b_0 - Q(b_0, b_1, b_3))\star_\vp \Vol_\vp \\ + g(u,v) (-7b_0 - Q(b_0, b_1, b_3))\star_\vp \Vol_\vp + \star_\vp \cG(u,v)
\end{multline*}
\[
\implies |\tg - g|_\vp \le  C|\tg- g|_\vp |\ga|_\vp + C|\ga|_\vp + \star_\vp \cG(u,v),
\]
where $C$ depends on the coefficients of the expansion $7b_0 + Q(b_0, b_1,b_3)$. If we set $\epsilon = 1/2C$ and assume that $|\ga|_\vp \le \eps$, then we can absorb the first term on the left-hand side into the right-hand side and obtain
\[|\tg - g|_\vp \le  C|\ga|_{\vp} + 2 \star_\vp \cG(u,v).
\]
Since $|\cG(u,v)|_\vp$ is $O(|\ga|_{\vp})$ by \eqref{equation estimate on difference} and the Hodge star is an isometry, we have 
\begin{equation*}
  |\tg - g|_\vp \le C|\tvp - \vp|_\vp.
\end{equation*}
This proves the first statement of the proposition. In case that the bound $|\ga|_\vp \le \eps$ is unavailable but the bound $|\tg- g|_\vp \le \ka$ holds, we obtain
\begin{align*}
|\tg - g|_\vp &\le  C\ka |\ga|_\vp + C|\ga|_\vp + \star_\vp \cG(u,v)\\
&\le C(1 + \ka) |\ga|_\vp.
\end{align*}
This proves the second statement in the proposition.
\end{proof}

\begin{lem}\label{lemma difference of hodge star estimate}
We can obtain the following bounds on the difference of the inverse metrics and the Hodge stars associated to $\vp$ and $\tvp$.
\[|\tg^{-1} - g^{-1}|_{\vp} \le C |\tg^{-1}|_{\vp}|\tvp - \vp|_{\vp}.
\]
and for an arbitrary $k$-form $\beta$
\begin{equation}\label{equation difference of hodge stars}
    |\star_\vp \beta - \star_{\tvp} \beta|_\vp \le C|\tvp - \vp|_\vp |\beta|_\vp.
\end{equation}
Here, $C$ depends on $\kappa$ if the second set of bounds from Proposition \ref{proposition estimates on difference of metrics} is used. 
\end{lem}

\begin{proof}
The first estimate follows immediately from the basic inequality
\[|\tg^{-1} - g^{-1}|_\vp \le C|\tg^{-1}||\tg - g|_\vp.
\]
We now prove the second statement. For $\alpha$ and $\beta$ arbitrary $k$-forms,
\begin{equation}
  (\star_\vp \beta - \star_{\tvp} \beta) \wedge \alpha  = g_\vp (\beta, \alpha)\Vol_{g_\vp} - g_{\tvp} (\beta, \alpha)\Vol_{g_{\tvp}},
\end{equation}
by the definition of the Hodge star. Taking the Hodge star again of this expression, by the same reasoning as in the proof of Proposition \ref{proposition estimates on difference of metrics}, we get
\begin{align*}g_\vp (\beta, \alpha)\star_\vp \Vol_{g_\vp} - g_{\tvp} (\beta, \alpha)\star_\vp \Vol_{g_{\tvp}} &= g_\vp (\beta, \alpha) - g_{\tvp} (\beta, \alpha) + g_{\tvp} (\beta, \alpha)O(|\ga|)\\
&= O(|\ga|_\vp |\beta|_\vp |\alpha|_\vp).
\end{align*}
Thus, 
\[
    |\star_\vp \beta - \star_{\tvp} \beta|_\vp \le C|\ga|_\vp |\beta|_\vp.
\]
\end{proof}

In the next proposition, we obtain similar estimates to those in Proposition \ref{proposition estimates on difference of metrics} for covariant derivatives of $g - \tg$ with respect to the Levi-Civita connection on $(M,g)$.

\begin{prop}\label{proposition estimates on derivatives of difference of metrics}
Under the same assumptions as Proposition \ref{proposition estimates on difference of metrics}, for $k \in \mathbb{N}$
\begin{equation}\label{equation bound on covariant derivative of tg}
|\na^k (\tg - g)|_{\vp} \le C_k  \bigg( 1 + \sum_{j=0}^k |\na^j(\tg - g)|_\vp + \sum_{j=0}^k |\na^j \vp|_\vp \bigg)\bigg( \sum_{j=0}^k |\na^j (\vp - \tvp)|_{\vp} \bigg).
\end{equation}
where $C_k > 0$ depends on the order of differentiation $k$.
\end{prop}
\begin{proof}
We return to the relation
\begin{equation}\label{equation explicit tg minus g}
\tg(u,v) - g(u,v) = \tg(u,v)\Big(-7b_0 -Q(b_0, b_1, b_3)\Big) + \star_{\vp}\cG(u,v),
\end{equation}
where
\begin{equation}
    Q(b_0, b_1, b_3) = \bigg(14(b_0)^2 + \frac{2}{3}|b_1|^2_\vp - \frac{1}{6} |b_3|^2_\vp\bigg) + r(b_0, b_1, b_3)
\end{equation}
and $\lim_{|\ga| \to 0} \frac{|r(b_0, b_1, b_3)|_{\vp}}{|\ga|_{\vp}^2} =0$. 
The proof of the proposition involves taking the $k$-times derivative of~\eqref{equation explicit tg minus g}. We illustrate the case $k = 1$: the proof in the cases $k > 1$ follows almost exactly the same steps.

Taking the covariant derivative of \eqref{equation explicit tg minus g} with respect to the Levi-Civita connection $\na = \na^g$,
\begin{multline}\label{equation first estimate on nabla tg}
    |\nabla (\tg -g)|_{\vp} \leq C|\nabla \tg|_{\vp}\Big(|b_0|_{\vp}+|Q(b_0, b_1, b_3)|_{\vp}\Big) + |\na (\star_{\vp}\cG_{ij})|_{\vp}\\ + C|\tg|_{\vp}\Big(|\na b_0|_{\vp} + \big(|b_0|_{\vp}+|b_1|_{\vp}+|b_3|_{\vp}\big)\big(|\na b_0|_{\vp}+|\na b_1|_{\vp}+|\na b_3|_{\vp}\big)\Big),
\end{multline}
where we write $\star_{\vp} \cG_{ij}$ to denote the two-tensor $(u,v)\mapsto \star_{\vp}\cG(u,v)$. Computing the covariant derivative with respect to $g$ and using a normal coordinate system centered at a point $p \in M$,
\begin{equation}\label{equation covariant derivative of star cG}
    |\na(\star_{\vp}\cG_{ij})|_{\vp} \leq C \Big((1+ |\ga|_{\vp}+ |\ga|^2_{\vp})|\na \ga|_{\vp}+(|\ga|_{\vp}+|\ga|^2_{\vp})|\na \vp|_{\vp}\Big).
\end{equation}
Assuming that $|\nabla^i \ga|_\vp < 1$ for $i = 0,1$ and using the bounds on $b_0$, $b_1$, and $b_3$ found above, we may combine \eqref{equation covariant derivative of star cG} with \eqref{equation first estimate on nabla tg} to find 
\begin{multline}\label{equation penultimate step before derivative tg bound}
    |\na (\tg - g)|_{\vp} \leq C\big((|\na \vp|_\vp |\ga|_{\vp}+|\na \ga|_{\vp})+ |\na \tg|_\vp |\ga|_\vp \\ +  |\tg|_\vp (|\na b_0|_{\vp} + |\ga|_{\vp}\big(|\na b_0|_{\vp} + |\na b_1|_{\vp} + |\na b_3|_{\vp}\big))\big).
\end{multline}

By definition, $b_0 = \langle \ga, \vp\rangle$, so 
$|\na b_0|_{\vp} \leq C(|\ga|_{\vp}+ |\na \ga|_{\vp})$.
By \cite[(2.11)-(2.12)]{LotayWeiLaplacianFlow} and the assumption that $\na \vp$ is $O(1)$ with respect to the $g$-norm, we have that $\star_{\vp}\vp$ and $\na \star_{\vp}\vp$ are also $O(1)$ with respect to the $g$-norm. Note that $\vp$ is automatically $O(1)$ with respect to the $g$-norm by \cite[(2.3)]{LotayWeiLaplacianFlow}. So, by differentiating \eqref{equation bound on 7 part of gamma}, we find that 
$|\na b_1|_{\vp} \leq C(|\ga|_{\vp}+ |\na \ga|_{\vp})$. Finally, by differentiating \eqref{equation gamma decomposition} and applying the bounds on $|\na b_0|$ and $|\na b_1|$, we find that 
$|\na b_3| \leq C(|\ga|_{\vp}+|\na \ga|_{\vp}).$

Applying these bounds on $|\na b_0|_{\vp}$, $|\na b_1|_{\vp}$, and $|\na b_3|_{\vp}$ to \eqref{equation penultimate step before derivative tg bound}, we conclude that
\begin{equation*}
|\na \tg|_{\vp} \leq C (1 + |\tg|_\vp + |\na \tg|_\vp + |\na \vp|_\vp)(|\vp - \tvp|_{\vp} + |\na (\vp - \tvp)|_{\vp}).
\end{equation*}
The proof in the case of larger values of $k$ is almost identical, with the coefficients $C_k$ depending on the number of terms that arise from the product rule under iterated differentiation.
\end{proof}

\begin{rmk}\label{remark bound on derivatives of tg-g}
From \eqref{equation bound on covariant derivative of tg}, we can obtain estimates of the form 
\[|\na^k (\tg - g)|_{\vp} \le C_k \bigg( \sum_{j=0}^k |\na^j (\vp - \tvp)|_{\vp} \bigg),
\]
where $C_k$ will depend on a priori estimates on $|\nabla^j (\tg -g)|_\vp$, $|\nabla^j (\tvp - \vp)|_\vp$, and $|\nabla^j \vp|_\vp$. These will vary depending on context, so we leave our expression in the more general form \eqref{equation bound on covariant derivative of tg}.
\end{rmk}

\section{Properties of AC \texorpdfstring{$G_2$}{G2}-shrinkers}\label{section properties of shrinkers}
\subsection{Basic features of closed AC \texorpdfstring{$G_2$}{G2}-structures}

We define the asymptotic conicality of a shrinking, closed $G_2$-end along the lines of \cite[Definition 1.1]{kotschwarwang2015} and show that this definition implies the expected asymptotics of all relevant quantities. 

\medskip

Let $(M, \vp)$ be a 7-dimensional manifold with a closed $G_2$-structure. For the moment, we consider only its Riemannian structure, i.e. the data $(M, g_\vp)$. An \textit{end} of $M$ is an unbounded connected component $V$ of $M \setminus K$ for some compact subset $K$ in $M$. We will consider ends $V$ of $M$ that are asymptotic to closed $G_2$-cones. A \textit{closed $G_2$-cone} over a closed Riemannian 6-manifold $(\Sigma^6, g_\Sg)$ is a manifold $(\cC^\Sg, \vp_C)$ such that $\cC^\Sg= (0, \infty) \times \Sigma^6$, and the 3-form $\vp_C$ is a closed $G_2$-structure whose associated metric is $g_{C} = dr^2 + r^2g_\Sg$. For $a \ge 0$, let $\cC^\Sg_a := (a, \infty) \times \Sg$ and define the \textit{dilation map} $\rho_{\lambda}: \cC^\Sg_0 \rightarrow \cC^\Sg_0$ for $\lambda > 0$, given by $\rho_\lambda(r, \sigma) := (\lambda r, \sg)$.

\begin{defn}\label{definition asymptotically conical}
Let $V$ be an end of $M$. We say that $(M, \vp)$ is \emph{asymptotic to the $G_2$-cone $(\cC^\Sg_0, \vp_C)$ along $V$} if, for some $r_0 >0$, there is a diffeomorphism $\Phi : \cC^\Sg_{r_0} \rightarrow V$ such that $\la^{-3} \rho_\la^* \Phi^* \vp \rightarrow \vp_C$ as $\la \rightarrow \infty$ in $C^3_{\text{loc}}(\cC^\Sg_{r_0}, g_C)$.
\end{defn}

\begin{lem}[\mbox{\cite[Lemma A.1]{kotschwarwang2015}}]\label{lemma decay estimates on G2 forms}
If $V$ is an end of $M$, $\Phi : \cC^\Sg_a \rightarrow V$ is a diffeomorphism for some $a > 0$, and $k \in \mathbb{N}_0$, then
\begin{equation}\label{equation ac condition varphi order k}
    \lim_{\la \rightarrow \infty} \la^{-3} \rho_\la^* \Phi^* \vp = \vp_C \;\; \text{ in } \;\;C^k_{loc}(\cC^\Sg_0, g_C),
\end{equation}
holds if and only if 
\[\lim_{b \rightarrow \infty} b^l || \nabla^{l}_{g_C} (\Phi^*\vp - \vp_C)||_{C^0(\cC^\Sg_b, g_C)} = 0
\]
for each $l =0,1,2,\ldots, k$.
\end{lem}

\begin{proof}
Let $k, \la \ge 1$ and $b > a$. The statement follows from the identity
\[\sup_{\cC^\Sg_b \setminus \cC^\Sg_{2b}}\Big|\nabla_{g_C}^k (\lambda^{-3}\rho_\la^*(\Phi^*\vp - \vp_C))\Big|_{g_C} = \sup_{\cC^\Sg_{\la b} \setminus \cC^\Sg_{\la 2b}} \la^k \Big|\nabla_{g_C}^k (\Phi^*\vp - \vp_C)\Big|_{g_C}.
\]
Note that $\vp_C$ is not necessarily parallel with respect to $\na_{g_C}$ and that $\vp_C$ is invariant under dilation and rescaling, i.e. $\lambda^{-3}\rho_\la^* \vp_C = \vp_C.$
\end{proof}

We now show that asymptotic conicality of $G_2$-structures implies asymptotic conicality on the level of metrics.

\begin{cor}\label{corollary phi conical implies g conical}
If $(M,\vp)$ is asymptotic to the $G_2$-cone $(\cC^\Sg_0, \vp_C)$ in $C^k_{\text{loc}}(\cC^\Sg_0, g_C)$, then $\la^{-2} \rho_{\la}^{*} \Phi^* g \to g_{C}$ as $\la \to \infty$ in $C^k_{\text{loc}}(\cC^\Sg_0, g_C)$.
\end{cor}
\begin{proof}
Lemma \ref{lemma decay estimates on G2 forms} implies that 
\[ \sup_{x \in \cC^\Sg_r} |\nabla^k_{g_C} (\Phi^* \vp - \vp_C)|(x) \rightarrow 0,
\]
as $r \rightarrow \infty$. Thus, for each $k$, we can find an $r_k$ such that $ ||\Phi^*\vp - \vp_C||_{C^k(\cC^\Sg_{r_k}, g_C)} < \epsilon_k$, as defined in Propositions \ref{proposition estimates on difference of metrics} and \ref{proposition estimates on derivatives of difference of metrics}. Furthermore, the conicality of $\vp_C$ yields the decay $|\nabla^j \vp_C| \rightarrow 0$ for all derivatives of order $j > 0$. Then Proposition \ref{proposition estimates on derivatives of difference of metrics} gives bounds of the form
\[|\nabla^k_{g_C}(\la^{-2} \rho_{\la}^{*} \Phi^* g - g_{C})|_{\vp_C} \le C_k|\nabla^k_{g_C}(\la^{-3} \rho_\la^* \Phi^* \vp - \vp_C)|_{\vp_C}.
\]
This bound, combined with the fact that $\lim_{b \rightarrow \infty} b^k || \nabla^{l}_{g_C} (\Phi^*\vp - \vp_C)||_{C^0(\cC^\Sg_b, g_C)} = 0$, implies that $\lim_{b \rightarrow \infty} b^l || \nabla^{l}_{g_C} (\Phi^*g - g_C)||_{C^0(\cC^\Sg_b, g_C)} = 0$. According \cite[Lemma A.1]{kotschwarwang2015}, this is equivalent to $C^k_{\text{loc}}$-convergence, which concludes the proof of the corollary.
\end{proof}

An immediate consequence of Corollary \ref{corollary phi conical implies g conical} is the following version of \cite[Lemma A.1]{kotschwarwang2015}.

\begin{cor}[\mbox{\cite[Lemma A.1]{kotschwarwang2015}}]\label{lemma decay estimates}
Let $(M, g)$ be a Riemannian manifold, $V$ an end of $M$, and $\Phi : \cC^\Sg_a \rightarrow V$ a diffeomorphism for some $a > 0$. For any nonnegative integer $k$ we say that the property $(AC_k)$ holds if
\begin{equation}\label{equation ac condition order k}
    \lim_{\la \rightarrow \infty} \la^{-3} \rho_\la^* \Phi^* \vp = \vp_C \;\; \text{ in } \;\;C^k_{loc}(\cC^\Sg_0, g_C). \tag{$AC_k$} 
\end{equation}
Then
\begin{itemize}
    \item[(a)] \eqref{equation ac condition order k} implies that 
    \[\la^{-2} \rho_{\la}^{*} \Phi^* g \to g_{C} \text{ as } \la \to \infty,\]
    which holds if and only if
    \[\lim_{b \rightarrow \infty} b^l || \nabla^{l}_{g_C} (\Phi^*g - g_C)||_{C^0(\cC^\Sg_b, g_C)} = 0
    \]
    for each $l =0,1,2,\ldots, k$.
    \item[(b)] If ($AC_0$) holds, then the metrics $\Phi^*g$ and $g_C$ are uniformly equivalent on $\overline{\cC^\Sg_b}$ for any $b>a$, and, for all $\eps >0$, there exists $b > a$ such that, for $(r, \sg) \in \cC^\Sg_b$, 
    \begin{equation}
        (1- \eps)|r-b| \le \bar{r}_b(r, \sg) \le (1+ \eps) |r-b|,
    \end{equation}
    where $\bar{r}_b(x) := d_{\Phi^*g}(x, \partial \cC^\Sg_b)$.
    \item[(c)] If ($AC_2$) holds, then for any $b >a$, there exists a constant $K = K(b, g_\sg) > 0$ such that
    \begin{equation}
        \sup_{x \in \cC^\Sg_b} (\bar{r}_b^2(x) + 1) |\Rm(\Phi^*g)|_{\Phi^*g}(x) \le K.
    \end{equation}
\end{itemize}
\end{cor}

\begin{proof}
The first part of statement (a) is a restatement of Corollary \ref{corollary phi conical implies g conical}. The rest of the proof proceeds as in \cite{kotschwarwang2015}.
\end{proof}

\subsection{Parametrizing closed AC \texorpdfstring{$G_2$}{G2}-shrinkers}

A Laplacian soliton for the Laplacian flow of closed $G_2$-structures on $M$ is a quadruple $(M,\vp, X, \lambda)$ satisfying
\begin{equation}\label{equation laplacian soliton equation}
\Delta_\vp \vp = \lambda \vp + \cL_X\vp,
\end{equation}
where $\vp$ is a closed $G_2$-structure, $\la \in \mathbb{R}$ and $X$ is a vector field on $M$. The soliton is \emph{shrinking} if $\la <0$, \emph{expanding} if $\la > 0$, and \emph{steady} if $\la = 0$. From this point onward, we consider only the case of shrinkers with dilation constant $\lambda = -3/2$, and we therefore suppress $\lambda$ in all further notation.
There is no loss of generality in this assumption since rescaling a shrinking soliton 
changes its dilation constant.

The shrinker equation then implies the following identity for the soliton metric.

\begin{prop}[\mbox{\cite[Proposition 9.4]{LotayWeiLaplacianFlow}}] 
The metric $g$ associated with a shrinking Laplacian soliton $(M, \vp, X) $ in local coordinates satisfies the equation
\begin{equation}\label{equation shrinker identity for metric}
    -\Ric_{ij} - \frac{1}{3}|T|^2 g_{ij} -2T_i^{\;k}T_{kj} = -\frac{1}{2} g_{ij} + \frac{1}{2}(\cL_{X} g)_{ij} 
\end{equation}
and the vector field $X$ satisfies $d^*(X \lrcorner \vp) = 0$.
\end{prop}

When $X$ is the gradient of some potential function $f$, we call the triple $(M,\vp, f)$ a \emph{gradient} shrinking  Laplacian soliton. For short, we may call it a $G_2$-shrinker. In this case, the identity \eqref{equation shrinker identity for metric} can be rewritten as
\begin{equation}\label{equation gradient shrinker identity for metric}
    -\Ric_{ij} - \frac{1}{3}|T|^2 g_{ij} -2T_i^{\;k} T_{kj} = -\frac{1}{2} g_{ij} + \nabla_i \nabla_j f. 
\end{equation}
We can regard the previous equation as a perturbation of the shrinking gradient Ricci soliton equations
$$-\Ric_{ij} = -\frac{1}{2} g_{ij} + \nabla_i \nabla_j f$$
by terms quadratic in the torsion tensor $T$. 
We use the previous equation to prove that any gradient shrinking Laplacian soliton with quadratic curvature decay admits a reparametrization that is compatible in a suitable sense with the level sets of the potential function $f$.

\begin{lem}[\mbox{\cite[Lemma A.2]{kotschwarwang2015}}]\label{lemma first parametrization}
Suppose that $(\cC^\Sg_{r_0}, \vp, f)$ is a gradient shrinking Laplacian soliton end with associated metric $g$, which satisfies
\begin{equation}
    (\bar{r}^2(x) + 1) |\Rm(g)| \le K
\end{equation}
for some constant $K>0$, where $\bar{r}(x) := d_g(x, \partial \cC^\Sg_{2r_0})$, and in addition $\lim_{r_i \rightarrow \infty} \bar{r}(r_i, \sg_i) \rightarrow \infty$ for all sequences $(r_i, \sg_i) \in \cC^\Sg_{r_0}$ with $r_i \rightarrow \infty$ as $i \rightarrow \infty$. Then there exists $s_0 > 0$, a closed $6$-dimensional manifold $\bar{\Sg}$, and a map $\bar{\Phi} : \cC^{\bar \Sg}_{s_0} \rightarrow \cC^\Sg_{r_0}$, where $\cC^{\bar \Sg}_{s_0} := (s_0, \infty) \times \bar{\Sg}$, 
with the following properties:
\begin{enumerate}
\item $\bar{\Phi}$ is a diffeomorphism onto its image, and $\bar{\Phi}(\cC^{\bar \Sg}_{s_0})$ is an end of the closure of $\cC^\Sg_{2r_0}$.
\item For all $(s, \bar{\sg}) \in \cC^{\bar \Sg}_{s_0}$,
\[\bar{f}(s, \bar{\sg}) = \frac{s^2}{4}, \;\;\; \text{ and }\;\;\; \frac{\partial \bar \Phi}{\partial s} = \bar f^{\frac{1}{2}} \frac{\bar \nabla \bar f}{|\bar \nabla \bar f|^2_{\bar g}}.
\]
\item There exists a constant $N > 0$ such that, for all $(s, \bar \sg) \in \cC^{\bar \Sg}_{2s_0}$,
\[N^{-1}(s-1) \le \bar{s}(s, \bar \sg) \le N(s+1) \;\;\; \text{ and }\;\;\; (s^2 + 1)|\Rm(\bar g) |_{\bar g}(s, \bar \sg) \le N.
\]
\end{enumerate}
\noindent
Here, $\bar f := f \circ \bar \Phi$, $\bar g := \bar \Phi^*g$, and $\bar s(x) := d_{\bar g}(x, \partial \cC^{\bar \Sg}_{2s_0})$.
\end{lem}

\begin{proof} 
The proof is a minor adaptation of~\cite[Lemma A.2]{kotschwarwang2015}. More specifically,  it is only necessary to find some $r_0 >0$ such that $f$ is proper and has no critical points on $\cC^\Sg_{r_0}$: at this point the proof of Kotschwar--Wang can be repeated verbatim. 
Notice that due to the quadratic curvature decay assumed, the left-hand side of \eqref{equation gradient shrinker identity for metric} decays at $O(r^{-2})$ and hence 
$
\nabla^2 f - \tfrac{1}{2}g$
decays quadratically as in the case of gradient Ricci shrinkers. 
We can therefore (mostly) follow the argument of Kotschwar--Wang: Let $x \in \cC^\Sg_{2r_0}$ and $\ga: [0,l] \rightarrow \cC^\Sg_{2r_0}$ be a unit-speed, minimizing geodesic with $\ga(0) = x_0 \in \partial \cC^\Sg_{2r_0}$, $\ga(l) = x$, and $l = \bar{r}(x)$. 
Combining the quadratic decay of $\nabla^2f - \tfrac{1}{2}g$ with 
\[
\frac{d^2}{dt^2}(f\circ \ga)(t) = \nabla^2 f (\dot \ga(t), \dot \ga(t))
\]
yields that
\[\frac{1}{2} - \frac{K}{t^2 +1} \le \frac{d^2}{dt^2} (f \circ \ga)(t) \le \frac{1}{2} + \frac{K}{t^2 + 1}.
\]
Integrating once yields
\[\frac{\bar{r}(\ga(t))}{2} - N_1 \le \frac{d}{dt} (f \circ \ga)(t) \le \frac{\bar{r}(\ga(t))}{2} + N_1,
\]
where $N_1$ depends on $K$ and $\sup_{\dd \cC^\Sg_{2r_0}}|\nabla f|$.
In particular, the lower bound implies that
\[\frac{d}{dt} (f \circ \ga)(l) = \nabla f(\ga(l))\cdot \dot \ga(l) \ge \frac{\bar{r}(x)}{4}.
\]
Thus $\nabla f$ cannot be zero for sufficiently large $x$. Integrating the previous inequalities yields
\begin{equation}\label{equation distance potential comparison}
\frac{\bar{r}^2(\ga(t))}{4} - N_2\,(\bar{r}(x) +1) \le f(x) \le \frac{\bar{r}^2(\ga(t))}{4} + N_2\,(\bar{r}(x) +1),
\end{equation}
where $N_2$ depends on $K$ and $\sup_{\dd \cC^\Sg_{2r_0}}(|f| + |\nabla f|)$. This implies that $f$ is proper. The previous inequalities together with the properness of $f$ and the non-vanishing of its gradient are enough to carry out the remainder of the argument in \cite[Lemma A.2]{kotschwarwang2015}.
\end{proof}

\begin{rmk}
Kotschwar--Wang used the fact that gradient Ricci shrinkers satisfy the identity 
$R + \abs{\nabla f}^2=f$ to establish that $f$ has no critical values on $\cC^\Sg_{r_0}$ for $r_0$ sufficiently large. There is an analogous but more complicated identity for gradient Laplacian shrinkers that we discuss shortly, but we did not make use of this identity in the proof above.
\end{rmk}

Henceforth, we will assume we are working in the parametrization described in Lemma \ref{lemma first parametrization}, but for the sake of clarity we will relabel $\cC^{\bar\Sg}_{s_0}$ as $\cC^{\bar\Sg}_{r_0}$ and continue to refer to the radial parameter on $\cC^{\bar\Sg}_{r_0}$ as ``$r$", instead of ``$s$" as in the statement of the lemma. In the rest of the section, we will write the reparametrized shrinker as $(\cC^{\bar\Sg}_{r_0}, \vp, f)$, where $\vp = \bar \Phi^* \vp$, etc., to reduce notational complexity. However, this convention will not carry over to the next section.

We now prove an estimate for the growth of an important quantity associated to a shrinker. 

\begin{lem}\label{lemma fundamental shrinker identity}
On any AC gradient Laplacian shrinker $(\cC^{\bar\Sg}_{r_0}, \vp,  f)$, 
\begin{equation}\label{equation fundamental shrinker identity}
    R + |\nabla f|^2 - f = O(\log r).
\end{equation}
\end{lem}

\begin{proof}
Let $\{x_1, \ldots, x_7\}$ be a system of geodesic normal coordinates in the neighborhood of a point $p$. Taking the divergence of~\eqref{equation gradient shrinker identity for metric} yields
\begin{align*}
    0 &= \nabla_j \left(\Ric_{ij} + \tfrac{1}{3}|T|^2 g_{ij} + 2T_i^{\;k}T_{kj} + \nabla_i \nabla_j f- \frac{1}{2}g_{ij}\right) \\
    &= \frac{1}{2} \nabla_i R + \nabla_i \Delta f - R_{jijk} \nabla_k f + \nabla_j \left(\tfrac{1}{3}|T|^2 g_{ij} + 2T_i^{\;k}T_{kj}\right),
\end{align*}
where the second equality uses the contracted Bianchi identity. Taking the trace of \eqref{equation gradient shrinker identity for metric} yields
\[
\Delta f = -\frac{2}{3}R + \frac{7}{2},
\]
where we have used the identity $-R = |T|^2$, which holds for any closed $G_2$-structure. Substituting this expression for $\Delta f$ into the previous equation we obtain
\begin{align*}
    0 &= -\frac{1}{6} \nabla_i R + R_{ik} \nabla_k f + \nabla_j \left(\tfrac{1}{3}|T|^2 g_{ij} + 2T_i^{\; k} T_{kj}\right) \\
    &= -\frac{1}{6} \nabla_i R + \bigg(- \tfrac{1}{3}|T|^2 \de_{ik} -2T_i^{\; l}T_{lk} + \frac{1}{2}\de_{ik} - \nabla_i \nabla_k f \bigg) \nabla_k f\\
    & \qquad \qquad + \nabla_j \left(\tfrac{1}{3}|T|^2 g_{ij} + 2T_i^{\; k}T_{kj}\right) \\
    &=-\frac{1}{6} \nabla_i R + \frac{1}{2}\nabla_i f - \frac{1}{2} \nabla_i |\nabla f|^2 - \tfrac{1}{3}|T|^2 \nabla_i f -2T_i^{\; l}T_{lk} \nabla_k f \\
    & \qquad \qquad + \nabla_j \left(\tfrac{1}{3}|T|^2 g_{ij} + 2T_i^{\; k}T_{kj}\right)\\
    &= -\frac{1}{2} \nabla_i R + \frac{1}{2}\nabla_i f - \frac{1}{2} \nabla_i |\nabla f|^2 - \tfrac{1}{3}|T|^2 \nabla_i f -2T_i^{\; l}T_{lk} \nabla_k f + \nabla_j 2T_i^{\; k}T_{kj},
\end{align*}
where in the second equality we again used~\eqref{equation gradient shrinker identity for metric}, and in the fourth that $R = -|T|^2$.
We can rewrite the last equality as
\begin{align*}
    \nabla_i (R + |\nabla f|^2 - f) = -\frac{2}{3}|T|^2 \nabla_i f - 4T_i^{\; l}T_{lk} \nabla_k f + 4\nabla_j T_i^{\; k}T_{kj}.
\end{align*}
At this point, a further reduction can be made. By \cite[Proposition 9.4]{LotayWeiLaplacianFlow}, $d^*(\nabla f \lrcorner \vp) = 0$. However, by \cite[(9.10)]{LotayWeiLaplacianFlow} 
\[d^*(\nabla f \lrcorner \vp) = (\curl(\nabla f))^\flat + (\nabla f \lrcorner T)^\flat. 
\]
Since $\nabla f$ is a gradient, $\curl(\nabla f) = 0$ and thus $\nabla f \lrcorner T = 0$.  This allows us to remove the term involving $\nabla f \lrcorner T$ from the right hand side.
\begin{align}\label{equation derivative of fundamental identity}
    \nabla_i (R + |\nabla f|^2 - f) = -\frac{2}{3}|T|^2 \nabla_i f + 4\nabla_j T_i^{\; k}T_{kj}.
\end{align}
Note that the right-hand side of this equation is $O(r^{-1})$ by Lemma \ref{lemma decay estimates} and Lemma \ref{lemma first parametrization}, since
\[\bigg|-\frac{2}{3}|T|^2 \nabla_i f + 4\nabla_j T_i^{\; k}T_{kj}\bigg| \le C(|T|^2|\nabla f| + |T||\nabla T|).
\]
Now, select a point $x_0$ on $\dd \cC^{\bar\Sg}_{r_0}$. For any $x \in \cC^{\bar\Sg}_{2r_0}$, take a unit-speed, minimizing geodesic $\ga : [0, l] \rightarrow \cC^{\bar\Sg}_{r_0}$ with $\ga(0) = x_0$ and $\ga(l) = x$. Integrating \eqref{equation derivative of fundamental identity}, we see 
\begin{align*}
    \bigg|\int_0^l \frac{d}{dt}\big((R + |\nabla f|^2 - f)\circ \ga(t)\big) dt \bigg| &\le  \int_0^l  |\nabla_{\dot \ga} (R + |\nabla f|^2 - f)| dt \\
    &\le \int_0^l  C\big((|T|^2|\nabla f| + |T||\nabla T|)(\ga(t)\big) |\dot \ga (t)| dt\\
    &\le \int_0^{\bar r(x)} \frac{C}{(t+1)^{-1}} dt \\
    &= C \log(r),
\end{align*}
where $\bar r(x)$ is the distance of $x$ to the boundary, and $r$ is the radial parameter. The final equality comes from the comparison (3) in Lemma \ref{lemma first parametrization}.
\end{proof}

In the following remark, we recall some well-known properties of complete gradient Ricci shrinkers 
and point out that they all fail in general on complete gradient Laplacian shrinkers.
\begin{rmk}
On any complete gradient Ricci shrinker $(M,g,f)$ with dilation constant $\lambda = -\tfrac{1}{2}$
we may (by adding a constant to the shrinker potential function $f$) assume that
\begin{equation}
\label{equation gradient Ricci shrinker conservation law}
R + \abs{\nabla f}^2 - f=0.    
\end{equation}
Then $(M,g,f)$ has the following properties:
\begin{enumerate}[left=0.25em]
    \item $M$ has finite fundamental group~\cite[Theorem 1.1]{Wylie}.
    \item The scalar curvature $R$ is strictly positive (except on the Gaussian shrinker on $\bR^n$)~\cite[Theorem 3]{PRS:MZ2012}.
    \item
    $\sqrt{f}$ is a $\tfrac{1}{2}$-Lipschitz function, that is, it satisfies $\abs{\sqrt{f(x)}-\sqrt{f(x_0})} \le \frac{1}{2}r(x)$, where $r(x)$ denotes the distance between $x$ and the reference point $x_0$.
    \item 
    $f$ satisfies the estimate $f(x) \ge \frac{1}{4}(r(x)-c_1)^2$, where $c_1$ is a positive constant depending on $n$ and on the metric $g$ on the unit ball $B_{x_0(1)}$~\cite[Proposition 2.1]{CaoZhu:JDG2010}.
    \item $g$ has at most Euclidean volume growth~\cite[Theorem 1.2]{CaoZhu:JDG2010}.
    \item $g$ has at least linear volume growth~\cite[Theorem 6.1]{Munteanu:Wang:CAG2012}.
\end{enumerate}

Fowdar~\cite[\S 5.2]{Fowdar:S1:inv:Laplacian:flow} constructed explicit complete gradient Laplacian shrinkers on $7$-manifolds of the form 
$\mathbb{R} \times P^6$ where $P^6$ is the total space of a principal $T^2$-bundle over a hyperK\"ahler $4$-manifold $(M^4,\omega_1,\omega_2,\omega_3)$
endowed with a connection $\theta=(\theta_1,\theta_2)$ whose curvature is $(\omega_2,\omega_3)$.
One particular case that arises is when $M^4$ is a flat $4$-torus $\mathbb{T}^4$ and $P^6$ is the Iwasawa manifold $I^6$,
a compact complex nilmanifold with first Betti number $b^1(I^6)=4$. The closed $3$-form $\varphi$
\[
\varphi = dt \wedge \left( \sqrt{3}\, e^{\frac{t}{6}}\, \omega_1 + \tfrac{1}{3}e^{\frac{t}{3}} \,\theta_1 \wedge \theta_2 \right) - e^{\frac{t}{3}} (\omega_3 \wedge \theta_1 - \omega_2 \wedge \theta_2)
\]
is a gradient Laplacian shrinker on $\bR \times I^6$ with dilation constant $\lambda = -\tfrac{1}{2}$ and soliton vector field $X=\tfrac{5}{2} \partial_t$.
The metric induced by~$\varphi$ is
\[
g_\varphi = dt^2 + \sqrt{3} e^{\frac{t}{6}}\, g_{\mathbb{T}^4} + \tfrac{1}{3}e^{\frac{t}{3}} (\theta_1\otimes\theta_1 + \theta_2\otimes \theta_2).
\]
$g_\varphi$ defines a complete metric on $\mathbb{R} \times I^6$ with two ends,  one of finite volume and the other with exponential volume growth; it has constant negative scalar curvature $R \equiv -\tfrac{3}{4}$; 
the soliton potential is the linear function $f=\tfrac{5}{2}t$. 

This example shows that none of the properties of complete gradient Ricci shrinkers
recalled above holds for complete gradient Laplacian shrinkers.
\end{rmk}

\section{Setting up the Parabolic Problem}\label{section setting up the parabolic problem}

In order to apply the parabolic Carleman estimates, we must obtain Laplacian flows from our $G_2$-shrinkers whose limits at the singular time coincide. Thus, when we construct a flow from our shrinker, we obtain a \textit{dynamical limit} whose relationship with the \textit{rescaling limit} given in Definition \ref{definition asymptotically conical} is unknown. In this section, we construct a flow such that the dynamical limit coincides with the rescaling limit. In particular, this implies that given any two $G_2$-shrinkers asymptotic to the same cone in the sense of Definition \ref{definition asymptotically conical}, flows can be constructed on the ends of each shrinker that coincide at the terminal time-slice.

The general outline of the argument proceeds as follows. A version of \cite[Proposition 2.1]{kotschwarwang2015}, Proposition \ref{proposition backward flow reparametrization}, establishes a Laplacian flow on the end $V \subset M$ which interpolates between the original shrinker $\vp$ and a smooth limit $\vp_D$, which is a     $G_2$-cone. Then, adapting an argument of Conlon--Deruelle--Sun in \cite{ConlonDeruelleSun19}, we show that $\vp$ has asymptotic cone $\vp_D$ with respect to the $C^k$-topology associated to the metric $g_{\vp_D}$. Finally, we show that the topology of convergence in $C^k_{loc}(V,g_{\vp_D})$ is equivalent to that of convergence in $C^k_{loc}(V,g_{\vp_C})$---and hence that $\vp_C = \vp_D$.

The following proposition is an adaptation of \cite[Proposition 2.1]{kotschwarwang2015} which constructs a shrinking flow associated to $\vp$, establishes several of its important properties, and establishes the existence of the dynamic limit $\vp_D$. In the sequel, we often consider the backward Laplacian flow and correspondingly the ``time-reversed" spacetime $\cC^{\bar\Sg}_{r_0} \times [0, \tau)$, for a time parameter $t<0$, $\tau = -t$. 

\begin{prop}\label{proposition backward flow reparametrization}
Given a gradient Laplacian shrinker $(M, \vp, f)$ asymptotic to the regular cone $(\cC^{\Sg}_0, \vp_C)$ along the end $V \subset M$, there exist constants $K_0, N_0,$ and $r_0 >0$, and a smooth family of maps $\Psi_\tau : \cC^{\Sg}_{r_0} \rightarrow  \cC^{\Sg}_{r_0}$ defined for $\tau \in (0,1]$ satisfying:
\begin{enumerate}
    \item For each $\tau \in (0,1]$, $\Psi_\tau$ is a diffeomorphism onto its image and $ \Psi_\tau(\cC^{\Sg}_{r_0})$ is an end of $\overline{\cC^{\Sg}_{r_0}}$.
    \item By an abuse of notation, let $\vp$ also denote the pullback of $\vp$ to $\cC^{\Sg}_{r_0}$. The family of closed $G_2$-forms $\vp(x,\tau):= \tau^{3/2} \Psi_\tau^* \vp(x)$ is a self-similar solution to the backward Laplacian flow
    \begin{equation}\label{equation backwards Laplacian flow}
        \frac{\dd \vp}{\dd \tau} = -\Delta_{\vp}\vp,
    \end{equation}
    for $\tau \in (0, 1]$ that extends smoothly as $\tau \searrow 0$ to $\vp(x,0) \equiv \vp_D(x)$ on $\cC_{r_0}^{\Sigma}$. Furthermore, the family of metrics $g(x,\tau) = \tau \Psi_\tau^* g(x)$ evolves by the modified backward Ricci flow
    \begin{equation}\label{equation backwards modified Ricci flow}
         \frac{\dd g_{ij}}{\dd \tau} = 2 \Ric_{ij} + \frac{2}{3}|T|^2 g_{ij} + 4 T_i^{\;  l}T_{lj},
    \end{equation}
    and each $g(\cdot,\tau)$ is the metric induced by $\vp(\cdot,\tau)$. Consequently, we find that the family $\{g(x,\tau)\}_{\tau \in (0,1]}$ extends smoothly as $\tau \searrow 0$ to the conical metric $g(x,0) \equiv g_D(x)$, induced by $\vp_D$.
    \item For all nonnegative integers $m$ we have the curvature decay estimate
    \begin{equation}\label{equation time dependent curvature decay}
        \sup_{\cC^{\Sg}_{r_0} \times [0,1]} (r^{m+2} +1) |\nabla^m \Rm(g) |_{\vp(\tau)} \le K_0,
    \end{equation}
    where $r$ denotes the radial function on the cone $\cC^{\Sg}_{0}$ and the norm and covariant derivatives are taken with respect to the metric $g(\tau)$ induced by $\vp(\tau)$.
    \item Let $f$ be the function on $\cC^{\Sg}_{r_0} \times (0,1]$ defined by $f(x,\tau) = \Psi_\tau^* f(x)$. The product $\tau f$ converges smoothly as $\tau \rightarrow 0$ to $r^2/4$ on $\overline{\cC^{\Sg}_{a}}$ for all $a > r_0$, and satisfies
    \begin{equation}\label{equation C0 estimates on f}
       r^2 - N_0\log(r) \le 4\tau f \le r^2 + N_0 \log(r), \;\;\; \tau \nabla f = \frac{r}{2}\frac{\dd}{\dd r},
    \end{equation}
    and
    \begin{equation}\label{equation derivative identities on f 1}
        \frac{\dd}{\dd \tau}(\tau f) = -\tau O(\log r), \;\; \tau^2 |\nabla f|^2 - \tau f= -\tau^2 O(\log r),
    \end{equation}
    \begin{equation}\label{equation derivative identities on f 2}
        \tau \nabla_i \nabla_j f  = \frac{g_{ij}}{2} -\tau \Ric_{ij} - \frac{\tau}{3}|T|^2 g_{ij} -2\tau T_{i}^{\;k}T_{kj}.
    \end{equation}
     on $\cC^{\Sg}_{r_0} \times (0,1]$ for some constant $N_0 >0$.
\end{enumerate}
\end{prop}

\begin{rmk}
In general, we will not distinguish between objects (e.g. $\vp, g, f$, etc.)  defined on the end $V \subset M$ and their pullbacks to $\cC^{\Sg}_{r_0}$.
\end{rmk}

\subsubsection{Proof of Proposition \ref{proposition backward flow reparametrization}}
We adopt the format of the proof in \cite{kotschwarwang2015}, separating the proof of the proposition into the proofs of several claims. For the proofs of the first two claims, we will parametrize the shrinker $(\cC^{\Sg}_{r_0},\vp, f)$ by the local diffeomorphism $\bar \Phi : \cC^{\bar \Sg}_{r_1} \rightarrow \cC^{\Sg}_{r_0}$ defined in Lemma \ref{lemma first parametrization}, and consider the pulled-back shrinker $(\cC^{\bar \Sg}_{r_1},\bar \vp, \bar f) := (\cC^{\bar \Sg}_{r_1},\bar \Phi^* \vp, \bar \Phi^* f)$.

We first prove \cite[Claim 2.3]{kotschwarwang2015} which relates the integral curves of $\bar \nabla \bar f$ to the radial trajectories. The proof is largely the same with one small modification to account for the fact that the right-hand side of \eqref{equation fundamental shrinker identity} does not vanish and in fact grows logarithmically in the radial variable.

\begin{clm}\cite[Claim 2.3]{kotschwarwang2015} 
There exists $r_2 > r_1$ depending only on $r_1$, $K$ (as in part (c) of Lemma \ref{lemma decay estimates}), and $N$ (as in part (3) of Lemma \ref{lemma first parametrization}), and a one-parameter family of local diffeomorphisms $\bar \Psi_s: \cC^{\bar \Sg}_{r_2} \rightarrow \bar \cC^{\bar \Sg}_{r_2}$ defined for $s \ge 0$, which satisfy
\begin{equation}\label{equation satisfied by reparametrizations}
    \frac{\partial \bar \Psi_s}{\partial s} = \bar \nabla \bar f \circ \bar \Psi_s \;\;\textrm{ and }\;\; \bar \Psi_0 = \mathrm{Id}_{\cC^{\bar \Sg}_{r_2}}.
\end{equation}
Moreover, for all $(r, \bar \sg) \in \cC^{\bar \Sg}_{r_2}$, $r_s := r \circ \bar \Psi_s$ satisfies
\begin{equation}\label{equation distance estimate}
    (r-1)e^{s/2} + 1 \le r_s ( r , \bar \sg) \le (r + 1)e^{s/2} - 1.
\end{equation}
\end{clm}

\begin{proof}
The proof of \cite[Claim 2.3]{kotschwarwang2015} gives the existence of $\bar \Psi_s$ and the rest of their argument applies directly with the following small change: when calculating the derivative of $r_s$ along the flow, we must use identity~\eqref{equation fundamental shrinker identity}
in place of the conservation law for gradient Ricci shrinkers
$\abs{\nabla f}^2 =f-R$. More specifically, we obtain
\begin{align*}\frac{\partial r_s}{\partial s} &= (\bar f^{-\frac{1}{2}}|\bar \nabla \bar f|_{\bar g}^2) \circ \bar \Psi_s\\
&= \left(\bar f^{-\frac{1}{2}}(\bar f - \tfrac{1}{3}\bar R + O(\log r) \,)\right)\circ \bar \Psi_s \\
&= \frac{r_s}{2} - \frac{2\bar \Psi_s^* \bar R}{3r_s} + O(r_s^{-1}\log r_s)\\
&= \frac{r_s}{2} + o(1).
\end{align*}
Thus, for $r_2 > r_1$ sufficiently large, the following inequality holds
\[\frac{1}{2}(r_s - 1) \le \frac{\partial r_s}{\partial s} \le \frac{1}{2} (r_s + 1)
\]
on $\cC^{\bar \Sg}_{r_2}$. The rest of the proof follows as in \cite{kotschwarwang2015}. 
\end{proof}

Set $s(t) := -\log(-t)$ for $t<0$ and define the family of closed $G_2$-structures
\begin{equation}\label{equation first shrinking flow}
    \bar \vp_t = (-t)^{3/2}\bar \Psi_{s(t)}^*\bar \vp
\end{equation}
on $\cC^{\bar \Sg}_{r_2} \times [-1,0)$. 
We claim that $\bar \vp_t$ is a solution of Laplacian flow with initial condition
$\bar \vp_{-1}=\bar \vp$. 

To this end we compute
\begin{align*}
    \frac{\dd \bar \vp_t}{\dd t} &= -\frac{3}{2}(-t)^{\frac{1}{2}} \bar \Psi_{s(t)}^*\bar \vp + (-t)^{\frac{3}{2}} \bar \Psi_{s(t)}^*(\cL_{\bar{\nabla} \bar f} \bar \vp) \frac{\dd s}{\dd t}\\
    &= (-t)^{\frac{1}{2}}\bar \Psi_{s(t)}^*\bigg(-\frac{3}{2} \bar \vp + \cL_{\bar{\nabla} \bar f} \bar \vp \bigg)\\
    &= (-t)^{\frac{1}{2}}\bar \Psi_{s(t)}^*(\Delta_{\bar \vp} \bar \vp) \\
    &= \big((-t)^{\frac{3}{2}}\big)^{\frac{1}{3}}(\Delta_{\bar \Psi_{s(t)}^* \bar \vp} \bar \Psi_{s(t)}^* \bar \vp) \\
    &= \Delta_{\vp_t}\vp_t,
\end{align*}
where the final equality follows from the scaling behavior of the Hodge Laplacian, \eqref{equation scaling hodge laplacian}.

The metric $\bar g_t$ associated to $\bar \vp_t$ is given by $\bar g_t = -t \bar \Psi_{s(t)}^* \bar g$. 
Since $\bar \vp_t$ evolves by Laplacian flow we know from~\cite[(3.6)]{LotayWeiLaplacianFlow} that 
$\bar g_t$ evolves by the equation
\[\frac{\partial \bar g_{t_{ij}}}{\partial t} = -2\bar \Ric_{ij} - \frac{2}{3}|\bar T|^2 \bar g_{t_{ij}} -4\tilde T_{i}^{\;k} \tilde T_{kj},
\]
where all metric quantities on the right-hand side are those associated to the time-dependent metric $\bar g_t$. This evolution equation can also be computed directly by appealing to \eqref{equation shrinker identity for metric} and the scaling identities for Ricci curvature and torsion.

We now show that $\bar \vp_t$ induces a Laplacian flow on the original parametrization. Define the following family of local diffeomorphisms:
\[ \Psi_s = \bar \Phi \circ \bar \Psi_s \circ \bar \Phi^{-1}. 
\]
By pulling back \eqref{equation first shrinking flow} by $(\bar \Phi^{-1})^*$, there exists $r_3>0$ such that we obtain the following solution to Laplacian flow:
\begin{equation}\label{equation second shrinking flow}
     \vp_t = (-t)^{3/2} \Psi_{s(t)}^* \vp,
\end{equation}
on $\cC^{\Sg}_{r_3} \times [-1,0)$ with $s(t) = - \log(-t)$ and $\vp_{-1} = \vp$, the closed shrinking $G_2$-structure under the original parametrization.

We prove a curvature estimate that is uniform in time, first for the special parametrization $(\cC^{\bar \Sg}_{r_2},\bar \vp_t)$ and then convert it into an estimate on $(\cC^{\Sg}_{r_3},\vp_t)$. 

\begin{clm}\cite[Claim 2.4]{kotschwarwang2015}\label{claim uniform curvature estimates}
For all $m \ge 0$, there exists a constant $K_m$ which depends only on $m$ and the constant $K$ in Lemma \ref{lemma decay estimates} such that the curvature tensor of $g_t = -t \Psi^*_{s(t)} g$ satisfies 
\begin{equation}\label{equation higher order curvature estimates}
    \sup_{\cC^{\Sg}_{r_3} \times [-1, 0)}(r^{m+2}(x)+1)\Big| \na^{m}\Rm(g_t)\Big|_{g_t} \leq K_m.
\end{equation}
\end{clm}
\begin{proof}
By Lemma \ref{lemma decay estimates}(c) and the scaling argument employed on \cite[pg.\,10]{kotschwarwang2015}, we have the following uniform curvature estimate
\begin{equation}\label{equation spacetime curvature bound}
    \sup_{\cC^{\bar \Sg}_{r_2} \times [-1,0)} (r^2(x) + 1)|\Rm(\bar g_t)|_{\bar g_t}(x,t) \le 8K. 
\end{equation}
By the local Shi-type derivative estimates in \cite[Theorem 6.3]{LotayWeiRealAnalyticity} and using the fact that $|\na T| \leq C|\Rm|$, 
\[|\nabla^{m}\Rm(\bar g_t)|_{\bar g_t}(x,t) \le K_m \;\;\text{ on } \;\; (\cC^{\bar \Sg}_{2r_2} \setminus \cC^{\bar \Sg}_{4r_2}) \times [-1/2, 0),
\]
for each $m \ge 0$, where $K_m$ depends on $m$ and $K$. As in \cite[\S 2.2.3]{kotschwarwang2015}, it follows from \eqref{equation distance estimate} and a scaling argument that
\begin{equation}
\sup_{\cC^{\bar \Sg}_{r_2} \times [-1, 0)}(r^{m+2}(x)+1)\Big| \na^{m}\Rm(\bar g_t)\Big|_{\bar g_t} \leq K_m.
\end{equation}
We can pull back this equation by $(\Phi^{-1})^*$ to obtain
\[\sup_{\cC^{\Sg}_{r_3} \times [-1, 0)}((2f^{1/2})^{m+2}(x)+1)\Big| \na^{m}\Rm(g_t)\Big|_{g_t} \leq K_m.
\]
The estimates \eqref{equation distance potential comparison} and Corollary \ref{lemma decay estimates} allow us to replace $2f^{1/2}$ by the radial parameter $r$ on $\cC^{\Sg}_{r_3}$ at the cost of possibly increasing $K_m$, which now depends on $N_2$. This yields the desired inequality, \eqref{equation higher order curvature estimates}. 
\end{proof}

We may now obtain a smooth limit $\vp_D = \vp_0$ on $\cC^{\Sg}_{r_3} \times \{0\}$ as $t \nearrow 0$.  This follows as in \cite[Theorem 5.1, Claim 5.3]{LotayWeiLaplacianFlow}, by using compactness of $\Sg$ and finding a finite atlas for $\cC^{\Sg}_{r_3}$.

\bigskip
We now prove the analogues of the identities obtained in Section 2.2.4 in \cite{kotschwarwang2015}. For notational clarity, we now write the original potential $f$ as $\mathring{f}$. We suppress the subscript, and simply refer to the time-dependent pair $f :=\Psi^*_{s(t)} \mathring{f}$ and $\vp := \vp_t$, which together form the structure of a shrinking gradient Laplacian soliton 
on $\cC^{\Sg}_{r_3}$ with constant $-3/(2(-t))$ in place of $-3/2$. All quantities associated to the metric $\mathring{g}$ will be denoted with a $\circ$ accent, e.g. $\mathring{R}$, $\mathring{T}$, etc.

\begin{clm}\label{claim with identities over time}
On $\cC^{\Sg}_{r_3} \times [-1,0)$, $f$ satisfies
\begin{equation}\label{equation time dependent shrinker identities}
\frac{\partial f}{\partial t} = |\nabla_{g} f|_{g}^2,  \;\;\textrm{ and }\;\; \nabla_i \nabla_j f  = - \frac{g_{ij}}{2t} -\Ric_{ij} - \frac{1}{3}|T|^2 g_{ij} -2T_{i}^{\;k}T_{kj}.
\end{equation}
\end{clm}

\begin{proof}
Differentiating $f = \Psi_{s(t)}^* \mathring f$ with respect to time yields
\begin{align*}
    \frac{\dd f}{\dd t} (x,t) =\frac{\dd}{\dd t} \big( \mathring f (\Psi_{s(t)}(x)) \big) &= (d\mathring{f})_{\Psi_{s(t)}(x)}\bigg(\frac{\dd \Psi_{s(t)}}{\dd t}(x)\bigg) \\
    &= \frac{1}{(-t)}\mathring{g} ( \mathring \nabla \mathring f, \mathring \nabla \mathring f)\\
    &= \frac{1}{(-t)^2} g((-t)\nabla f, (-t) \nabla f) \\
    &= |\nabla_g f|^2_g.
\end{align*}
In the third line, we used the rescaling properties stated in Lemma \ref{lemma scaling of metric}.

We now calculate the transformation of Ricci curvature and torsion under the time-dependent pair $(\vp(t), g(t))$.
\begin{align*}
    \Ric(g)_{ij} + \frac{1}{3}|T|_g^2g_{ij} + 2g^{kl}T_{ik}T_{lj} &= \Psi_{s(t)}^*\bigg((\Ric(\mathring g)_{ij} +  \frac{1}{3}(-t)^{-1}|\mathring T|_{\mathring g}^2(-t)\mathring g_{ij} \\
    &\qquad \qquad  \qquad + 2 (-t)^{-1}  \mathring g^{kl} (-t)^{\frac{1}{2}} \mathring T_{ik} (-t)^{\frac{1}{2}} \mathring T_{lj}\bigg)\\
    &= \Psi_{s(t)}^*\bigg( \Ric(\mathring g)_{ij} + \frac{1}{3}|\mathring T|_{\mathring g}^2\mathring g_{ij} + 2\mathring g^{kl}\mathring T_{ik}\mathring T_{lj} \bigg) \\
    &= \Psi_{s(t)}^*\bigg( \frac{1}{2}\mathring g_{ij} - \mathring \nabla_i \mathring \nabla_j \mathring f\bigg) \\
    &= \frac{1}{2(-t)} g_{ij} - \nabla_i \nabla_j f 
\end{align*}
This proves the third identity. 
\end{proof}

\begin{cor}\label{corollary dynamic potential versus distance}
If $(x,t) \in \cC^{\Sg}_{r_3} \times [-1,0)$ and $p \in \partial \cC^{\Sg}_{r_3}$, there exist time-independent constants $N_3, N_4>0$ such that $-tf$ satisfies
\begin{equation}\label{equation dynamic potential gradient versus distance}
\frac{\mathrm{dist}_{g_t}(p,x)}{2} - tN_3 \le |\nabla (-tf)| \le \frac{\mathrm{dist}_{g_t}(p,x)}{2} +tN_3,
\end{equation}
and
\begin{equation}\label{equation dynamic potential versus distance}
\frac{\mathrm{dist}_{g_t}(p,x)^2}{4} - tN_4(\mathrm{dist}_{g_t}(p,x) + 1) \le -tf \le \frac{\mathrm{dist}_{g_t}(p,x)^2}{4} +tN_4(\mathrm{dist}_{g_t}(p,x) + 1),
\end{equation}
where all quantities are measured with respect to $g_t$. 
\end{cor}

\begin{proof}
Integrating the time-dependent soliton identities exactly as in the derivation of \eqref{equation distance potential comparison} yields the estimates.
\end{proof}

\begin{prop}\label{proposition time dependent shrinker potential estimate}
On $\cC^{\Sg}_{r_3} \times [-1,0)$, $f$ satisfies
\begin{equation}\label{equation log growth of shrinker potential}
 |\nabla f|_{g}^2 - \frac{f}{(-t)} = O(\log r).
\end{equation}
\end{prop}

\begin{proof}
We examine the identity derived in Lemma \ref{lemma fundamental shrinker identity}:
\[\mathring \nabla_i (\mathring R + |\mathring \nabla \mathring f|_{\mathring g}^2 - \mathring f) = -\frac{2}{3}|\mathring T|_{\mathring g}^2 \mathring \nabla_i \mathring f + 4\mathring \nabla_j \mathring T_{i}^{\;k}\mathring T_{kj}.
\]
Pull back by $\Psi_{s(t)}^*$ and divide both sides by $(-t)$.
\begin{equation*}\frac{1}{(-t)}\nabla_i (\Psi_{s(t)}^* \mathring R + |\nabla  f|_{\Psi_{s(t)}^* \mathring g}^2 - f) = \frac{1}{(-t)}\bigg(
-\frac{2}{3}|\Psi_{s(t)}^*\mathring T|_{\Psi_{s(t)}^*\mathring g}^2 \mathring \nabla_i f + 4 \mathring \nabla_j \Psi_{s(t)}^* \mathring T_{i}^{\;k}\Psi_{s(t)}^* \mathring T_{kj})   \bigg)
\end{equation*}
Using the rescaling identities, this equation can be rewritten in terms of the time-dependent quantities as
\[
\nabla_i \bigg(R + |\nabla f|^2_g - \frac{f}{(-t)}\bigg) = \frac{1}{(-t)}\bigg(-\frac{2}{3}(-t)|T|^2_g \nabla_i (-tf) + (-t)4\nabla_jT_{i}^{\;k}T_{kj}\bigg).
\]
As before, select a point $x_0$ on $\dd \cC^{\Sg}_{r_3}$. For any $(x,t) \in \overline{\cC^{\Sg}_{2r_3}} \times [-1,0)$, take a unit speed geodesic $\ga : [0, l] \rightarrow \cC^{\Sg}_{r_3}$ with respect to $g_t$ such that $\ga(0) = x_0$ and $\ga(l) = x$. Integrating the time-dependent identity along the geodesic $\ga$, we obtain
\[
    \int_0^{l}\nabla_{\dot \ga} \bigg(R + |\nabla f|^2_g - \frac{f}{(-t)}\bigg) ds  = \int_0^l \dot \ga \lrcorner \bigg(-\frac{2}{3} |T|^2_g \nabla_i (-tf) + 4\nabla_j T_{i}^{\;k}T_{kj}\bigg)ds
\]
Let $\sE = R + |\nabla f|^2_g - \frac{f}{(-t)}$.
\begin{align*}
    |\sE(x,t) - \sE(x_0,t)| &\le C\int_0^l (|T|_g^2+|\na T|_g |T|_g)|\nabla (-tf)|_g |\dot \gamma|_g ds \\
    &\le C\int_0^l (|T|_g^2)|\nabla (-tf)|_g |\dot \gamma|_g ds \\
    &\le C \int_0^l s^{-2} s ds \\
    &= C\log(\mathrm{dist}_{g_t}(p,x)).
\end{align*}
The last inequality comes from the uniform curvature estimate \eqref{equation spacetime curvature bound}. It remains to show that $\mathrm{dist}_{g_t}(p,x)$ is comparable to the radial parameter $r$ on $\cC^{\Sg}_{r_3}$. Notice that scaling of the metric gives us
\[\mathrm{dist}_{g_t}(p,x) = \sqrt{-t} \mathrm{dist}_{\mathring g}(\Psi_{s(t)}(p),\Psi_{s(t)}(x)).
\]
Arguing as in the proof of Claim \ref{claim uniform curvature estimates} and invoking the estimates \eqref{equation distance potential comparison} and Corollary \ref{lemma decay estimates}, which tell us that $\sqrt{-t}r(\Psi_{s(t)} (x)) \simeq r$. We conclude that 
\[ \sE(x,t) = O(\log(r(x))).
\]
Note that $\sE$ is time dependent, but is uniformly bounded for all $t$ by a constant times $\log r$.
\end{proof}

\begin{lem}\label{lemma mixed time and space derivative estimates on shrinker potential}
On $\cC^{\Sg}_{r_3} \times [-1,0)$,
\begin{equation}\label{equation for -tf} \frac{\partial}{\partial t}(-tf) = -t(\sE - R),    
\end{equation}
\begin{equation}\label{equation higher spatial derivatives of shrinker potential}
|\nabla^k \sE(x,t)| = O(r(x)^{-k})\;\;\;\text{for }k\ge 1,
\end{equation}
and
\begin{equation}\label{equation mixed time and space derivatives of shrinker potential}
|\partial_t \nabla^k \sE(x,t)| = O(r(x)^{-(k+1)}) \;\;\text{for } k\ge 0
\end{equation}
\end{lem}

\begin{proof}
We prove the first inequality using the first identity in Claim \ref{claim with identities over time}.
\begin{align*}
    -t\left(|\nabla f|^2 + \frac{f}{t}\right) &= -t(\sE - R)\\
    -t(\frac{\partial f}{\partial t} + \frac{f}{t}) &= -t(\sE - R)\\
    \frac{\partial}{\partial t}(-tf) &= -t(\sE - R).
\end{align*}
Now, recall the formula
\begin{equation}\label{equation first derivative of shrinker potential}
    \nabla_i \bigg(R + |\nabla f|^2_g - \frac{f}{(-t)}\bigg) = -\frac{2}{3}|T|^2_g \nabla_i (-tf) + 4\nabla_jT_{i}^{\;k}T_{kj},
\end{equation}
and note that by the asymptotically conical condition,  
\[
|\na_i \sE| = \bigg| -\frac{2}{3}|T|^2_g \nabla_i (-tf) + 4\nabla_jT_{i}^{\;k}T_{kj} \bigg|_{g_t} = O(r^{-1}).
\]
By the asymptotics of $(g_t)^{-1}$,
\[
|\na_{l_1} \cdots \na_{l_k} \sE| = \bigg|\na_{l_1} \cdots \na_{l_k} \bigg( -\frac{2}{3}|T|^2_g \nabla_i (-tf) + 4\nabla_jT_{i}^{\;l}T_{lj} \bigg) \bigg|_{g_t} = O(r^{-k-1}).
\]
This yields \eqref{equation higher spatial derivatives of shrinker potential}. Finally, we differentiate \eqref{equation first derivative of shrinker potential} with respect to time
\begin{align*}
    \partial_t \nabla_i \bigg(R + |\nabla f|^2_g - \frac{f}{(-t)}\bigg) &= \partial_t \bigg( -\frac{2}{3}|T|^2_g \nabla_i (-tf) + 4\nabla_jT_{i}^{\;k}T_{kj}\bigg) \\
    &=  -\frac{2}{3}\partial_t (|T|^2_g) \nabla_i (-tf) + -\frac{2}{3} (|T|^2_g) \nabla_i \partial_t (-tf) \\
    & \qquad \qquad + 4\partial_t (\nabla_jT_{i}^{\;k}T_{kj})
\end{align*}
The evolution equation for the torsion given in \cite[(3.9)]{LotayWeiLaplacianFlow} implies that 
\[|\partial_t (|T|^2_g) \nabla_i (-tf)| = O(r^{-3}), \;\;\text{and }  |\partial_t (\nabla_jT_{i}^{\;k}T_{kj})| = O(r^{-5}).
\]
By \eqref{equation for -tf},
\begin{align*}
-\frac{2}{3} (|T|^2_g) \nabla_i \partial_t (-tf) &= -\frac{2}{3} (|T|^2_g) \nabla_i( -t(\sE - R)) + -\frac{2}{3} (|T|^2_g) [\na_i, \dd_\tau](-tf) \\
&= O(r^{-3}).
\end{align*}
This proves \ref{equation mixed time and space derivatives of shrinker potential} for the cases $k \ge 1$. The same argument of integrating along geodesics as in Proposition \ref{proposition time dependent shrinker potential estimate} is sufficient to prove the case $k = 0$.
\end{proof}

We now prove that $-tf$ extends to a smooth function on the space-time $\cC^{\Sg}_{r_3} \times [-1,0]$.

\begin{lem}
As $t \rightarrow 0^-$, the family $-tf(x,t)$ converges to a function $q \in C^\infty(\cC^{\Sg}_{r_3})$ in the $C^\infty_{loc}$-topology. Furthermore, $q$ satisfies
\begin{equation}\label{equation identities for q}
    \bigg|q(x) - \frac{r^2(x)}{4}\bigg| \le N \log r(x), \;\;|\nabla q|_{g_D}^2 = q, \;\text{ and }\nabla^2_{g_D} q = \frac{1}{2}g_D,
\end{equation}
where $g_D = g(0)$. Finally, the metric $g_D$ is conical with radial distance function $r_{g_D} = 2 \sqrt{q}$.
\end{lem}

\begin{proof}
Let $a \in (-1,0)$. Integrating \eqref{equation for -tf}, we obtain 
\begin{align}\label{equation integrated eqn for -tf}
    \int_{-1}^a\frac{\partial}{\partial t}(-tf)dt &= \int_{-1}^{a}-t(\sE - R)dt
\end{align}
Observe that given any spatial compact set $K \subset \cC^{\Sg}_{r_3}$, $\sE - R$ is uniformly bounded on the space-time $K \times [-1,0)$. In particular, the right hand side of the above equation is integrable on $\{x\} \times [-1,a)$ for all $a$ and fixed $x \in K$. We simplify
\begin{align*}
    -af(x,a) - f(x,-1) &= \int_{-1}^{a}-t(\sE - R)(x)dt \\
    -af(x,a) &= f(x,-1) + \int_{-1}^{a}-t(\sE - R)(x)dt.
\end{align*}
Taking the limit as $a \rightarrow 0$, we see that $-tf$ converges locally uniformly to some continuous function $q$. To obtain the local uniform convergence of all higher derivatives, we spatially differentiate the equation \eqref{equation integrated eqn for -tf} by $\nabla^m$. Lemma \ref{lemma mixed time and space derivative estimates on shrinker potential} gives us estimates on $\nabla^m (\sE - R)$ that are uniform in time and space on any compact set $K \times [-1,0]$. We can then argue in the exact same way as above to obtain the local uniform convergence of all derivatives. Thus, $-tf$ converges locally smoothly as $t \to 0$ to the limit function $q$, and 
\begin{equation*}
    \int_{-1}^0\frac{\partial}{\partial t}(-tf)(x,t) dt = \int_{-1}^{0}-tO(\log r(x))dt \implies q(x) - f(x,-1) = O(\log r(x)) 
\end{equation*}
By the comparison of $r$ and $f$,
\[ \bigg|q(x) - \frac{r^2(x)}{4}\bigg| \le N \log r(x).
\]
This allows us to see that $q$ is proper and positive on sufficiently distant regions. Multiplying \eqref{equation log growth of shrinker potential} by $t^2$ gives us
\[|\nabla (-tf)|^2 + tf = t^2 O(\log r) \implies |\nabla q|_{g_D}^2 = q,
\]
where $g_D = g(0)$. A similar argument with the third identity of Claim \ref{claim with identities over time} gives us
\[\nabla^2_{g_D} q = \frac{1}{2}g_D.
\]
All of the conditions listed in Section 2.2.4 of \cite{kotschwarwang2015} are satisfied, and thus, we have that $g_0$ is conical with radial distance function $r_{g_D} = 2 \sqrt{q}$. That is, there exists a diffeomorphism $\hat{\Phi}: \cC^{\hat \Sg}_{r_4} \to \cC^{\Sg}_{r_3}$ such that
$$(q \circ \hat{\Phi})(\hat{r}, \hat{\sg}) = \frac{\hat r^2}{4}, \text{ and } \hat{\Phi}^*(g_D) = \hat{g} := d\hat r^2 + \hat{r}^2 g_{\hat{\Sg}},$$
where $\hat \Sg$ is a closed $6$-manifold which is not necessarily the same as $\Sg$.
\end{proof}

\begin{lem}
There exists an isometry $F: (\cC^{\Sg}_0, g_C) \rightarrow (\cC^{\Sg}_0, g_D)$ from the rescaling limit to the dynamical limit. 
\end{lem}

\begin{proof}
By integrating the evolution equation for the metric
\[\int_{-1}^0 \frac{\dd \hat {g}_{ij}}{\dd t} dt = \int_{-1}^0 -2\widehat \Ric_{ij} - \frac{2}{3}|\hat T|^2g_{ij} - 4\hat T_{i}^{\; l}\hat T_{lj}dt,
\]
we can estimate the difference
\[|\hat \Phi^*(g(-1)) - \hat g_c|_{\hat g_c} \le \frac{N}{\hat r^2 + 1}. 
\]
The argument given in Section 2.3 of \cite{kotschwarwang2015} can be repeated verbatim to yield an isometry $\hat F$ between $(\cC^{\Sg}_0, g_C)$ and $(\cC^{\hat \Sg}_0, \hat g)$. The composition $F := \hat \Phi \circ \hat F$ yields the desired isometry  $F: (\cC^{\Sg}_0, g_C) \rightarrow (\cC^{\Sg}_0, g_D)$.
\end{proof}

The remaining difficulty is to show that $\vp_C = \vp_D$. Unlike \cite{kotschwarwang2015}, we cannot simply pull back the flow $\vp(t)$ by $F$ to obtain a flow limiting to $\vp_C$---this is due to the fact that $F$ may be an isometry but not an automorphism of $G_2$-structures. Instead, we show that the dynamical limit of the flow we have chosen must coincide exactly with the rescaling limit. 

Henceforth, we drop the $\mathring \vp$ notation and denote the initial, static $G_2$-structure simply by $\vp$ and its associated metric by $g$. The time slices at time $t$ of the flow of $G_2$-structures will be denoted by $\vp_t$ and their associated metrics by $g_t$. 

\begin{lem}\cite[Claim 3.9]{ConlonDeruelleSun19}\label{lemma dynamical limit of metrics is rescaling limit} For all $x \in \cC^{\Sg}_{r_3}$ and $k \in \bZ_{\ge 0}$,
\begin{equation}\label{equation decay of metric to dynamic limit}
    |(\nabla^{g})^k( g- g_D)|_{ g}(x) \le C_k r_{g_D}^{-2-k},
\end{equation}
where $ g = g_{-1}$ is the initial shrinker metric.
\end{lem}

\begin{proof}
The proof is identical to that of \cite[Claim 3.9]{ConlonDeruelleSun19} after replacing $-2\Ric(g_t)$ by 
$$-2\Ric(g_t) - \frac{2}{3}|T|^2g_t - 4T*T.$$
Claim \ref{claim uniform curvature estimates} and the uniform comparability of the metrics $\{g_t\}_{t \in [-1,0)}$ allows us to make this substitution without further modifications.
\end{proof}

The next step is to prove an analogous result for the difference $\vp - \vp_D$ and its derivatives. Ultimately, by Lemma \ref{lemma decay estimates on G2 forms} this will tell us that $\vp$ is asymptotic to $\vp_D$ in the $C^k_{loc}(C^{\Sg}_0, g_D)$-topology. Then, using the fact that $(\cC^{\Sg}_0, g_C)$ and $(\cC^{\Sg}_0, g_D)$ are isometric, we show that this topology is equivalent to the $C^k_{loc}(C^{\Sg}_0, g_C)$-topology. Thus, the limits $\vp_C$ and $\vp_D$ must be the same, finally proving that the dynamical limit and the rescaling limit coincide. 

\begin{lem}\label{lemma dynamical limit of G2 structures is rescaling limit} For all $x \in \cC^{\Sg}_{r_3}$ and $k \in \bZ_{\ge 0}$,
\begin{equation}\label{equation decay of G2 form to dynamic limit}
    |(\nabla^{ g})^k( \vp- \vp_D)|_{ g}(x) \le C_k r_{g_D}^{-2-k},
\end{equation}
where $\vp = \vp_{-1}$ is the initial $G_2$-shrinker.
\end{lem}

\begin{proof}
We again follow the proof of \cite[Claim 3.9]{ConlonDeruelleSun19}, but we must be more careful to check the validity of each step. By the equation for Laplacian flow and the uniform equivalence of the metrics $g_t$, we have
\begin{equation}\label{equation decay to dynamical limit k equals 0}
 | \vp - \vp_D|_{ g}(x) \le \int_{-1}^0 \bigg|\frac{\dd \vp}{\dd t}(t)\bigg|_{ g} dt \le C \int_{-1}^0 \big|\Delta_{\vp_t} \vp_t \big|_{ g} dt
\end{equation}
By the uniform equivalence of the metrics $g_t$, it is sufficient to estimate $|\Delta_{\vp_t}\vp_t|_{g_t}$. Recall the operator $i_\vp : S^2T^*M \rightarrow \La^3 T^*M$ from symmetric 2-tensors to 3-forms which is defined locally by
\[ i_\vp (h) = \frac{1}{2} h_i^l \vp_{ljk} dx^i \we dx^j \we dx^k.\]
It is shown in \cite[(2.20)-(2.21)]{LotayWeiLaplacianFlow} that the Hodge Laplacian satisfies
\[ \Delta_\vp \vp = i_\vp(h),\]
where $h$ is a symmetric 2-tensor given by
\[h_{ij} = -\Ric_{ij} - \frac{1}{3}|T|^2g_{ij} - 2 T_{i}^{\;k} T_{kj}.
\]
By \cite{Bryant03Remarks}, $|i_\vp(h)|^2 = 8|h|^2$, so it suffices to estimate the magnitude of $|h_t|_{g_t}$. By Claim \ref{claim uniform curvature estimates}, we know that each summand in $h_{ij}$ has quadratic decay and thus 
\[ |\Delta_{\vp_t}\vp_t|_{g}(x) \le C |\Delta_{\vp_t}\vp_t|_{g_t}(x) = C|h|_{g_t}(x) \le C d_{g_t}(p,x)^{-2},
\]
for $x \in \cC^{\Sg}_{r_3}$ and $p \in \dd \cC^{\Sg}_{r_3}$. Applying the estimate \eqref{equation dynamic potential versus distance} then yields the desired estimate
\[|\Delta_{\vp_t}\vp_t|_{ g}(x) \le Cr_{g_D}(x)^{-2}.
\]
Plugging into \eqref{equation decay to dynamical limit k equals 0} gives the estimate
\[| \vp - \vp_D|_{ g}(x) \le Cr_{g_D}(x)^{-2},
\]
which completes the proof of the case $k=0$. For the case $k=1$, we again follow the basic outline of \cite{ConlonDeruelleSun19}.
\begin{align*}
\partial_t|\nabla^{ g} (\vp_t -  \vp)|^2_{ g}(x) &\ge -2|\nabla^{ g} h(\vp_t)|_{ g}(x)|\nabla^{ g} (\vp_t -  \vp)|_{ g}(x)\\
&\ge -2 \bigg( |\nabla^{g_t} h(\vp(t))|_{ g}(x) \\
& \qquad \qquad + \big( \big|\big(\nabla^{ g} - \nabla^{g_t}\big)h(\vp_t)\big|_{ g}(x)\big)\bigg)|\nabla^{ g} (\vp_t -  \vp)|_{ g}(x)\\
&\ge -C\bigg( d_{g_t}(p,x)^{-3} + \big( \big|\big(\nabla^{ g} - \nabla^{g_t}\big)h(\vp_t)\big|_{ g}(x)\big)\bigg)|\nabla^{ g} (\vp_t -  \vp)|_{ g}(x)\\
&\ge -C\big(r_{g_D}(x)^{-3} + |\nabla^{ g}( g - g_t)|_{ g}(x)|h(\vp_t)|_{ g}(x))\big)|\nabla^{ g} (\vp_t -  \vp)|_{ g}(x)
\end{align*}
Here we invoke Proposition \ref{proposition estimates on derivatives of difference of metrics} to obtain an inequality of the form
\[
|\na^{ g} ( g - g_t)|_{ g} \le C  \bigg( 1 + \sum_{j=0}^1 |(\na^{ g})^j( g - g_t)|_{ g} + \sum_{j=0}^1 |(\na^{ g})^j  \vp|_{ g} \bigg)\bigg( \sum_{j=0}^1 |(\na^{ g})^j ( \vp - \vp_t)|_{ g} \bigg).
\]
By Lemma \ref{lemma dynamical limit of metrics is rescaling limit}, the terms $\sum_{j=0}^1 |(\na^{ g})^j( g - g_t)|_{ g}$ are uniformly bounded on $\cC^{\Sg}_{r_3}$ and the asymptotic conicality of $ \vp$ allows us to bound the summands $|(\na^{ g})^j  \vp|_{ g}$. The $k = 0$ case tells us that $| \vp - \vp_D|_{ g}(x) \le Cr_{g_D}(x)^{-2}$, and we know that $|h(\vp_t)|_{ g} = O(r_{g_D}^{-2})$. These inequalities combine to give the following estimate
\[|\nabla^{ g}( g - g_t)|_{ g}(x)|h(\vp_t)|_{ g}(x)) \le C(|\na^{ g} ( \vp - \vp_t)|_{ g}(x) + r_{g_D}(x)^{-4}).
\]
Plugging this into the our previous calculation yields
\begin{align*}
\partial_t|\nabla^{ g} (\vp_t -  \vp)|^2_{ g}(x)
&\ge -C\big(r_{g_D}(x)^{-3} + |\na^{ g} ( \vp - \vp_t)|_{ g}(x)\big)|\nabla^{ g} (\vp_t -  \vp)|_{ g}(x)\\
&\ge -C|\nabla^{ g} (\vp_t -  \vp)|^2_{ g}(x) -Cr_{g_D}(x)^{-6},
\end{align*}
where the final line follows from Young's inequality. For the purposes of integrating this inequality, we consider the backward flow with time parameter $\tau = -t$ and note that on the interval $\tau \in (0,1]$, the inequality 
\[\partial_\tau|\nabla^{ g} (\vp_\tau -  \vp)|^2_{ g}(x) \ge -C|\nabla^{ g} (\vp_\tau -  \vp)|^2_{ g}(x) -Cr_{g_D}(x)^{-6}
\]
holds, as all of our estimates involved taking norms and thus were insensitive to sign. Since $|\nabla^{ g} (\vp_\tau -  \vp)|^2_{ g}(x) = 0$ when $\tau = 1$, integrating this differential inequality between $\tau \in (0,1)$ and $1$ gives the estimate
\[|\nabla^{ g} (\vp_\tau -  \vp)|^2_{ g}(x) \le C\rho(x)^{-6},
\]
for all $x \in \cC^{\Sg}_{r_3}$ and $\tau \in (0,1]$, for a uniform constant $C>0$. Sending $\tau \searrow 0$ gives the lemma. 
\end{proof}

\begin{cor}
The initial closed $G_2$-structure $\vp$ is asymptotic to $\vp_D$ in the $C^k_{loc}(C^{\Sg}_0, g_D)$-topology.
\end{cor}

\begin{proof}
The corollary follows immediately from Lemma \ref{lemma dynamical limit of G2 structures is rescaling limit} and Lemma \ref{lemma decay estimates on G2 forms}.
\end{proof}

\begin{prop}\label{prop dynamical limit equals rescaling limit}
The dynamical limit and the rescaling limit coincide. That is,
\[\vp_C \equiv \vp_D \;\; \text{ on }\cC^{\Sg}_0.
\]
\end{prop}

\begin{proof}
The smooth, vertex-preserving isometry $F: (\cC^{\Sg}_0, g_C) \rightarrow (\cC^{\Sg}_0, g_D)$ must be generated by an isometry $F|_{\Sg \times \{1\}}$ of the link. Thus, one can take a finite atlas $\{U_k\}_{k=1}^m$ of the link $\Sg$ such that $g_C|_\Sg$ and $g_D|_\Sg$ are comparable metrics when restricted to each chart $U_k$. Extend this to a countable atlas of $\cC^{\Sg}_0$ given by $\{U_k \times (1/j, j)\}_{j,k}$, where $k \in \{1, \ldots, m\}$ and $j \in \bZ \cap [2,\infty)$. As $g_D$ and $g_C$ are isometric, they are again comparable on each chart $U_k \times (1/j, j)$.

This means that the topology of convergence in $C^k_{loc}(U_k \times (1/j, j), g_C)$ is equivalent to that in  $C^k_{loc}(U_k \times (1/j, j), g_D)$ on each chart $U_k \times (1/j, j)$. Since these charts form an atlas for the cylinder $\cC^{\Sg}_0$, the $C^k_{loc}(\cC^{\Sg}_0, g_C)$-limit of $\la^{-3}\rho_\la^*\vp$ as $\la \rightarrow \infty$ must coincide with its $C^k_{loc}(\cC^{\Sg}_0, g_D)$-limit. Since these are $\vp_C$ and $\vp_D$ respectively, $\vp_C \equiv \vp_D$ on $\cC^\Sg_0$. 
\end{proof}

Reset $r_0 := r_3$. We conclude that given any two $G_2$-structures $\vp$ and $\tvp$ asymptotic to $\vp_C$, we may construct backward Laplacian flows $\vp_\tau$ and $\tvp_\tau$ on $\cC^\Sg_{r_0}$, for $\tau \in [0,1]$, such that $\vp_1 = \vp$, $\tvp_1 = \tvp$, and $\vp_\tau$ and $\tvp_\tau$ coincide at $\tau = 0$ (i.e. $\vp_0 = \tvp_0 = \vp_C$). Then, for the flows thus constructed, to prove Theorem \ref{theorem main}, it suffices to prove the following theorem.

\begin{thm}\label{theorem parabolic version main theorem}
Let $(\cC^\Sg_{r_3}, \vp_\tau)$ and $(\cC^\Sg_{r_3}, \tvp_\tau)$ be the shrinking flows constructed above. Then there exists $r' \ge r_3$ and $\tau' \in (0,1)$ such that $\vp_\tau \equiv \tvp_\tau$ on $\cC^\Sg_{r'} \times [0, \tau']$.
\end{thm}

\subsection{Some identities in preparation for Carleman estimates}

In this section, we prove $G_2$-Laplacian flow versions of various lemmas and identities in \cite[\S 4]{kotschwarwang2015}.

Consider a flow $\vp(x,\tau)$ associated to the gradient shrinker $(M, \vp, f)$, constructed as in the proof of Proposition \ref{proposition backward flow reparametrization}, and let $f(x,\tau)$ denote the pullback of the potential function, $\Psi_\tau^*f$. As in \cite[Lemma 4.5]{kotschwarwang2015}, we introduce the $\tau$-dependent radial function 
\begin{equation}\label{definition time dependent radial function}
     h(x,\tau) := \left\{
     \begin{array}{lr}
       2\sqrt{\tau f(x,\tau)} & \textrm{for}\;\;\; \tau > 0; \\
       r(x) & \textrm{for}\;\;\; \tau = 0.
     \end{array}
   \right.
\end{equation}
By identity \eqref{equation C0 estimates on f}, we see immediately that 
\begin{equation}\label{equation comparison of h and r}
    \frac{1}{2}r(x) \le h(x,\tau) \le 2r(x)
\end{equation}
holds.

\begin{lem}\label{lemma estimates on time dep distance function}
On $\cC_{r_0}^{\Sigma} \times (0, 1]$, the derivatives of $h$ satisfy
\begin{equation}\label{equation derivative identities on h 1}
    \nabla h = \frac{2\tau}{h} \nabla f, \;\;\; h \nabla_i \nabla_j h = g - \nabla_i h \nabla_j h + 2\tau \bigg(-\Ric_{ij} - \frac{1}{3} |T|^2g_{ij} - 2T_{i}^{\;k}T_{kj}\bigg),
\end{equation}
\begin{equation}\label{equation derivative identities on h 2}
  |\nabla h|^2 = 1 + \frac{4\tau^2}{h^2}(\sE - R), \;\;\; h \Delta h = 7 -\frac{4}{3}\tau R - |\nabla h|^2,
\end{equation}
and
\begin{equation}
\label{equation partial tau h}
   \frac{\partial h}{\partial \tau} = \frac{h}{2\tau}(1 - |\nabla h|^2) = \frac{2\tau}{h}(R-\sE).
\end{equation}
\end{lem}

\begin{rmk}
Recall that estimates for the growth of the quantity $\sE(x,\tau)$ and all of its derivatives are given in Lemma \ref{lemma mixed time and space derivative estimates on shrinker potential}. In particular, there exists $C>0$, independent of time, such that $|\sE(x,\tau)| \le C\log r$ on the space-time $\cC^\Sg_{r_0} \times (0,1]$.
\end{rmk}

\begin{proof}
The first identity in~\eqref{equation derivative identities on h 1} follows directly from the definition. The second identity in~\eqref{equation derivative identities on h 1} is a consequence of \eqref{equation time dependent shrinker identities}, the time-dependent gradient shrinker equation for the metric. The first identity in~\eqref{equation derivative identities on h 2} is derived by squaring the first identity in \eqref{equation derivative identities on h 1} and substituting \eqref{equation log growth of shrinker potential}. That is,
\begin{align*}
    |\na h|^2 &= \frac{4\tau^2}{h^2}|\na f|^2 \\
    &= \frac{4\tau^2}{h^2}\bigg(\frac{f}{\tau} + \sE - R\bigg)\\
    &= 1 + \frac{4\tau^2}{h^2}(\sE - R).
\end{align*}

The second identity in~\eqref{equation derivative identities on h 2} is the trace of the second identity in~\eqref{equation derivative identities on h 1}, where we use the fact that $|T|^2 = - R$ and $T_{ij}$ is antisymmetric for closed $G_2$-structures. The last identity~\eqref{equation partial tau h}  follow by combining~\eqref{equation derivative identities on f 1},~\eqref{equation derivative identities on h 1}, and~\eqref{equation derivative identities on h 2}. Indeed,

\begin{align*}
    \frac{\dd h}{\dd \tau} = 2 \frac{\dd}{\dd \tau} \sqrt{\tau f(x,\tau)} &= \frac{1}{\sqrt{\tau f(x,\tau)}} \frac{\dd}{\dd \tau}\big(\tau f(x,\tau)\big) \\
    &= \frac{2}{h}\bigg( f(x,\tau) + \tau \frac{\dd f}{\dd \tau}\bigg)\\
    &= \frac{2}{h\tau}\big( \tau f(x,\tau) - \tau^2 |\nabla f|^2\big)\\
    &= \frac{2}{h \tau}\bigg(\frac{h^2}{4} - \frac{h^2}{4}|\nabla h|^2 \bigg)\\
    &= \frac{h}{2\tau} (1 - |\nabla h|^2).
\end{align*}
\end{proof}

We now prove our weakened version of \cite[Lemma 4.6]{kotschwarwang2015}.

\begin{lem}
There exists a universal constant $C$ such that
\begin{equation}\label{equation comparability h and r}
    |h(x,\tau) - r(x)| \le \frac{C \tau^2 \log{r(x)}}{r(x)}
\end{equation}
for all $(x,\tau) \in \cC_{r_0}^{\Sigma}\times [0, \tau_0).$
\end{lem}
\begin{proof}
Fixing any $x \in \cC^\Sg_{r_0}$, multiplying both sides of~\eqref{equation partial tau h} by $2h$, and integrating with respect to $\tau$ yields
\begin{align*}
     \int_0^\tau 2h\frac{\partial h}{\partial \tau} d\tau &= \int_0^\tau 4\tau (R-\sE) d\tau \\
     \int_0^\tau \frac{\partial h^2}{\partial \tau} - \frac{\partial}{\partial \tau}r^2(x) d\tau &= \int_0^\tau 4\tau (R-\sE) d\tau \\
     |h^2(x,\tau) - r^2(x) - (h^2(x, 0) - r^2(x))| &\le C\tau^2 \log r.
\end{align*}
Observe that by definition $h^2(x, 0) - r^2(x) = 0$. Then simplify
\[
\abs{h(x,\tau) - r(x)} \abs{h(x, \tau) + r(x)} \le C\tau^2 \log r(x).
\] 
The comparability of $r(x)$ and $h(x,\tau)$ yields the claim.
\end{proof}

\subsection{An ODE-PDE system}\label{section ODE PDE}

In this section, we construct an ODE-PDE system which will be used to prove the backward uniqueness problem Theorem \ref{theorem parabolic version main theorem}. One flow of $G_2$-structures and its associated quantities are denoted by letters with tildes, and another flow of $G_2$-structures and its associated quantities are denoted by letters without tildes. These two $G_2$-structures come from $\vp_{\tau}$ and $\tvp_{\tau}$ in Theorem \ref{theorem parabolic version main theorem}. The ODE-PDE system is constructed from evolution equations for differences of geometric quantities from each of these two $G_2$-structures.

Consider the following definitions. We adopt notation for these quantities to be as consistent as possible with our source for evolution equations~\cite[\S 6.2]{LotayWeiLaplacianFlow}.

\begin{align*}\label{equation definitions}
    \ga&:= \vp - \tvp\\
    \mu&:= g - \tg\\
    A&:=\na - \tna\\
    B&:= \na A\\
    F&:= T-\tT\\
    V&:= \na T - \tna \tT\\
    S&:= \Rm - \tRm\\
    Q&:=\na \Rm - \tna \tRm\\
    W&:= \na^2 T - \tna^2 \tT
\end{align*}
The quantity $A$ is the $(2,1)$-tensor $A_{ij}^k = \Ga_{ij}^k - \tGa_{ij}^k$. Now, we find differential inequalities for the evolution of norms of the above quantities. All norms will be taken with respect to $g_t = g(t)$, i.e.\ $|\cdot| = |\cdot|_{g_t}$. Recall that $\tau = -t$. Let $K$ and $\tilde{K}$ be the constants from part (c) of Lemma \ref{lemma decay estimates} associated to $g$ and $\tg$, respectively. 

\begin{lem}\label{lemma evolution equations for ODE PDE system}
There exists $r_0$ and $C>0$, depending on $K$ and $\tilde{K}$, such that the following inequalities hold on $\cC_{r_0}^{\Sg} \times [-1,0]$:
\begin{equation*}\label{equation evolution for gamma}
    \Big|\frac{\dd}{\dd t}\ga \Big| \leq C|V| + \frac{C}{r}|A|,
\end{equation*}
\begin{equation*}\label{equation evolution for metric}
    \Big|\frac{\dd}{\dd t}\mu \Big| \leq C|S| + \frac{C}{r}\Big(|\ga| + |F|\Big),
\end{equation*}
\begin{equation*}\label{equation evolution for A}
    \Big|\frac{\dd}{\dd t}A \Big| \leq C|\na S| +\frac{C}{r}\Big(|A|+|\ga|+ |V|+|F|\Big),
\end{equation*}
\begin{equation*}\label{equation evolution for B}
    \Big|\frac{\dd}{\dd t}B\Big| \leq C|\na Q| + \frac{C}{r}\Big(|A| + |B|+|\ga|  + |V| +|W| +|F|\Big),
\end{equation*}
\begin{equation*}\label{equation evolution for F}
    \Big|\frac{\dd}{\dd t}F \Big| \leq C|\na V| +\frac{C}{r}\Big(|A| +|\ga|+|V|+|F|+|S|\Big),
\end{equation*}
\begin{equation*}\label{equation evolution for V}
    \Big|\frac{\dd}{\dd t}V - \De V\Big| \leq \frac{C}{r}\Big(|A| + |B| + |\ga|+|V|+|W| +|F|+ |S|+|Q|\Big),
\end{equation*}
\begin{equation*}\label{equation evolution for S}
    \Big|\frac{\dd}{\dd t}S - \De S\Big| \leq \frac{C}{r}\Big(|A| + |B|+|\ga|+ |V| +|W| +|F|+|S|\Big),
\end{equation*}
\begin{equation*}\label{equation evolution for Q}
    \Big|\frac{\dd}{\dd t}Q - \De Q\Big| \leq \frac{C}{r}|\na W| + \frac{C}{r}\Big(|A| +|B|+ |\ga| +|V| +|W|  + |F|+|S|+|Q| \Big),
\end{equation*}
\begin{multline*}\label{equation evolution for W}
    \Big|\frac{\dd}{\dd t}W - \De W\Big| \leq \frac{C}{r}\Big(|\na Q| + |\na W|\Big)\\
    +\frac{C}{r}\Big(|A| + |B|+|\ga|+ |V|+|W|+ |F|   +|S|+|Q| \Big),
\end{multline*}
and
\begin{multline}|\ga|+|\mu|+|A| + |B|+ |F|+|V|+|S|+|Q|+|W|\\
+ |\na V| + |\na S| +|\na Q| + |\na W|\leq \frac{C}{r^2}\end{multline}
\end{lem}
\begin{proof}
We begin by verifying the last statement on the decay of all quantities. By Proposition \ref{prop dynamical limit equals rescaling limit} applied to Lemmas \ref{lemma dynamical limit of metrics is rescaling limit} and \ref{lemma dynamical limit of G2 structures is rescaling limit}, we have that 
\begin{equation}\label{equation decay of derivatives of g-tg}|\na^{k}(g - \tg)| \leq |\na^k (g-g_c)| + |\na^k (g_c - \tg)| \leq C_k r^{-2-k},\end{equation}
and
\begin{equation}\label{equation decay of derivatives of vp-tvp}|\na^{k}(\vp - \tvp)| \leq |\na^k (\vp-\vp_c)| + |\na^k (\vp_c - \tvp)| \leq C_k r^{-2-k},\end{equation}
and the constants $C_k$ are independent of time, which follows from the arguments in Lemmas \ref{lemma dynamical limit of G2 structures is rescaling limit} and \ref{lemma dynamical limit of G2 structures is rescaling limit}. By the arguments of Lemmas \ref{lemma dynamical limit of metrics is rescaling limit} and \ref{lemma dynamical limit of G2 structures is rescaling limit}, the estimates~\eqref{equation decay of derivatives of g-tg} and~\eqref{equation decay of derivatives of vp-tvp} also hold for $g_t - \tg_t$ and $\vp_t - \tvp_t$. Now,
$$\Ga(g_s) - \Ga(g_t) = g_t^{-1}\ast \na (g_t - g_s).$$
Taking the limit as $s \to 0$, using Proposition \ref{prop dynamical limit equals rescaling limit}, and then taking the norm, we find that 
$$|\Ga(g_c) - \Ga(g_t)| \leq |\na(g_t - g_c)|.$$
Using Lemma \ref{lemma dynamical limit of metrics is rescaling limit} and applying the triangle inequality to $A = \Ga(g_t) - \Ga(\tg_t)$, we find that 
\begin{equation}\label{equation decay on difference of Christoffel symbols}
|A| \leq \frac{C}{r^3}.
\end{equation}
By the definition of torsion~\cite[(2.11)]{LotayWeiLaplacianFlow},
$$T_i^{\; j} = \frac{1}{24}\na_i \vp_{lmn} (\star_{\vp} \vp)^{jlmn},$$
and so by adding and subtracting the terms with and without tildes,
\begin{multline*}T - \tT= (g^{-1} - \tg^{-1})\ast \na \vp \ast (\star_{\vp} \vp) +  \na \vp \ast (\star_{\vp}(\vp - \tvp)) + \na \vp \ast \big((\star_{\vp} - \star_{\tvp})\tvp\big)\\ + \na(\vp - \tvp) \ast (\star_{\tvp} \tvp) +  A \ast \tvp \ast (\star_{\tvp} \tvp),\end{multline*}
where $\ast$ refers to contractions in both $g$ and $\tg$. Note that the AC condition gives uniform bounds on $|\tvp|$, $|\tg|$, $\abs{\star_{\tvp}\tvp}$, and $|\na \vp|$. Using~\eqref{equation decay of derivatives of vp-tvp}, ~\eqref{equation decay on difference of Christoffel symbols}, and Lemma \ref{lemma difference of hodge star estimate}, we find that
\begin{equation}\label{equation decay on difference of torsions estimate}
    |T - \tT| \leq \frac{C}{r^2}
\end{equation}

Now, by Proposition \ref{proposition backward flow reparametrization}, we know that $|\na^m \Rm| = O(r^{-2})$ for all $m \geq 0$. By~\eqref{equation decay on difference of Christoffel symbols} and~\eqref{equation decay on difference of torsions estimate}, we find that $|B|$, $|V|$, and $|W|$ are $O(r^{-2})$ as well. It follows immediately from the estimates on derivatives of curvature and torsion that $|\na V|$, $|\na S|$, $|\na W|$, and $|\na Q|$ are also $O(r^{-2})$. Putting it all together, we conclude that
\begin{multline}\label{equation spacetime O(r^-2) decay of all quantities}\sup_{\cC_{r}^{\Sg} \times [-1, 0]} \Big(|\ga|+|\mu|+|A| + |B|+ |F|+|V|+|S|+|Q|+|W|\\
+ |\na V|+|\na S| + |\na Q| + |\na W|\Big) \leq \frac{C}{r^2}.\end{multline}
Note that every constant $C$ or $C_k$ used thus far in this argument has not depended on time, which allows us to take the supremum over $t \in [-1,0]$ in~\eqref{equation spacetime O(r^-2) decay of all quantities}.

We may now begin to verify the stated differential inequalities, using the evolution equations of Lotay--Wei for these quantities~\cite[Lemma 6.2]{LotayWeiLaplacianFlow} as well as some additional estimates.

By Proposition \ref{proposition estimates on difference of metrics} and bounds on $|\tg|$ due to the AC condition, we know that $|\ga|\leq C |\mu|$. Using the fact that $R = -|T|^2$, the decay of curvature~\eqref{equation time dependent curvature decay}, and~\eqref{equation decay on difference of torsions estimate}, we find that for some $r_0$, on $\cC_{r_0}^{\Sg}$,
$$|T|, |\tT| \leq \frac{C}{r}.$$
Applying these estimates to the calculations for $\frac{\dd \ga}{\dd t}$ and $\frac{\dd \mu}{\dd t}$ in~\cite[Lemma 6.2]{LotayWeiLaplacianFlow}, we find the two stated differential inequalities for these quantities. 

By~\cite[(6.13)]{LotayWeiLaplacianFlow},
\begin{multline}\label{equation Lotay Wei evolution for A}
\frac{\dd A}{\dd t} = g^{-1}\ast \na S + \tg^{-1}\ast \tT \ast V +  \na T \ast F + \tRm \ast A\\
 + (\tg^{-2}\ast \tT \ast \tna \tT + \tg^{-1}
 \ast\tna \tRm + \tg^{-1}\ast \na T \ast T) \ast \mu,
\end{multline}
where the $\ast$ here represents simple contractions in the indicated tensors. If $\cS$ is a $(k,l)$-tensor, then
$$|\cS|^2_{g} \leq |\cS|^2_{\tg} + C(k,l, |\tg|)|\mu||\cS|^2_{g}$$
where we use the fact that for $r_0$ large enough, $|\mu|^m \leq |\mu|$ on $\cC_{r_0}^{\Sg}$ for $m\geq 1$. Since $|\mu| \ll 1$ in $\cC_{r_0}^{\Sg}$, we find that
$$|\cS|^2_{g} \leq C |\cS|^2_{\tg},$$
and this implies that we may switch between taking the norm with respect to $g$ and with respect to $\tg$. This implies that $|\tna^k \tRm|$ is $O(r^{-2-k})$ and $|\tna^k \tT|$ is $O(r^{-1-k})$, using expressions for $\tna^k \tT$ in terms of curvature quantities~\cite[\S 2.3]{LotayWeiLaplacianFlow}. We may now estimate every term in~\eqref{equation Lotay Wei evolution for A} and we find our third differential inequality for $\frac{\dd A}{\dd t}$.

Now, by Remark \ref{remark bound on derivatives of tg-g}, $|\na \mu| \leq C (|\ga|+|\na \ga|)$. We estimate the term $|\na \ga|$ by applying ~\eqref{equation nabla phi and torsion},
\begin{multline*}
\na \ga =  \na \vp - \tna \tvp + A \ast \tvp =  (g^{-1} - \tg^{-1})\ast T \ast (\star_{\vp}\vp) + (T - \tT) \ast (\star_{\vp}\vp)\\
+ \tT \ast (\star_{\vp}(\vp - \tvp)) +\tT \ast ((\star_{\vp} - \star_{\tvp})\tvp)  + A \ast \tvp.
\end{multline*}
By Lemma \ref{lemma difference of hodge star estimate}, we find that 
\begin{equation}\label{equation na ga bound}
    |\na \ga| \leq C(|\ga|+|F| + |A|),
\end{equation}
and so
\begin{equation}\label{equation na mu bound}
    |\na \mu| \leq C(|\ga| + |F| + |A|),
\end{equation}
on $\cC_{r_0}^{\Sg}$. Also, on $\cC_{r_0}^{\Sg}$,
\begin{equation}\label{equation na F bound}
    |\na F| = |\na T - \tna \tT - (\na - \tna)\tT| \leq |V| + C|A||\tT| \leq |V| + \frac{C}{r}|A|,
\end{equation}
\begin{equation}\label{equation na V bound}
    |\na V| = |\na^2 T - \tna^2 \tT - (\na - \tna)\tna \tT| \leq |W| + C|A||\tna \tT|\leq |W| + \frac{C}{r^2}|A|,
\end{equation}
\begin{equation}\label{equation na S bound}
    |\na S| = |\na \Rm - \tna \tRm - (\na - \tna)\tRm| \leq |Q| + C|A||\tRm| \leq |Q| + \frac{C}{r^2}|A|.
\end{equation}
By the argument of ~\cite[Lemma 6.6]{LotayWeiLaplacianFlow}, we have that on $\cC_{r_0}^{\Sg}$,
\begin{align}\Big|\frac{\dd B}{\dd t}\Big| &\leq \Big| \na \frac{\dd A}{\dd t}\Big| + |A|\Big| \frac{\dd g}{\dd t}\Big|\nonumber\\
& \leq \Big| \na \frac{\dd A}{\dd t}\Big| + \frac{C}{r}|A|\nonumber\\
\intertext{By differentiating~\eqref{equation Lotay Wei evolution for A}, and writing $\na \tg^{-1} = \tg^{-1}\ast \tg \ast \na (\tg - g)$,}
&\leq C(|\na^2 S|) + \frac{C}{r}(|A| + |B| + |F|+ |\na F| + |\na V| + |V| + |\na \mu| + |\mu|)\nonumber\\
\intertext{Applying~\eqref{equation na mu bound}, ~\eqref{equation na F bound}, and~\eqref{equation na V bound},}
&\leq C(|\na^2 S|) + \frac{C}{r}(|A| + |B| + |\ga|+ |V|+|F|+ |W| )\nonumber\\
&\leq C|\na Q| + \frac{C}{r}\Big(|A| + |B|+|\ga|  + |V| +|W|+|F|\Big),\nonumber
\end{align}
where the last line follows by
\begin{align*}
    |\na^2 S| &= |\na(\na \Rm - \tna \tRm) - \na((\na - \tna)\tRm)|\\
    &\leq |\na Q| + C|B||\tRm| + C|A||\na \tRm|\\
    &\leq |\na Q| + C|B| |\tRm| + C|A|^2|\tRm| + C |A||\tna \tRm|\\
    &\leq |\na Q| + \frac{C}{r^2}(|A| + |B|).
\end{align*}
This completes the derivation of the differential inequality for $\frac{\dd B}{\dd t}$.

By the derivation in~\cite[Lemma 6.2]{LotayWeiLaplacianFlow},
\begin{multline*}
    \frac{\dd F}{\dd t} = \na V + \tna \tT \ast A + (\tT+\tT \ast (\star_{\tvp}\tvp))\ast S + \tT \ast \tvp \ast V+ \Rm \ast \tT \ast (\star_{\tvp} \tvp - \star_{\vp} \vp)\\
    + \na T \ast T \ast \ga  + (\Rm + \Rm \ast (\star_{\tvp}\tvp) + T^2 + \tT^2 +T\ast \tT) \ast F
\end{multline*}
Using Lemma \ref{lemma difference of hodge star estimate} and the known decay on torsion and curvature, the differential inequality for $\frac{\dd F}{\dd t}$ follows.

Now, applying the estimates we have found to the evolution equation in~\cite[Lemma 6.2]{LotayWeiLaplacianFlow}, just as in our previous derivations, we find
\begin{align}\label{equation estimate on V}
    \Big|\frac{\dd V}{\dd t} - \De V\Big| - &\big|\text{div}\big((g^{ab}\na_b - \tg^{ab}\tna_b)\tna_i\tT_{jk}\big)\big|\nonumber\\
    &\leq \frac{C}{r}(|A|+ |\na S|+ |F|+ |\ga|+|V|+|S| + |\na V|)\nonumber\\
\intertext{Applying~\eqref{equation na S bound} and~\eqref{equation na V bound},}
    &\leq \frac{C}{r}(|A|+  |\ga|+ |V|+ |W|+|F|+ |S| + |Q|)
\end{align}
Now, we estimate the divergence term as
\begin{align}\label{equation div estimate on na T}\big|\text{div}\big((g^{ab}\na_b - \tg^{ab}\tna_b)\tna_i\tT_{jk}\big)\big|& = \big|g^{-2}\ast \na^2 \tna \tT - g^{-1} \tg^{-1}\ast \na \tna^2 \tT - g^{-1}\ast\na \tg^{-1} \ast\tna^2 \tT\big|\nonumber\\
&\leq  \big|g^{-2} \ast \na^2 \tna \tT - g^{-1} \tg^{-1}\ast \na \tna^2 \tT\big|\nonumber\\
&\qquad + |\na (g^{-1} -\tg^{-1})  \tna^2 \tT|,\nonumber\\
\intertext{When $|\mu|$ is small enough, we bound $|\na (g^{-1} - \tg^{-1})|$ by $C|\na \mu|$ to find that}
&\leq  \big|g^{-2} \ast \na^2 \tna \tT - g^{-1} (g^{-1} +(\tg^{-1} - g^{-1}))\ast \na \tna^2 \tT\big| \nonumber\\
&\qquad+ \frac{C}{r}|\na \mu|\nonumber\\
&\leq \big|g^{-2} \ast \na^2 \tna \tT - g^{-2} \ast \na \tna^2 \tT\big|+ \frac{C}{r}(|\mu|+|\na \mu|)\nonumber\\
&\leq \big|\na(\na - \tna)\tna \tT\big|+ \frac{C}{r}(|\mu|+|\na \mu|)\nonumber\\
&\leq \frac{C}{r}(|\mu|+|\na \mu| +|A|+|B|),\nonumber\\
\intertext{and from ~\eqref{equation na ga bound} and~\eqref{equation na mu bound},}
&\leq \frac{C}{r}(|\ga| + |F|+|A|+ |B|).\end{align}

The differential inequality for $\frac{\dd V}{\dd t}$ now follows from~\eqref{equation na V bound} and ~\eqref{equation div estimate on na T}. The differential inequality for $\frac{\dd S}{\dd t}$ follows similarly, using the evolution equation in~\cite[Lemma 6.2]{LotayWeiLaplacianFlow} and the following estimate, analogous to ~\eqref{equation div estimate on na T},
\begin{align*}\big|\text{div}\big((g^{ab}\na_b - \tg^{ab}\tna_b)\widetilde{Rm}_{ijk}^l\big)\big|& \leq \frac{C}{r}(|\mu|+|\na \mu| +|A|+|B|)\\
&\leq \frac{C}{r}(|\ga| + |F|+|A|+ |B|).\end{align*}
Now, from~\cite[(4.12)]{LotayWeiLaplacianFlow},
\begin{align}\label{equation evolution of na Rm}
    \frac{\dd}{\dd t}\na \Rm &= \De \na \Rm + \Rm \ast \na \Rm + \na \Rm \ast T^2 + \Rm \ast T \ast \na T\\
    &\qquad +  \na^3 T \ast T + \na^2 T \ast \na T,
\end{align}
where the $\ast$ represents metric contractions by $g$. For any tensors $\cS$, $\tilde{\cS}$, we compute
\begin{align}\label{equation difference of De na Rm}
    \De \cS - \tilde{\De} \tilde{\cS} &= \De (\cS - \tilde{\cS}) + A \ast \tna \tT + B \ast \tT + A \ast \tT + A \ast \na \tT + \tGa \ast A\ast \tT  \nonumber\\
    & \qquad + (g^{-1} - \tg^{-1})\ast (\tna^2 \tilde{\cS} - \tGa \ast \tna \tilde{\cS}).
\end{align}
Note that we can replace contractions in $\tg$ by contractions in $g$ by estimating the difference in terms of $|\mu|$ and decaying curvature and torsion terms. If $\ast$ represents contractions in just $\tg$, then
\begin{equation}\label{equation which gets na W}
    \na^3 T \ast T - \tna^3 \tT \ast \tT = \na W \ast T + A \ast \tna^2 \tT \ast T + \tna^3 \tT \ast F
\end{equation}
Using ~\eqref{equation evolution of na Rm},~\eqref{equation difference of De na Rm}, and previously found estimates on decaying terms, we find the desired differential inequality for $\frac{\dd }{\dd t}Q$:
\begin{align*}
    \Big|\frac{\dd}{\dd t}Q - \De Q\Big| &\leq \frac{C}{r} |\na W| + \frac{C}{r}\big(|A|+ |B| + |\mu| +|V|+|W|+|F|+|S|+|Q|\big)\\
    &\leq \frac{C}{r} |\na W| + \frac{C}{r}\big(|A|+ |B| + |\ga| +|V|+|W|+|F|+|S|+|Q|\big).
\end{align*}
Now, we have from~\cite[(4.14)]{LotayWeiLaplacianFlow} that
\begin{align}\label{equation evolution of na^2 T}
    \frac{\dd}{\dd t}\na^2 T &= \De \na^2 T + \na^3 T \ast T \ast \vp + \na \Rm \ast (\na T + \na T \ast \psi + T^2 \ast \vp)+ \na^2 T \ast \na T \ast \vp \nonumber\\
    &\quad +\Rm \ast (\na^2 T + \na^2 T \ast \psi + \na T \ast T \ast \vp + T^3 \ast \vp) + \na^2 \Rm \ast (T+ T\ast \psi)\nonumber\\
    &\quad + \na T \ast T^3 \ast \vp + \na^2 T\ast (T^2 + T^2 \ast \psi)+ \na T \ast \na T \ast (T + T \ast \psi)
\end{align}
The differential inequality for $\frac{\dd}{\dd t}W$ follows from~\eqref{equation evolution of na^2 T} after applying~\eqref{equation difference of De na Rm} and expressing the terms involving $\na^3 T$ and $\na^2 \Rm$ using $\na W$ and $\na Q$, just as in~\eqref{equation which gets na W}.
\end{proof}

Let $\cS^{(k,l)}$ be the vector bundle of $(k,l)$-tensors over $\cC_{r_0}^{\Sg}$. Then, define
$$\cX := \cS^{(0,3)} \oplus \cS^{(1,3)}+\oplus \cS^{(1,4)}\oplus \cS^{(0,4)},$$
$$\cY := \cS^{(0,3)} \oplus \cS^{(0,2)}+\oplus \cS^{(1,2)}+\oplus \cS^{(1,3)}+\oplus \cS^{(0,2)}.$$
Now, for $t \in [-1, 0]$, we define the section $\bf{X}(t) \in C^{\infty}(\cX)$ by
$$\textbf{X} = V \oplus S \oplus Q\oplus W,$$
and we define the section $\bf{Y}(t) \in C^{\infty}(\cY)$ by
$$\textbf{Y} = \ga \oplus \mu \oplus A \oplus B \oplus F.$$
These sections form our ODE-PDE system, as described in the following proposition. Recall that $\tau = -t$.
\begin{prop}\label{proposition ODE PDE system}
There is $C>0$ and $r_0>0$ such that for all $(x,\tau) \in \cC_{r_0}^{\Sg} \times [0,1]$,
\begin{align*}
    \Big|\frac{\partial \bf{Y}}{\partial \tau}\Big| &\leq C\big(|\bf{X}| + |\bf{\na X}|\big) + \frac{C}{r} |\bf{Y}|\\
    \Big|\frac{\partial \bf{X}}{\partial \tau} + \De \bf{X}\Big| &\leq \frac{C}{r}\big(|\bf{X}| + |\nabla \bf{X}| + |\bf{Y}|\big)
\end{align*}
\end{prop}
\begin{proof}
Putting together the results of Lemma \ref{lemma evolution equations for ODE PDE system}, 
\begin{align*}
    \Big|\frac{\dd}{\dd \tau}\bf{Y}\Big| &\leq C \big(|V|+|S| +|\na V|+|\na S| + |\na Q|\big)\\
    &\qquad \qquad + \frac{C}{r} \Big(|A|+|B|+|\ga|+|V|+|W|+|F|+|S|\Big)\\
    &\leq C (|\bf{X}|+|\na {X}|) + \frac{C}{r} |\bf{Y}| + \frac{C}{r}(|V|+|S|+|W|)\\
\intertext{By adjusting the constants $C$ and $r_0 \geq 1$,}
    &\leq C (|\bf{X}|+|\na \bf{X}|) + \frac{C}{r} |\bf{Y}|.
\end{align*}
This gives the differential inequality for $\bf{Y}$. The differential inequality for $\frac{\dd\bf{X}}{\dd \tau}$ follows immediately from Lemma \ref{lemma evolution equations for ODE PDE system}.

\end{proof}

\section{Carleman Estimates}\label{section carleman estimates}

In this section, we derive Carleman estimates analogous to those in \cite{EscauriazaSereginSverak} which will eventually be applied in Section 6 to show that the tensors $\bf{X}$ and $\bf{Y}$ decay at least exponentially and ultimately vanish everywhere.

We will adopt the same notational conventions as in \cite{kotschwarwang2015} for Section \ref{section carleman estimates} and Section \ref{section proof of backwards uniqueness}---namely, we assume a backward Laplacian flow $\vp_\tau$ with associated metric $g_\tau$ and potential function $f_\tau$ on the space-time $\cC^\Sg_{r_0} \times [0,\tau_0]$, with $r_0 \ge 1$ and $\tau_0 \in (0,1)$, that satisfies all the properties stated in Proposition \ref{proposition backward flow reparametrization} with constant $K_0 \ge 1$. Henceforth, we take all inner products with respect to $g_\tau$ and let all connections be Levi-Civita with respect to $g_\tau$. The parameter $r$ will continue to denote both the radial coordinate on $\cC^\Sg_{r_0}$ as well as the radial distance with respect to the conical metric which we continue to denote $g_C$ (associated to $\vp_C$). As in \cite{kotschwarwang2015} we take $A_0$ to be the volume of the link $\Sg$ with respect to the restriction of $g_C$ to the cross-section at $\{r=1\}$. Finally, for the rest of section Section \ref{section carleman estimates}, we consider a generic time-dependent section $Z$ of the $(\ka, \nu)$-tensor bundle $\cZ = T^\kappa_\nu(\cC^\Sg_{r_0})$. As in \cite{kotschwarwang2015}, all constants without explicit dependencies will depend on the constants $A_0, K_0, \nu,$ and $\ka$.

\subsection{Carleman estimates to imply backward uniqueness}

We being by deriving the fundamental divergence identity behind the Carleman estimates that we use to prove backward uniqueness. This identity is analogous to that in \cite[Lemma 1]{EscauriazaSereginSverak} and very similar to that in \cite[Lemma 4.1]{kotschwarwang2015}. However, Lemma \ref{lemma fundamental divergence identity} differs from \cite[Lemma 4.1]{kotschwarwang2015} precisely where the evolution equations for $\dd_\tau g$, $\dd_\tau \vol_{g(\tau)}$, and $\dd_\tau \Ga^k_{ij}$ under the Laplacian flow are used in lieu of those for Ricci flow. Keeping this in mind, we carefully re-derive the Laplacian flow version of this identity.

Let $Z \in C^\infty(\cZ \times [0,\tau_0])$ and $\sF, \sH \in C^\infty(\cC^\Sg_{r_0} \times [0,\tau_0])$ with $\sH > 0$ and $\phi = \log \sH$. We define the following operators acting on sections in $C^\infty(\cZ \times [0,\tau_0])$.

\begin{equation}\label{equation antisymmetric operator}
    \cA := \frac{\dd }{\dd \tau} - \na_{\na \phi} + \frac{\sF}{2} \mathrm{Id}
\end{equation}
\begin{equation}\label{equation symmetric operator}
    \cS := \Delta + \na_{\na \phi} - \frac{\sF}{2} \mathrm{Id}
\end{equation}
The fundamental divergence identity \eqref{equation fundamental divergence identity} is built up from the identity
\begin{equation}\label{equation antisymmetric symmetric identity}
    \bigg\langle \cA Z, \bigg(\frac{\dd}{\dd \tau} + \Delta\bigg)Z \bigg \rangle - 2 |\cA Z|^2 = \langle \cA Z, \cS Z \rangle, 
\end{equation}
and the following Rellich-Ne\v{c}as identity applied to a vector field $\na \sH$,
\begin{align}\label{equation rellich necas identity}
\na_i \big[ |\na Z|^2 \na_i \sH - 2\langle \na_{\na \sH} Z, \na_i Z \rangle \big] &= 2\langle \na_i \na_j Z, \na_j Z \rangle \na_i \sH + |\na Z|^2\Delta \sH \\
& \qquad -2\langle \na_i (\na_{\na \sH} Z), \na_i Z \rangle  - 2\langle \na_{\na \sH} Z, \Delta Z\rangle  \nonumber\\
&= 2\langle [\na_i, \na_j]Z, \na_j Z \rangle \na_i \sH + |\na Z|^2\Delta \sH \nonumber \\
& \qquad - 2 \langle \na_i Z, \na_j Z \rangle \na_i\na_j \sH - 2\langle \na_{\na \sH} Z, \Delta Z\rangle. \nonumber 
\end{align}

\begin{lem}\label{lemma fundamental divergence identity}
For all $Z \in C^\infty(\cZ \times [0,\tau_0])$ and all smooth $\sF$ and $\sH > 0$, the following identity holds on $\cC^\Sg_{r_0} \times [0,\tau_0]$:
\begin{multline}\label{equation fundamental divergence identity}
\nabla_i \bigg\{ 2 \bigg\langle \frac{\partial Z}{\partial \tau}, \nabla_i Z \bigg\rangle \sH + |\nabla Z|^2\nabla_i \sH - 2 \langle \nabla_{\nabla \sH} Z, \nabla_i Z \rangle + \frac{\sF\sH}{2}\nabla_i |Z|^2 \\ + \frac{1}{2}(\sF\nabla_i \sH - \sH \nabla_i \sF) |Z|^2 \bigg\} d\mu
- \frac{\partial}{\partial \tau} \bigg\{ \bigg(|\nabla Z|^2 + \frac{\sF}{2}|Z|^2\bigg) \sH d\mu \bigg\} \\
=\bigg\{ 2 \bigg \langle \cA Z, \bigg(\frac{\partial}{\partial \tau} + \De\bigg)Z\bigg \rangle \sH - 2|\cA Z|^2 \sH \\ +  \bigg( \sF - \sH^{-1}\bigg(\frac{\partial \sH}{\partial \tau} - \Delta \sH + \frac{2}{3}R\sH\bigg)\bigg)\bigg( |\nabla Z|^2 + \frac{\sF}{2}|Z|^2 \bigg) \sH \\
-\frac{1}{2}\bigg(\frac{\partial \sF}{\partial \tau} + \Delta \sF\bigg)|Z|^2\sH - 2 \nabla_i \nabla_j \phi \langle \nabla_i Z, \nabla_j Z \rangle \sH - \frac{\partial g}{\partial \tau}(\nabla Z, \nabla Z)\sH\\
- \frac{\sF\sH}{2}\frac{\partial g}{\partial \tau}( Z,  Z)\sH + 2 E(Z, \nabla Z)\sH \bigg\} d\mu, 
\end{multline}
where 
\begin{multline}
\label{equation EZ}
E(Z, \nabla Z) := \bigg\langle \bigg[ \nabla_i, \frac{\partial}{\partial \tau}\bigg] Z, \nabla_iZ \bigg \rangle - \langle [\nabla_i, \nabla_j]Z, \nabla_i Z \rangle \nabla_j \phi \\
= \frac{1}{2}g^{qm}\big( \nabla_i \eta_{pm} + \nabla_p \eta_{im} - \nabla_m \eta_{ip} + R^j_{pmi} \nabla_j \phi \bigg) \langle \Theta_q^p Z, \nabla_i Z \rangle,
\end{multline}
and $\Theta_q^p$ is the operator
\begin{multline}\Theta_q^p(Z^\beta_\al) = \delta_{\alpha_1}^p Z^{\beta_1 \beta_2 \cdots \beta_\kappa}_{q\al_2 \cdots \al_\nu} + \delta_{\alpha_2}^p Z^{\beta_1 \beta_2 \cdots \beta_\kappa}_{\al_1 q \al_3 \cdots \al_\nu} + \cdots \delta_{\alpha_\nu}^p Z^{\beta_1 \beta_2 \cdots p}_{\al_1 \al_2 \cdots q}\\ - \delta_{q}^{\beta_1} Z^{p \beta_2 \cdots \beta_\kappa}_{\al_1 \al_2 \cdots \al_\nu}  - \delta_{q}^{\beta_2} Z^{\beta_1 p \cdots \beta_\kappa}_{\al_1 \al_2 \cdots \al_\nu} - \cdots - \delta_{q}^{\beta_\ka} Z^{\beta_1 \beta_2 \cdots p}_{\al_1 \al_2 \cdots \al_\nu}.
\end{multline}
\end{lem}

\begin{proof}
Combine the identities \eqref{equation antisymmetric symmetric identity} and \eqref{equation rellich necas identity} with the evolution equations 
\[
    \frac{\dd g_{ij}}{\dd \tau} = 2 \Ric_{ij} + \frac{2}{3}|T|^2 g_{ij} + 4T_{i}^{\;k}T_{kj},\;\;\;\; \frac{\dd}{\dd \tau}\vol_{g(\tau)} = \frac{2}{3}R,\;\;\; \text{ and } 
\]
\[   \frac{\dd}{\dd_\tau} \Ga^k_{ij} = \frac{1}{2}g^{kl}(\na_i \eta_{jl} + \na_j\eta_{il} - \eta_l \eta_{ij}), 
\]
where 
\[\eta = 2\Ric + \frac{2}{3}|T|^2g + 2T*T. 
\]
This yields the identity \eqref{equation fundamental divergence identity} after expanding the left-hand side. The expression \eqref{equation EZ} follows by substituting the evolution equation for the Christoffel symbols in for the commutator $[\na_i, \dd_\tau]$.
\end{proof}

Next, we show that an analogue of the inequality~\cite[(4.6)]{kotschwarwang2015} from Remark 4.2 holds in our setting.

\begin{lem} If $\nabla \phi = \Upsilon \nabla f$ for some function $\Upsilon$, the following identity holds:
\begin{equation}
    |E(Z, \nabla Z)| \le C(|\nabla \Rm| + |\Rm||T|) (|\nabla Z|^2 + (1+ \Upsilon^2)|Z|^2)
\end{equation}
\end{lem}
\begin{proof}
Differentiating \eqref{equation derivative identities on f 2} yields
\[ \nabla_i \Ric_{jk} + \nabla_i \nabla_j \nabla_k f = - \nabla_i \big(\tfrac{1}{3}|T|^2 g_{jk} + 2T_j^{\; l} T_{lk}\big).
\]
Exchanging the roles of $i$ and $j$ and subtracting the two equations gives
\[R_{ijk}^p \nabla_p f = \nabla_i \Ric_{jk} - \nabla_j \Ric_{ik} - \nabla_i \big(\tfrac{1}{3}|T|^2 g_{jk} + 2T_j^{\; l} T_{lk} \big) + \nabla_j \big(\tfrac{1}{3}|T|^2 g_{ik} + 2T_{i}^{\; l} T_{lk} \big).
\]
Substituting this into~\eqref{equation EZ}, we obtain
\begin{multline*}
\Bigg(\nabla_i \eta_{pq} + (1 + \Upsilon)(\nabla_p \eta_{iq} - \nabla_i \eta_{pq}) - \Upsilon(\nabla_p \big(\tfrac{1}{3}|T|^2 g_{iq} + 2T_{i}^{\; l} T_{lq}\big) - \nabla_i \big(\tfrac{1}{3}|T|^2 g_{pq} + 2T_p^l T_{lq}\big)) \\ - \Upsilon \left( \nabla_i \left(\tfrac{1}{3}|T|^2 g_{jk} + 2T_j^{\; l} T_{lk} \right) - \nabla_j \left(\tfrac{1}{3}|T|^2 g_{ik} + 2T_{i}^{\; l} T_{lk} \right) \right)\Bigg)\langle \Theta_q^p Z, \nabla_i Z \rangle.
\end{multline*}
Note that 
\[\nabla_i (|T|^2g_{jk}) = g_{jk}\nabla_i |T|^2 = - g_{jk}\nabla_i R
\]
and 
\[ \nabla_i (T_j^{\; l} T_{lk}) \le C|\nabla T||T| \le C|\Rm||T|.
\]
Thus, the additional corrections terms compared to the Ricci shrinker case can be estimated by
\[\bigg| \nabla_i \bigg(\frac{1}{3}|T|^2 g_{jk} + 2T_j^{\; l} T_{lk} \bigg) - \nabla_j \bigg(\frac{1}{3}|T|^2 g_{ik} + 2T_{i}^{\; l} T_{lk} \bigg) \bigg| \le C (|\nabla \Rm| + |\Rm||T|). 
\]
After estimating the torsion terms with this bound, obtain
\[|E(Z, \nabla Z)| \le C (|\nabla \Rm| + |\Rm||T|)(|\nabla Z|^2 + (1 + \Upsilon^2)|Z|^2)
\]
using Cauchy-Schwarz and Young's inequality.
\end{proof}

\begin{rmk}
When this lemma is applied, the salient bit is that the decay $h^2(|\Rm| + |\nabla \Rm|) \le CK_0$. So it is not so important that the bound is specifically $|\nabla \Ric|$ as in \cite[Remark 4.2]{kotschwarwang2015}. 
\end{rmk}

\subsubsection{Weighted $L^2$-inequality for the operator $\dd_\tau + \De$}

When we choose 
\begin{equation}
\label{equation FH}
    \sF := \sH^{-1} \left(\dd_\tau \sH - \De \sH + \frac{2}{3}R \sH\right),
\end{equation}
we obtain a cancellation of terms on the right-hand side of~\eqref{equation fundamental divergence identity}. 
As in \cite{kotschwarwang2015}, we integrate \eqref{equation fundamental divergence identity} over $\cC^\Sg_{r_0} \times [0,\tau_0]$ to obtain the following inequality, analogous to \cite[Lemma 4.3]{kotschwarwang2015}.

\begin{lem}
There exists a constant $N = N(\ka, \nu, K_0)$ such that if $Z \in C^{\infty}(\cZ \times [0, \tau_0])$ is compactly supported in $\cC^\Sg_{r_0}$ for each $\tau$ and satisfies $Z(\cdot, 0) = 0$, then, for any smooth $\sH >0$, we have
\begin{multline}\label{equation first L2 quadratic lower bound}
    \frac{1}{2}\int_0^{\tau_0} \int_{\cC^\Sg_{r_0} \times \{\tau\}} \bigg| \frac{\dd Z}{\dd \tau} + \De Z\bigg|^2 \sH d\mu d\tau + \int_{\cC^\Sg_{r_0}\times \{\tau_0\}} \bigg(|\nabla Z|^2 + \frac{\sF}{2}|Z|^2\bigg) \sH d\mu \\
    \ge\int_0^{\tau_0} \int_{\cC^\Sg_{r_0} \times \{\tau\}} (Q_1(\nabla Z, \nabla Z) + Q_2(Z,Z) -2E(Z, \nabla Z)) \sH d\mu d\tau,
\end{multline}
where $\sF$ is defined by~\eqref{equation FH} and 
\begin{align*}
    Q_1(\nabla Z, \nabla Z) &= 2(\nabla_i \nabla_j \phi) \langle \nabla_i Z, \nabla_j Z \rangle - \frac{N}{r^2}|\nabla Z|^2 \\
    Q_2(Z, Z) &= \frac{1}{2}\bigg(\frac{\dd \sF}{\dd \tau} + \De \sF \bigg) |Z|^2 - \frac{N|\sF|}{r^2}|Z|^2
\end{align*}
\end{lem}

\begin{proof}
Integration by parts and Young's inequality yield the result. Note that the decay terms of order $r^{-2}$ in $Q_1$ and $Q_2$ come from the fact that the right-hand side of the evolution equation 
\[ \frac{\dd g_{ij}}{\dd \tau} = 2 \Ric_{ij} + \frac{2}{3}|T|^2 g_{ij} + 4T_{i}^{\;k}T_{kj}
\]
decays quadratically.
\end{proof}

\subsubsection{A weighted $L^2$-inequality for the ODE component}

The $L^2$-inequality for the ODE component is the same as \cite[Lemma 4.4]{kotschwarwang2015}, with the same proof. We restate it here:

\begin{lem}\cite[Lemma 4.4]{kotschwarwang2015}
There exists a constant $N = N(\ka, \nu, K_0)$ such that if $Z \in C^\infty(\cZ \times [0,\tau_0])$ is compactly supported in $\cC^\Sg_{r_0}$ for each $\tau$ and satisfies $Z(\cdot, 0) =0$, then for all smooth $\sH >0$,
\begin{equation}
    -\int_0^{\tau_0} \int_{\cC^\Sg_{r_0} \times \{\tau\}} \bigg( N + \frac{\dd \phi}{\dd \tau} \bigg) |Z|^2 \sH d\mu d\tau \le \int_0^{\tau_0} \int_{\cC^\Sg_{r_0} \times \{\tau\}} \bigg| \frac{\dd Z}{\dd \tau}\bigg|^2 \sH d\mu d\tau.
\end{equation}
\end{lem}

\subsubsection{First Carleman estimate for the PDE system}

For this estimate, we use the same weight function $\sH_{1; \alpha, \tau_0}$ as the one defined in \cite[(4.19)]{kotschwarwang2015}. Fix $\de \in (0,1)$ and for any $\al >0$ define the function
\begin{equation}
    \sH_1 := \sH_{1; \alpha, \tau_0} = \exp\big[\al(\tau_0-\tau)h^{2-\de}(x,\tau) + h^2(x,\tau)\big].
\end{equation}
Let $\phi_1 := \phi_{1; \al, \tau_0} := \log \sH_1$. Define also
\begin{equation}
    \sF_1 := \sF_{1; \alpha, \tau_0} := \frac{\dd \phi_1}{\dd \tau} - \De \phi_1 - | \nabla \phi_1|^2 + \frac{2}{3}R
\end{equation}

We estimate the spatial and temporal derivatives of $\phi_1$ and obtain a pointwise estimate on $\sF_1$. This result is analogous to \cite[Lemma 4.7]{kotschwarwang2015}.

\begin{lem}\label{lemma estimates on phi 1}
For any $\al$ and any $\tau_0$, $\de \in (0,1)$, the function $\phi_1 = \phi_{1;\al, \tau_0}$ satisfies
\begin{align}
\frac{\dd \phi_1}{\dd \tau} &= 4\tau (R-\sE) - \al h^{2-\de}\bigg( 1 - \frac{2(2-\de)\tau(\tau_0 - \tau)(R - \sE)}{h^2} \bigg), \\
\nabla \phi_1 &= \big(\al(2-\de)(\tau_0 - \tau) h^{1-\de} + 2h \big) \nabla h\\
\nabla_i \na_j \phi_1 &= \frac{\al(2-\de)(\tau_0 - \tau)}{h^\de} (g_{ij} - 2\tau(\Ric_{ij} + \tfrac{1}{3}|T|^2 g_{ij} + 2 T_i^{\;k}T_{kj}) - \de \nabla_i h \nabla_j h) \\ 
& \qquad + 2(g_{ij} - 2\tau(\Ric_{ij} + \tfrac{1}{3}|T|^2 g_{ij} + 2 T_i^{\;k}T_{kj}))
\end{align}
These estimates imply that there exists a constant $r_1 \ge r_0$ depending only on $\de$ and $K_0$ such that on $\cC^\Sg_{r_1} \times [0,\tau_0]$,
\begin{align*}
    \nabla^2 \phi_1 &\ge g, \;\;\;\text{ and } \\
    0 \ge \sF_1 &\ge - N\big(1 + h^2 + \al h^{2-\de} + \al^2 (\tau_0 - \tau)^2 h^{2-2\de} \big)
\end{align*}
for all $\al \ge 1$.
\end{lem}

\begin{proof}
The first three inequalities follow from the identities proven in Lemma \ref{lemma estimates on time dep distance function}. The rest of the proof proceeds along the same lines as in \cite{kotschwarwang2015}, with some small but necessary modifications. To prove that 
\[\nabla^2 \phi_1 \ge g,
\]
first note that the symmetric 2-tensors $\Ric$, $|T|^2 g$, and $T_i^{\;k}T_{kj}$ decay quadratically, so in particular, we can select $r_1$ large enough that 
\[|\Ric| + \tfrac{1}{3}|T|^2 + |T * T| \le \frac{(1- \de)}{4}
\]
on $\cC_{r_1}^\Sg \times [0,\tau_0]$.  Since $\tau_0 < 1$, this implies that
\[ g_{ij} - 2\tau\big(\Ric_{ij} + \tfrac{1}{3}|T|^2 g_{ij} + 2 T_i^{\;k}T_{kj}\big) \ge \frac{(1+ \de)}{2}g_{ij},
\]
in the sense of quadratic forms. It remains to prove that the first summand in the expression for $\nabla_i \nabla_j \phi_1$ is positive-definite. Restricting to the orthogonal complement of $\nabla h$,
\begin{multline}\label{equation possibly non pos def term} g_{ij} - 2\tau\big(\Ric_{ij} + \tfrac{1}{3}|T|^2 g_{ij} + 2 T_{i}^{\;k}T_{kj}\big) - \de \nabla_i h \nabla_j h \\
= g_{ij} - 2\tau\big(\Ric + \tfrac{1}{3}|T|^2 g_{ij} + 2 T_{i}^{\;k}T_{kj}\big) \ge \frac{(1+ \de)}{2}g_{ij},
\end{multline}
so it remains to check positive-definiteness on the space spanned by $\nabla h$. Evaluating on this subspace, we obtain
\begin{equation*}
     \big(g_{ij} - 2\tau(\Ric_{ij} + \tfrac{1}{3}|T|^2 g_{ij} + 2 T_{i}^{\;k}T_{kj}) - \de \nabla_i h \nabla_j h \big)\nabla_i h \nabla_j h \ge \frac{(1+ \de)}{2}|\nabla h|^2 - \de |\nabla h|^4. 
\end{equation*}
Substituting in the expression for $|\nabla h|^2$ given in Lemma \ref{lemma estimates on time dep distance function}, we obtain
\begin{align*}
    \frac{(1+ \de)}{2}|\nabla h|^2 - \de |\nabla h|^4 &= |\nabla h|^2\bigg(\frac{(1+ \de)}{2} - \de \bigg(1 + \frac{4\tau^2}{h^2}(\sE -R) \bigg) \bigg) \\
    &= |\nabla h|^2\bigg(\frac{(1 - \de)}{2} - \frac{4\de\tau^2}{h^2}(\sE - R) \bigg).
\end{align*}
Increasing $r_1$ if necessary, we see that \eqref{equation possibly non pos def term} is positive-definite and thus we can conclude that $\nabla \nabla \phi_1 \ge g$. 

To begin estimating $\sF_1$, we first calculate
\[ |\nabla \phi_1|^2 = \big(\al(2-\de)(\tau_0 - \tau)h^{1-\de} + 2h\big)^2|\nabla h|^2,
\]
\[ \De \phi_1 = 2(7 - \frac{4}{3}\tau R) + \frac{\al(2-\de)(\tau_0 - \tau)}{h^\de}\bigg(7 - \frac{4}{3}\tau R - \de |\nabla h|^2 \bigg).
\]
For $r_1$ sufficiently large, the second summand in the expression for $\De \phi_1$ is positive, and the first summand can be bounded below by $7$---so $\De \phi_1 \ge 7$. Furthermore,
\[\frac{\dd \phi_1}{\dd \tau} = -\al h^{2-\de} \bigg( 1 - \frac{2(2-\de)\tau(\tau_0 - \tau) (R-\sE)}{h^2} - \frac{4\tau (R - \sE)}{\al h^{2-\de}} \bigg) \le - \frac{\al}{2}h^{2-\de},
\]
for sufficiently large $r_1$. Combining these estimates, we obtain the upper bound on $\sF_1$ on $\cC_{r_1}^\Sg \times [0, \tau_0]$.
\begin{align*}
    \sF_1 &= \frac{\dd \phi_1}{\dd \tau} - \De \phi_1 - | \nabla \phi_1|^2 + \frac{2}{3}R \\
    &\le  - \frac{\al}{2}h^{2-\de} - 7 + CK_0r^{-2}\\
    &\le - \frac{1}{2}\bigg(\frac{\al}{2}h^{2-\de} + 7\bigg)\\
    &\le 0.
\end{align*}
This establishes the upper bound. To compute the lower bound, we calculate an explicit expression for $\sF_1$.
\begin{align*}
    \sF_1 &= 4\tau (R-\sE) - \al h^{2-\de}\bigg( 1 - \frac{2(2-\de)\tau(\tau_0 - \tau)(R - \sE)}{h^2} \bigg) \\
     &\qquad - 2(7 - \frac{4}{3}\tau R) + \frac{\al(2-\de)(\tau_0 - \tau)}{h^\de}\bigg(7 - \frac{4}{3}\tau R - \de |\nabla h|^2 \bigg) \\
     &\qquad - \big(\al(2-\de)(\tau_0 - \tau)h^{1-\de} + 2h\big)^2|\nabla h|^2 + \frac{2}{3}R
\end{align*}
We can rearrange the terms to group like powers of $\al$. 
\begin{align}\label{equation F1 in powers of alpha}
    \sF_1 &= \bigg(- 2(7 - \frac{4}{3}\tau R) + \frac{2}{3}R + 4\tau (R-\sE)  +4h^2|\nabla h|^2 \bigg) \\ 
    &\qquad - \al h^{2-\de}\bigg( 1 - \frac{2(2-\de)\tau(\tau_0 - \tau)(R - \sE)}{h^2} \nonumber \\
    &\qquad \qquad + \frac{(2-\de)(\tau_0 - \tau)}{h^2}\bigg(7 - \frac{4}{3}\tau R - \de |\nabla h|^2 \bigg) - 4(2-\de)(\tau_0 -\tau)|\nabla h|^2 \bigg) \nonumber \\
     &\qquad - \al^2h^{2-2\de}\big((2-\de)^2(\tau_0 - \tau)^2|\nabla h|^2\big). \nonumber
\end{align}
Considering the asymptotics of each term, one can see that for sufficiently large $r_1$, on $\cC_{r_1}^\Sg \times [0,\tau_0]$
\[|\sF_1| \le N(1 + h^2 + \al h^{2-\de} + \al(\tau_0 - \tau)^2h^{2-2\de}).
\]
This concludes the proof of the lemma.
\end{proof}
Next, the expression $(\dd_\tau + \De)\sF_1$ must be bounded from below. In our setting, we can still obtain the lower bound given in \cite[Lemma 4.8]{kotschwarwang2015}. We give their statement accompanied by the appropriately modified proof below.
\begin{lem}\label{lemma lower bound on derivative of F1}
For all $\de \in (0,1)$, there exists $r_2 \ge r_0$ depending only on $\de$ and $K_0$, such that the function $\sF_1 = \sF_{1; \al, \tau_0}$ satisfies
\begin{equation}
    \frac{\dd \sF_1}{\dd \tau} + \De \sF_1 \ge 3 \al h^{2-\de} + \al^2(\tau_0 - \tau)h^{2-2\de}
\end{equation}
on $\cC^\Sg_{r_2} \times [0,\tau_0]$ for all $\al \ge 1$ and $\tau_0 \in (0,1]$.
\end{lem}
\begin{proof}First, recall that $\sE_\tau = O(\log r)$, $\nabla \sE =O(r^{-1})$, and $\De \sE = O(r^{-2})$ by Lemma \ref{lemma mixed time and space derivative estimates on shrinker potential}. Calculate the value of the operator $(\dd_\tau + \De)$ on each individual power of $\al$ occurring in \eqref{equation F1 in powers of alpha}. 
\begin{multline*}
    \frac{\dd}{\dd \tau} \bigg(- 2(7 - \frac{4}{3}\tau R) + \frac{2}{3}R + 4\tau (R-\sE)  +4h^2\bigg(1+ \frac{4\tau^2}{h^2}(\sE-R)\bigg) \bigg) \\
    \ge -8h \frac{\dd h}{\dd\tau} -  C_1\sE - C_2 \frac{\dd \sE}{\dd \tau} - C_3 \frac{\dd  R}{\dd \tau} - C_4
\end{multline*}    
Plugging in $\dd_\tau h = 2\tau h^{-1}(R -\sE)$, we obtain that this whole expression is bounded below by $-C_5 \log(r)$. Take the Laplacian.
\begin{multline*}
    \De  \bigg(- 2(7 - \frac{4}{3}\tau R) + \frac{2}{3}R + 4\tau (R-\sE)  +4h^2\bigg(1+ \frac{4\tau^2}{h^2}(\sE-R)\bigg) \bigg) \\
    \ge -C(h\De h + |\nabla h|^2) \ge - C
\end{multline*}
by Lemma \ref{lemma estimates on time dep distance function} and Lemma \ref{lemma mixed time and space derivative estimates on shrinker potential}. Thus, 
\[ \bigg(\frac{\dd}{\dd \tau} + \De \bigg)\bigg(- 2(7 - \frac{4}{3}\tau R) + \frac{2}{3}R + 4\tau (R-\sE)  +4h^2\bigg(1+ \frac{4\tau^2}{h^2}(\sE-R)\bigg)\bigg) \ge -C \log(r).
\]
This can be absorbed into the terms of higher order in $\al$, which we now calculate. Taking the time derivative and the Laplacian are lengthy but straightforward calculations. Plugging in the decay estimates for the curvature and the shrinker potential $\sE$ given in Lemma \ref{lemma mixed time and space derivative estimates on shrinker potential} into the resulting expressions yields the following asymptotics for the terms with coefficient $\al$.
\begin{multline*}
    - \al \frac{\dd}{\dd \tau}\bigg( h^{2-\de}\bigg( 1 - \frac{2(2-\de)\tau(\tau_0 - \tau)(R - \sE)}{h^2}  \\
    + \frac{(2-\de)(\tau_0 - \tau)}{h^2}\bigg(7 - \frac{4}{3}\tau R - \de |\nabla h|^2 \bigg) - 4(2-\de)(\tau_0 -\tau)|\nabla h|^2 \bigg)\bigg)
\end{multline*}
\[= \al\big((4(2-\de))h^{2-\de} + o(h^{-\de/2}) \big)
\]
and
\begin{multline*}
    - \al \De\bigg( h^{2-\de}\bigg( 1 - \frac{2(2-\de)\tau(\tau_0 - \tau)(R - \sE)}{h^2}  \\
    + \frac{(2-\de)(\tau_0 - \tau)}{h^2}\bigg(7 - \frac{4}{3}\tau R - \de |\nabla h|^2 \bigg) - 4(2-\de)(\tau_0 -\tau)|\nabla h|^2 \bigg)\bigg)
\end{multline*}
\[= \al\big(o(h^{-\de/2}) \big)
\]
Turning to the terms with coefficient $\al^2$, we observe that after differentiating by $\dd_\tau + \De$, the dominant term is that of order $h^{2-2\de}$.
\begin{align*}
    - \al^2\frac{\dd}{\dd \tau}\bigg(h^{2-2\de}\big((2-\de)^2(\tau_0 - \tau)^2|\nabla h|^2\big)\bigg) &= \al^2\big(2(\tau_0 -\tau)(2-\de)^2h^{2-2\de} + o(h^{-\de})\big)
\end{align*}
and
\begin{align*}
    - \al^2\De \bigg(h^{2-2\de}\big((2-\de)^2(\tau_0 - \tau)^2|\nabla h|^2\big)\bigg) &= \al^2\big(o(h^{-\de})\big)
\end{align*}
By increasing $r_2$ as needed, we take the sum of the order 0,1, and 2 terms in $\al$ and absorb the lower order terms into the dominant terms. Since the terms with coefficient $\al, \al^2 \ge 1$ dominate $\log(r)$, for sufficiently large $r_2$, the terms of order 0 in $\al$ can be absorbed into those of order $1$ and $2$ in $\al$. Then, possibly further increasing $r_2$ to absorb other small terms, we obtain the desired inequality,
\begin{equation*}
    \frac{\dd \sF_1}{\dd \tau} + \De \sF_1 \ge 3\al(2-\de)h^{2-\de} + \al^2(2-\de)^2(\tau_0 - \tau) h^{2-2\de}
\end{equation*}
on $\cC_{r_2}^\Sg \times [0,\tau_0]$.
\end{proof}

Now we can prove the first Carleman estimate, which is identical to that in \cite{kotschwarwang2015}.

\begin{prop}\label{proposition first carleman estimate}\cite[Proposition 4.9]{kotschwarwang2015}
For all $\de \in (0,1)$, there exists a constant $r_3 = r_3(\de, K_0) \ge r_0$ such that for all $\al \ge 1$ and all $Z \in C^\infty (\cZ \times [0, \tau_0])$ satisfying that $Z(\cdot, 0) = 0$ and that $Z(\cdot, \tau)$ is compactly supported in $\cC_{r_3}^\Sg$ for each $\tau \in [0,\tau_0]$, we have the estimate
\begin{multline}\label{equation first carleman estimate PDE}
\al ||Z \sH_1^{1/2}||^2_{L^2(\cC^\Sg_{r_3}\times [0,\tau_0])} + ||\nabla Z \sH_1^{1/2}||^2_{L^2(\cC^\Sg_{r_3}\times [0,\tau_0])} \\
\le \frac{1}{2}||(\dd_\tau + \De) Z \sH_1^{1/2}||^2_{L^2(\cC^\Sg_{r_3}\times [0,\tau_0])} + ||\nabla Z \sH_1^{1/2}||^2_{L^2(\cC^\Sg_{r_3}\times \{\tau_0\})}
\end{multline}
and
\begin{equation}\label{equation first carleman estimate ODE}
    \al ||Z \sH_1^{1/2}||^2_{L^2(\cC^\Sg_{r_3}\times [0,\tau_0])} \le 2||\dd_\tau Z \sH_1^{1/2}||^2_{L^2(\cC^\Sg_{r_3}\times [0,\tau_0])},
\end{equation}
where $\sH_1 = \sH_{1; \al, \tau_0}$.
\end{prop}

\begin{proof}
After replacing \cite[Lemma 4.7]{kotschwarwang2015} and \cite[Lemma 4.8]{kotschwarwang2015} by Lemma \ref{lemma estimates on phi 1} and Lemma \ref{lemma lower bound on derivative of F1} respectively, the proof of \cite[Proposition 4.9]{kotschwarwang2015} goes through and is sufficient to prove Proposition \ref{proposition first carleman estimate}.
\end{proof}

\subsection{Carleman estimates to imply rapid decay}\label{subsection Carleman estimates to imply rapid decay}

Since the weight function $\sH_1$ in the first Carleman estimate grows at the rate $\exp(Ch^2)$, we must show that our solution satisfying the ODE-PDE system actually has decay of the same order or faster. This is done by means of a second Carleman estimate, similar to that in \cite[(1.4)]{EscauriazaSereginSverak} and \cite[(5.23)]{kotschwarwang2015}, which is given in Proposition \ref{proposition second Carleman estimate PDE}. This estimate differs significantly from the two cited estimates in that it contains a $\log$ growth term in $|Z|^2$ on the right-hand side of the inequality. However, it will be shown in Section \ref{section proof of backwards uniqueness} how the exponential decay of the solution $Z$ can be salvaged from the inequality as stated in \eqref{equation second Carleman estimate}. 

As in \cite{kotschwarwang2015}, we begin with a perturbed version of the divergence identity given in Lemma \ref{lemma fundamental divergence identity}, which is analogous to \cite[Lemma 5.1]{kotschwarwang2015}.

\begin{lem}\label{lemma perturbed fundamental divergence identity}
For all $Z \in C^\infty(\cZ \times [0,\tau_0])$, $\sF,\sH \in C^\infty(\cC^\Sg_{r_0} \times [0,\tau_0])$ with $\sH > 0$, and positive functions $\sigma, \theta \in C^\infty([0,\tau_0])$ with $\sg$ increasing, the following identity holds:
\begin{multline}
\te \sg^{-\al}\nabla_i \bigg\{ 2 \bigg\langle \frac{\partial Z}{\partial \tau}, \nabla_i Z \bigg\rangle \sH + |\nabla Z|^2\nabla_i \sH - 2 \langle \nabla_{\nabla \sH} Z, \nabla_i Z \rangle \\ 
+ \frac{\sF\sH}{2}\nabla_i |Z|^2 + \frac{1}{2}(\sF\nabla_i \sH - \sH \nabla_i \sF) |Z|^2 \bigg\} d\mu \\
- \frac{\partial}{\partial \tau} \bigg\{ \bigg(|\nabla Z|^2 + \frac{\sF}{2}|Z|^2\bigg) \te \sg^{-\al}\sH d\mu \bigg\} \\
=\bigg\{ 2 \bigg \langle \cA Z, \bigg(\frac{\partial}{\partial \tau} + \De\bigg)Z\bigg \rangle  - 2|\cA Z|^2  -\frac{1}{2}\bigg(\frac{\partial \sF}{\partial \tau} + \Delta \sF\bigg)|Z|^2 \\
\hspace{2cm} + \bigg( \sF - \sH^{-1}\bigg(\frac{\partial \sH}{\partial \tau} - \Delta \sH\bigg) + \frac{2}{3}R - \al \frac{\dot \sg}{\sg} \bigg)\bigg( |\nabla Z|^2 + \frac{\sF}{2}|Z|^2 \bigg)  \\
\hspace{2cm} - \frac{\dot \te}{\te} \bigg( |\nabla Z|^2 + \frac{\sF}{2}|Z|^2 \bigg) - 2 \nabla_i \nabla_j \phi \langle \nabla_i Z, \nabla_j Z \rangle - \frac{\partial g}{\partial \tau}(\nabla Z, \nabla Z) \\ - \frac{\sF\sH}{2}\frac{\partial g}{\partial \tau}( Z,  Z) + 2 E(Z, \nabla Z) \bigg\} \te \sg^{-\al} \sH d\mu, 
\end{multline}
where $\cA$ and $E(Z, \nabla Z)$ are defined as before. 
\end{lem}

We choose the function $\sF$ to give a cancellation in the right-hand side of the identity. 
\[\sF := \frac{1}{\sH} \bigg( \frac{\partial \sH}{\partial \tau} - \Delta \sH \bigg) + \frac{2}{3}R - \al \frac{\dot \sg}{\sg},
\]
and let
\[\widetilde{\sF} := \sF + \al \frac{\dot \sg}{\sg}.
\]
We also choose
\[\theta := \frac{\sg}{\dot \sg}
\]
for the remainder of the section. The results of \cite[\S 5.2]{kotschwarwang2015} and their proofs go through unchanged into our setting. The first result that we restate is the following integral inequality, which will yield the $L^2$ Carleman estimate.

\begin{lem}\cite[Lemma 5.2]{kotschwarwang2015}\label{lemma integral inequality for backward heat operator} There exists a constant $N = N(\ka, \nu, K_0)$ such that for any $\al$ and any $Z \in C^\infty(\cZ \times [0, \tau_0])$ that is compactly supported in $\cC^\Sg_{r_0} \times [0, \tau_0)$ and vanishes on $\cC^\Sg_{r_0} \times \{0\}$, the inequality 
\begin{multline}\label{equation integral inequality for backward heat operator}
\int_0^{\tau_0}\int_{\cC^\Sg_{r_0} \times \{\tau\}} \frac{\sg^{1-\al}}{\dot \sg} \big(Q_3(\nabla Z, \nabla Z) + Q_4(Z,Z) - 2E(Z, \nabla Z)\big)\sH d\mu d\tau \\
\le \int_0^{\tau_0}\int_{\cC^\Sg_{r_0} \times \{\tau\}} \frac{\sg^{1-\al}}{\dot \sg} \bigg| \frac{\partial Z}{\partial \tau} + \Delta Z \bigg|^2 \sH d\mu d\tau
\end{multline}
holds, where
\begin{equation}
Q_3(\nabla Z, \nabla Z) \ge \bigg( 2 \nabla_i \nabla_j \phi - \frac{\sg}{\dot \sg} \ddot{\widehat{\log \sg}}g_{ij} \bigg) \langle \nabla_i Z, \nabla_j Z \rangle - \frac{N}{r^2} |\nabla Z|^2
\end{equation}
and
\begin{equation}
Q_4( Z,  Z) \ge \frac{1}{2}\bigg( \frac{\partial \wsF}{\partial \tau} + \Delta \wsF + \frac{\dot \te}{\te}  \wsF \bigg)|Z|^2 - \frac{N}{r^2} |\sF||Z|^2.
\end{equation}
\end{lem}

The next inequality helps to estimate $|Z|$ by terms of the form $|\nabla Z|$ and $|(\dd_\tau + \De)Z|$.

\begin{lem}\label{lemma control Z by gradient}
There exists a constant $N = N(\ka, \nu, K_0)$ such that, for all $\al >0$, all smooth positive $\sH = \sH(x,\tau)$, and all positive increasing $\sg = \sg(\tau)$, the inequality 
\begin{multline}
\int_0^{\tau_0}\int_{\cC^\Sg_{r_0} \times \{\tau\}} \sg^{-2\al}\bigg(\al \frac{\dot \sg}{\sg} - \tilde{\sF} - \frac{N}{r^2}\bigg)|Z|^2 \sH d\mu d\tau \\
\le \int_0^{\tau_0}\int_{\cC^\Sg_{r_0} \times \{\tau\}} \sg^{-2\al}\bigg( 2 |\nabla Z|^2 + \frac{\sg}{\al \dot \sg} \bigg| \frac{\partial Z}{\partial \tau} + \Delta Z \bigg|^2 \bigg) \sH d\mu d\tau
\end{multline}
holds for every $Z \in C^\infty(\cZ \times [0, \tau_0])$ that is compactly supported in $\cC^\Sg_{r_0} \times [0, \tau_0)$ and vanishes on $\cC^\Sg_{r_0} \times \{0\}$.
\end{lem}

\subsection{An approximate solution to the conjugate heat equation}

We use precisely same function defined in \cite{kotschwarwang2015} as our integration kernel. For any $a \in (0,1)$ and $\rho \in (r_0, \infty)$, define
\begin{equation}
    \sH_2(x,\tau) := \sH_{2; a, \rho}(x,\tau) = (\tau + a)^{-7/2} \exp \bigg(-\frac{(h(x,\tau) - \rho)^2}{4(\tau + a)} \bigg)
\end{equation}
on $\cC^\Sg_{r_0} \times [0,\tau_0]$. We note that analogously to the Ricci soliton case, $\sH_{2;0,0} = \tau^{-7/2}e^{-f}$ satisfies
\[\partial_\tau \sH_{2; 0,0} - \Delta \sH_{2; 0,0} + \frac{2}{3}R\sH_{2; 0,0} = 0.
\]

\subsubsection{Estimates on the derivatives of $\sH_2$}
We compute the image of $\sH_2$ under the conjugate heat operator $\dd_\tau - \De$. 

\begin{lem}\label{lemma conjugate heat operator acting on H2}
Let $\sH_2$ be defined as above. The following identity holds
\begin{align}\label{equation conjugate heat operator acting on H2}
    \sH_2^{-1} \bigg( \frac{\partial \sH_2}{\partial \tau}  -\Delta  \sH_2 \bigg) &= \bigg( \frac{\tau^2 (h- \rho)^2}{h^2(\tau + a)^2} + \frac{2 \tau^2 \rho}{h^3 (\tau + a)} + \frac{\tau(h-\rho)}{h(\tau + a)} \bigg)(R - \sE) \nonumber\\
    &\qquad - \frac{3\rho  }{ h (\tau + a)}   + \frac{2\tau R}{3(\tau + a)}\bigg( \frac{\rho}{h} - 1 \bigg).
\end{align}
\end{lem}
\begin{proof}
We compute each term separately. First, take the time derivative:
\begin{equation}\label{equation time derivative of H2}
    \sH_2^{-1} \frac{\partial \sH_2}{\partial \tau} = \frac{(h- \rho)^2}{4(\tau + a)^2} - \frac{7}{2(\tau + a)} - \frac{\tau(h-\rho)}{h(\tau + a)} (R-\sE).
\end{equation}
The last summand is obtained by substituting in the expression for $\tfrac{\dd h}{\dd \tau}$ given in Lemma \ref{lemma estimates on time dep distance function}.
The spatial second derivative is computed just as in \cite[(5.10)]{kotschwarwang2015}, by first observing that $\sH_2^{-1}\na \sH_2 = -(h-\rho)/(2(\tau + a))\na h$ and plugging in expressions for the derivatives of $h$ from Lemma \ref{lemma estimates on time dep distance function}.
\begin{multline*}
     \sH_2^{-1} \nabla_i \nabla_j  \sH_2 =  \frac{1}{2(\tau + a)}\bigg(-g_{ij} + 2\tau\big(\Ric_{ij} + \tfrac{1}{3}|T|^2g_{ij} + 2T_i^{\;k}T_{kj} \big)\bigg) \\
     + \frac{\rho}{2(\tau + a)} \nabla_i \nabla_j h + \frac{(h- \rho)^2}{4(\tau + a)^2} \nabla_i h \nabla_j h.
\end{multline*}
Tracing, we find that
\begin{multline}\label{equation Laplacian of H2}
     \sH_2^{-1} \Delta  \sH_2 = - \frac{7}{2(\tau + a)} + \frac{2\tau R}{3(\tau + a)} + \frac{\rho}{2(\tau + a)} \Delta h + \frac{(h- \rho)^2}{4(\tau + a)^2} |\nabla h|^2.
\end{multline}
Subtract \eqref{equation Laplacian of H2} from \eqref{equation time derivative of H2}.
\begin{align*}
 \sH_2^{-1} \bigg( \frac{\partial \sH_2}{\partial \tau}  -\Delta  \sH_2 \bigg) &= \frac{(h- \rho)^2}{4(\tau + a)^2}(1 - |\nabla h|^2) - \frac{\rho \Delta h}{2(\tau + a)} \\ 
 &\qquad - \frac{2\tau R}{3(\tau + a)} - \frac{\tau(h-\rho)}{h(\tau + a)} (R-\sE)\\
 &= \frac{(h- \rho)^2}{4(\tau + a)^2}(\frac{4\tau^2}{h^2}(R-\sE)) - \frac{\rho }{2 h (\tau + a)}\bigg(7 - |\nabla h|^2 - \frac{4\tau R}{3}\bigg) \\
 &\qquad - \frac{2\tau R}{3(\tau + a)} - \frac{\tau(h-\rho)}{h(\tau + a)} (R-\sE)\\
 &= \frac{\tau^2 (h- \rho)^2}{h^2(\tau + a)^2}(R-\sE) - \frac{6\rho}{2 h (\tau + a)} + \frac{\rho}{2h (\tau + a)}(1-|\na h|^2)  \\
 &\qquad + \frac{2 \rho \tau R }{3 h (\tau + a)} - \frac{2\tau R}{3(\tau + a)} - \frac{\tau(h-\rho)}{h(\tau + a)} (R-\sE)\\
 &= \bigg( \frac{\tau^2 (h- \rho)^2}{h^2(\tau + a)^2} + \frac{2 \tau^2 \rho}{h^3 (\tau + a)} - \frac{\tau(h-\rho)}{h(\tau + a)} \bigg)(R- \sE) \\
 &\qquad - \frac{3\rho }{h (\tau + a)}   + \frac{2\tau R}{3(\tau + a)}\bigg( \frac{\rho}{h} - 1 \bigg) \\
\end{align*}
This concludes the proof of the lemma.
\end{proof}

\begin{lem}
For all $a \in (0,1)$, $\ga > 0$, and $\rho > r_0 \ge 1$, there exist constants $C = C(\ga) > 0$ and $r_4 = r_4(\ga, K_0) \ge r_0$, such that $\sH_2 = \sH_{2; a, \rho}$ and $\widetilde{\sF}_2 = \widetilde{\sF}_{2;a, \rho}$ satisfy
\begin{equation}\label{equation compare kernels with h and r}
    \frac{1}{2}e^{-\frac{(r(x) - \rho)^2}{4(\tau + a)}} \le  (\tau + a)^\frac{7}{2} \sH_2(x, \tau) \le 2 e^{-\frac{(r(x) - \rho)^2}{4(\tau + a)}},
\end{equation}
and
\begin{equation}\label{equation pointwise bounds on F2}
    -C \log(r) - \frac{3\rho}{h(\tau + a)} - \frac{1}{8h} \le \widetilde{\sF}_2 \le  C \log(r) - \frac{3\rho}{h(\tau + a)} + \frac{1}{8h},
\end{equation}
and $\phi_2 := \log \sH_2$ satisfies 
\begin{equation}\label{equation lower bound on log H2}
    \nabla \nabla \phi_2 \ge - \frac{g}{2(\tau + a)} - \frac{g}{48}
\end{equation}
on the domain $\cC^\Sg_{\max\{r_4,\ga \rho\}} \times [0, \tau_0]$.
\end{lem}

\begin{proof}
To prove the first inequality \eqref{equation compare kernels with h and r}, we appeal to the estimate \eqref{equation comparability h and r} and obtain 
\[\frac{(r-\rho)^2}{(\tau + a)} - \frac{CK_0\tau^2 \log(r)}{(\tau + a)r} \le \frac{(h - \rho)^2}{(\tau + a)}\le \frac{(r-\rho)^2}{(\tau + a)} + \frac{CK_0\tau^2 \log(r)}{(\tau + a)r}.
\]
When $\tau, a \in (0,1)$, $\tfrac{\tau^2}{\tau + a} \le 1$, so for sufficiently large $r_4 \ge r_0$ (depending only on $K_0$), $\tfrac{\log r}{r}$ is arbitrarily small and the error terms can be made arbitrarily close to $e^0 = 1$. This gives \eqref{equation compare kernels with h and r}.

The second inequality \eqref{equation pointwise bounds on F2} comes from examining the terms of \eqref{equation conjugate heat operator acting on H2}. First, note that $\big|\tfrac{\tau}{\tau + a}\big|\le 1$ and $\big|\tfrac{h-\rho}{h}\big|\le C(\ga)$, so the coefficient of $(R-\sE)$ in \eqref{equation conjugate heat operator acting on H2} is a constant $C(\ga)$. Thus, this term can be bounded above and below by $\pm C\log (r)$. Observe that
\[\bigg|  \frac{2\tau R}{3(\tau + a)}\bigg( \frac{\rho}{h} - 1 \bigg) \bigg| \le |R|\big|2 \ga^{-1} - 1\big| \le \frac{C K_0}{h^2},
\]
where $C$ depends on $\ga$. Since we can make $r_4 \ge r_0$ large enough that $C K_0 h^{-1} \le \tfrac{1}{8}$, we obtain the estimate \eqref{equation pointwise bounds on F2}.  

The third inequality \eqref{equation lower bound on log H2}, the lower bound on the Hessian of $\phi_2$ is obtained along the same lines as \cite[(5.16)]{kotschwarwang2015} with the requisite modifications made for the Laplacian flow case.
\begin{align*}
        \na_i \na_j \phi_2 &= -\frac{1}{2(\tau + a)}((h-\rho)\na_i \na_j h + \na_i h \na_j h)\\
        &= -\frac{1}{2(\tau + a)}\big(g_{ij} - 2\tau\big(\Ric_{ij} + \tfrac{1}{3}|T|^2g_{ij} - 2T_i^{\;k}T_{kj} \big)\big) \\
        &\qquad+\frac{\rho}{2h(\tau + a)}\big(g_{ij} - 2\tau\big(\Ric_{ij} + \tfrac{1}{3}|T|^2g_{ij} - 2T_i^{\;k}T_{kj} \big) -\na_i h\na_j h\big)\\
        &= -\frac{g_{ij}}{2(\tau+a)} + \frac{\tau(h- \rho)}{h(\tau + a)}\big(\Ric_{ij} + \tfrac{1}{3}|T|^2g_{ij} - 2T_i^{\;k}T_{kj} \big) \\
        &\qquad +\frac{\rho g_{ij}}{2h(\tau + a)} - \frac{\rho \na_i h \na_j h}{2h(\tau+a)}
\end{align*}
Since $|\Ric_{ij} + \tfrac{1}{3}|T|^2g_{ij} - 2T_i^{\;k}T_{kj}| \le CK_0 h^{-2}$ and $|\na h|^2 \le 1 + C \tau^2 \log(r)h^{-2}$, we can obtain the following estimates (in the sense of quadratic forms) with a constant $C=C(\ga)$.
\[\frac{\tau(h- \rho)}{h(\tau + a)}\big(\Ric_{ij} + \tfrac{1}{3}|T|^2g_{ij} - 2T_i^{\;k}T_{kj} \big) \ge - \frac{CK_0g_{ij}}{h^2}
\]
and
\[\frac{\rho \na_i h \na_j h}{2h(\tau+a)} \le \frac{\rho g}{2h(\tau+a)} + \frac{C \log(r)g}{h^2}.
\]
For sufficiently large $r_4 \ge r_0$, we can make $h^{-2}$ and $\log(r)h^{-2}$ arbitrarily small. Substituting these lower bounds into our equation gives the lower bound \eqref{equation lower bound on log H2}.
\end{proof}

\subsubsection{Estimates on the derivatives of $\widetilde{\sF}_2$}

To simplify our analysis, we group terms as follows. Rewrite
\begin{align*}\label{equation conjugate heat operator acting on H2}
   \wsF_2 &= \bigg( \frac{\tau^2 (h- \rho)^2}{h^2(\tau + a)^2} + \frac{2 \tau^2 \rho}{h^3 (\tau + a)} + \frac{\tau(h-\rho)}{h(\tau + a)} \bigg)(R - \sE) \\
    &\qquad - \frac{3\rho  }{ h (\tau + a)}   + \frac{2\tau R}{3(\tau + a)}\bigg( \frac{\rho}{h} - 1 \bigg)  + \frac{2}{3}R,
\end{align*}
as
\begin{equation}\label{equation F2 with compact notation} \wsF_2 = \sI(x,\tau) \sE - \frac{\rho (n-1) }{2 h (\tau + a)}   + R\sJ\bigg(\frac{\rho}{h}, \tau \bigg),
\end{equation}
where 
\[ \sI(x,\tau) = \frac{\tau^2 (h- \rho)^2}{h^2(\tau + a)^2} + \frac{2 \tau^2 \rho}{h^3 (\tau + a)} - \frac{\tau(h-\rho)}{h(\tau + a)}
\]
and 
\[\sJ(s,\tau) = \frac{2\tau }{3(\tau + a)}\cdot s + \frac{2a }{3(\tau + a)}. \]
Using this notation, we prove the following lemma.

\begin{lem}
For all $a \in (0,1)$, $\ga > 0$, and $\rho > r_0$, there exist constants $N$ and $r_5 \ge r_0$ both depending only on $\ga$ and $K_0$ such that $\wsF_2 = \wsF_{2; a, \rho}$ satisfies
\begin{equation}\label{equation lower bound for backward heat on F2}
    \frac{\dd \wsF_2}{\dd \tau} + \De \wsF_2 \ge -\frac{N}{(\tau + a)}\log(r) + \frac{3\rho}{h(\tau + a)^2}
\end{equation}
on the set $\cC^\Sg_{\max\{r_5, \ga\rho\}} \times [0, \tau_0]$.
\end{lem}

\begin{proof}
Note that throughout the proof we tacitly assume that $r_5 \ge r_0$ increases whenever necessary. We begin by applying the operator $\dd_\tau + \De$ to the expression \eqref{equation F2 with compact notation}, and then analyzing each term.

\begin{align*}\frac{\dd \wsF_2}{\dd \tau} + \De \wsF_2 &= (\sI_\tau + \De \sI) \sE + \sI \sE_\tau + 2 \langle \nabla \sI , \nabla \sE \rangle + \sI \De \sE\\
&\qquad + \frac{3\rho}{(\tau + a)^2 h}  + \frac{3\rho}{(\tau + a) h^2} \bigg(\frac{\dd h}{\dd \tau} + \De h - \frac{2|\nabla h|^2}{h} \bigg) \\ 
&\qquad + \sJ \bigg( \frac{\dd R}{\dd \tau} + \De R \bigg)  + \sJ_\tau R - \frac{\rho \sJ_s R}{h^2}\bigg(\frac{\dd h}{\dd \tau} + \De h\bigg)  \\ 
&\qquad - \frac{2\rho \sJ_s}{h^2} \langle \nabla R, \nabla h \rangle + \frac{\rho R}{h^3}\bigg( 2 \sJ_s + \frac{\rho \sJ_{ss}}{h} \bigg) |\nabla h|^2.
\end{align*}
We first estimate 
\begin{align*}
    \bigg|\frac{3\rho}{(\tau + a) h^2} \bigg(\frac{\dd h}{\dd \tau} + \De h - \frac{2|\nabla h|^2}{h} \bigg)\bigg| = \frac{C(\ga)}{(\tau + a)} \frac{\log (r)}{h} \le \frac{N}{(\tau + a)}.
\end{align*}
Next, we estimate the terms with $\sJ$ and its derivatives.
\begin{align*}
|\sJ(\rho h^{-1}, \tau)| &= \bigg| \frac{2\tau }{3(\tau + a)}\cdot s + \frac{2a }{3(\tau + a)}\bigg| \le C(\ga^{-1} + 1)\\
|\sJ_\tau(\rho h^{-1}, \tau)| &= \bigg|\frac{2\rho}{3h} \frac{\tau^2}{(\tau + a)^2} \frac{a}{\tau^2} - \frac{2}{3}\frac{a}{(\tau +a)^2}\bigg| \le \frac{C\ga^{-1}}{\tau + a} \\
|\sJ_s (\rho h^{-1}, \tau)| & = \bigg|\frac{2\tau}{3(\tau + a)}\bigg| \le \frac{2}{3} \\
|\sJ_{ss} (\rho h^{-1}, \tau)| & = 0
\end{align*}
Observing that 
\[\frac{a}{\tau^2} = \frac{a}{(\tau +a)^2} \cdot \frac{(\tau + a)^2}{\tau^2}
\]
gives the second inequality. The curvature decay estimates of Claim \ref{claim uniform curvature estimates} and the identities in Lemma \ref{lemma estimates on time dep distance function} give us that each term involving $\sJ$ or its derivatives can be subsumed into a term of the form $N/(\tau + a)$.

It remains to estimate the terms involving $\sI$ and its derivatives. It is immediate that 
\[|\sI(x,\tau)| = \bigg| \frac{\tau^2 (h- \rho)^2}{h^2(\tau + a)^2} + \frac{2 \tau^2 \rho}{h^3 (\tau + a)} - \frac{\tau(h-\rho)}{h(\tau + a)} \bigg| \le C,
\]
where $C$ depends on $\ga$. This and Lemma \ref{lemma mixed time and space derivative estimates on shrinker potential} yields the estimates
\[ |\sI \sE_\tau| \le Cr^{-1}, |\sI \De \sE| \le C h^{-2},
\]
which can be absorbed into the term $-N\log(r)/(\tau + a)$ in \eqref{equation lower bound for backward heat on F2}. Next, we differentiate with respect to $\tau$.
\begin{align*}
     \sI_\tau &= \frac{(1-\rho/h)}{(1 + a/\tau)^2}\frac{2\rho}{h^2} \frac{\dd h}{\dd \tau} + 2 \frac{(1 - \rho/h)^2}{(1 + a/\tau)^3}\frac{a}{\tau^2}\\ 
     &\qquad +\frac{4\tau}{(\tau + a)}\frac{\rho}{h^3} - \frac{6\tau^2 \rho}{h^4 (\tau + a)} \frac{\dd h}{\dd \tau} - \frac{2\tau^2 \rho}{h^2(\tau+a)^2} \\
     &\qquad - \frac{1}{(1 + a/\tau)}\frac{\rho}{h^2}\frac{\dd h}{\dd \tau} - \frac{(1 - \rho/h)}{(1 + a/\tau)^2}\frac{a}{\tau^2}.\\
     &= \frac{(1-\rho/h)}{(1 + a/\tau)^2}\frac{2\rho}{h^2} \frac{\dd h}{\dd \tau} + 2 \frac{(1 - \rho/h)^2}{(1 + a/\tau)}\frac{a}{(\tau+a)^2}\\ 
     &\qquad +\frac{4\tau}{(\tau + a)}\frac{\rho}{h^3} - \frac{6\tau^2 \rho}{h^4 (\tau + a)} \frac{\dd h}{\dd \tau} - \frac{2\tau^2 \rho}{h^2(\tau+a)^2} \\
     &\qquad - \frac{1}{(1 + a/\tau)}\frac{\rho}{h^2}\frac{\dd h}{\dd \tau} - \frac{a(1 - \rho/h)}{(\tau + a)^2}.
\end{align*}
We obtain the second equality by applying the identity
\[\frac{a}{\tau^2} = \frac{a}{(\tau +a)^2} \cdot \big(1 + a/\tau)^2.
\]
Thus, $|\sI_\tau| \le N/(\tau + a)$, where $N$ depends on $\ga$. Finally, we estimate the spatial derivatives of $\sI$. 
\begin{align*}
    \nabla \sI &= \frac{(1-\rho/h)}{(1 + a/\tau)^2}\frac{2\rho}{h^2} \nabla h - \frac{6 \tau^2 \rho}{h^4 (\tau + a)}\nabla h - \frac{1}{(1 + a/\tau)}\frac{\rho}{h^2}\nabla h \\
    \De \sI &= \frac{1}{(1 + a/\tau)^2}\frac{2\rho^2}{h^4} |\nabla h|^2 -\frac{(1-\rho/h)}{(1 + a/\tau)^2}\frac{4\rho}{h^3} |\nabla h|^2 \\
    &\qquad +\frac{(1-\rho/h)}{(1 + a/\tau)^2}\frac{2\rho}{h^2} \De h  + \frac{24 \tau^2 \rho}{h^5 (\tau + a)}|\nabla h|^2 - \frac{6 \tau^2 \rho}{h^4 (\tau + a)}\De h \\
    &\qquad + \frac{2}{(1 + a/\tau)}\frac{\rho}{h^3}|\nabla h|^2  - \frac{1}{(1 + a/\tau)}\frac{\rho}{h^2}\De h
\end{align*}
Applying Claim \ref{claim uniform curvature estimates} and the identities in Lemma \ref{lemma estimates on time dep distance function} tell us that $|\na \sI| + |\De \sI| \le N$, where $N$ depends on $\ga$. This allows us to make the final estimates
\[|(\sI_\tau + \De \sI)| \le \frac{N}{(\tau + a)}, |\langle \na \sE, \na \sI \rangle| \le |\na \sE||\na \sI| \le Nr^{-1}.
\]
Combining all these estimates gives the desired lower bound.
\end{proof}
 
\subsubsection{Carleman estimate for the PDE component}

We now set the function $\sg(\tau) := \tau e^{-\frac{\tau}{3}}$ and define its translates $\sg_a(\tau) := \sg(\tau + a)$. We repeat some useful estimates and an identity for $\sg_a$ from \cite[\S 5]{kotschwarwang2015}.
\begin{equation}\label{equation bounds on sigma and its derivatives}
    \frac{1}{3e}(\tau +a) \le \sg_a(\tau) \le (\tau + a) \qquad \text{and} \qquad \frac{1}{3e} \le \dot \sg_a(\tau) \le 1,
\end{equation}
for $\tau \in [0,1]$. The following identity will also be useful.
\begin{equation}
    -\frac{\sg_a}{\dot \sg_a} \ddot{\widehat{\log \sg_a}} = \frac{1}{(\tau + a)(1 - (1/3)(\tau + a))}
\end{equation}
Let $Q_3$ and $Q_4$ be as in Lemma \ref{lemma integral inequality for backward heat operator}.

\begin{lem}\label{lemma quadratic terms lower bound}
For all $\al > 0$, $a \in (0,1)$, and $\ga > 0$, there exist $N$ and $r_6 \ge r_0$ depending on $\ga, \ka, \nu,$ and $K_0$, such that the quadratic forms $Q_3$ and $Q_4$ and the commutator term $E(Z, \nabla Z)$ satisfy
\begin{equation}\label{equation lower bound on quadratic terms}
    Q_3(\nabla Z, \nabla Z) + Q_4(Z,Z) - 2E(Z, \nabla Z) \ge \frac{1}{4}|\na Z|^2 - \sg_a^{-1}\bigg(N \log r + \frac{\alpha}{1000000} \bigg) |Z|^2
\end{equation}
on $\cC^\Sg_{\max\{r_6, \ga\rho\}} \times [0, \tau_0]$.
\end{lem}

\begin{proof}
We consider all quantities on the set $\cC^\Sg_{\max\{r, \ga\rho\}} \times [0, \tau_0]$ where $r \ge \max\{r_3,r_4,r_5\}$. The lower bound 
\begin{equation*}
    Q_3(\nabla Z, \nabla Z) \ge \bigg(\frac{7}{24} - \frac{N}{r^2}\bigg)|\nabla Z|^2
\end{equation*}
follows from the same argument as in the proof of \cite[Lemma 5.6]{kotschwarwang2015}. We now estimate $Q_4$ from below, starting with the inequality
\[
Q_4(Z,Z) \ge \frac{1}{2} \bigg( \frac{\dd \wsF_2}{\dd \tau} + \De \wsF_2 + \frac{\dot \te}{\te} \wsF_2\bigg)|Z|^2 - \frac{N|\sF_2|}{r^2}|Z|^2.
\]
From \eqref{equation pointwise bounds on F2} and \eqref{equation lower bound for backward heat on F2}, we know that 
\begin{align*}
 \frac{1}{2} \bigg( \frac{\dd \wsF_2}{\dd \tau} + \De \wsF_2 + \frac{\dot \te}{\te} \wsF_2\bigg) &\ge \frac{3\rho}{2h(\tau + a)^2}\bigg( 1 - \frac{1}{(1 - (1/3)(\tau + a))} \bigg) -\frac{N}{(\tau + a)} \log r \\
 &\qquad - \frac{1}{(\tau + a)(1- (1/3)(\tau + a))}\bigg( N\log r + \frac{1}{16h} \bigg)\\
 &\ge -N(\log r) \sg_a^{-1}.
\end{align*}
We now estimate $\sF_2$.
\begin{align*}
    |\sF_2| = |\wsF_2 - \al \dot{\widehat{\log \sg_a}}| \le (N \log r + \al) \sg_a^{-1}.
\end{align*}
By absorbing appropriate terms and making $r$ sufficiently large, we have
\begin{equation}
    Q_4(Z,Z) \ge - \sg_a^{-1} \bigg( N \log r + \frac{\al}{1000000} \bigg) |Z|^2.
\end{equation}
The same arguments as in \cite[Lemma 5.6]{kotschwarwang2015} allow us to bound the commutator term $E(Z, \nabla Z)$ as follows.
\[|E(Z, \nabla Z)| \le \frac{N}{r^2}(|\nabla Z|^2 + |Z|^2).
\]
Absorbing this quantity gives us the expression \eqref{equation lower bound on quadratic terms} and completes the proof.
\end{proof}

\begin{prop}\label{proposition second Carleman estimate PDE} For any $a \in (0,1)$, $\ga >0$, and $\rho \geq r_0$, there exists $C>0$ depending only on $\ga$, and $\alpha_0 >0$, $r_7 > r_0$ depending only on $\ga$, $\ka$, $\nu$, and $K_0$ such that for any smooth sections $Z$ of $\cZ \times [0,\tau_0]$ which is compactly supported in $\cC^\Sg_{\max\{r_7, \ga \rho\}}  \times [0,\tau_0]$, and satisfies $\cZ(\cdot, 0) \equiv 0$, we have 
\begin{multline}\label{equation second Carleman estimate}
    \sqrt{\al} ||\sigma_a^{-\al - 1/2}Z\hat{\sH}_2^{1/2}||_{L^2(\cC^\Sg_{r_0}  \times [0,\tau_0])} + ||\sigma_a^{-\al} \na Z \hat{\sH}_2^{1/2}||_{L^2(\cC^\Sg_{r_0}  \times [0,\tau_0])}\\
    \leq C||\sigma_a^{-\al} (\dd_{\tau} Z + \De Z) \hat{\sH}_2^{1/2}||_{L^2(\cC^\Sg_{r_0}  \times [0,\tau_0])} + C ||\sigma_a^{-\al - 1/2}(\log r)^\frac{1}{2} Z\hat{\sH}_2^{1/2}||_{L^2(\cC^\Sg_{r_0}  \times [0,\tau_0])}
\end{multline}
for all $\al \geq \al_0$, where
\begin{equation}\label{equation definition new weight second carleman}
\hat{\sH}_2(x, \tau) = \hat{\sH}_{2; a, \rho}(x, \tau) := \exp\Big(-\frac{(r(x) - \rho)^2}{4(\tau + a)}\Big).
\end{equation}
\end{prop}

\begin{proof}
Plug the lower bounds given in \eqref{equation lower bound on quadratic terms} into \eqref{equation integral inequality for backward heat operator} to obtain the following $L^2$ estimates over the domain $\cC = \cC^\Sg_{\max\{r_7, \ga \rho\}}  \times [0,\tau_0]$.
\begin{align*}
    \frac{1}{12e}|| \sg_a^{-\al} \nabla Z \sH_2^{1/2}||^2_{L^2(\cC)} &\le N ||\sg_a^{-\al -1/2}(\log r)^{1/2} Z\sH_2^{1/2}||^2_{L^2(\cC)} \\ 
    &\qquad \qquad + \frac{\al}{500000} ||\sg_a^{-\al -1/2}Z\sH_2^{1/2}||^2_{L^2(\cC)}\\
    &\qquad \qquad \qquad+ ||\sg_a^{-\al}(\dd_\tau Z + \De Z)\sH_2^{1/2}||^2_{L^2(\cC)},
\end{align*}
for any $\al > 1$. Multiplying by constants and taking the square root, obtain
\begin{align}\label{equation bound norm gradient Z}
   ||\sg_a^{-\al} \nabla Z \sH_2^{1/2}||_{L^2(\cC)} &\le N ||\sg_a^{-\al -1/2}(\log r)^{1/2} Z\sH_2^{1/2}||_{L^2(\cC)} \\ 
   &\qquad \qquad + \sqrt{\frac{\al}{5000}} ||\sg_a^{-\al -1/2}Z\sH_2^{1/2}||_{L^2(\cC)} \nonumber \\ 
   &\qquad \qquad + C||\sg_a^{-\al}(\dd_\tau Z + \De Z)\sH_2^{1/2}||_{L^2(\cC)}. \nonumber
\end{align}
Now, apply Lemma \ref{lemma control Z by gradient} to obtain the following inequality.
\begin{align*}
    \frac{\al}{10}||\sg_a^{-\al - 1/2}Z\sH_2^{1/2}||^2_{L^2(\cC)} & \le N ||\sg_a^{-\al -1/2}(\log r)^{1/2}Z\sH_2^{1/2}||^2_{L^2(\cC)}\\
    &\qquad \qquad+ 2|| \sg_a^{-\al} \nabla Z \sH_2^{1/2}||^2_{L^2(\cC)} \\
    &\qquad \qquad + \frac{20}{\al} ||\sg_a^{-\al}(\dd_\tau Z + \De Z)\sH_2^{1/2}||^2_{L^2(\cC)}.
\end{align*}
Multiply by $1/8$ and take the square root of both sides to obtain 
\begin{align}\label{equation bound norm Z}
    \sqrt{\frac{\al}{80}}||\sg_a^{-\al - 1/2}Z\sH_2^{1/2}||_{L^2(\cC)}  & \le N ||\sg_a^{-\al -1/2}(\log r)^{1/2}Z\sH_2^{1/2}||_{L^2(\cC)}\\
    &\qquad \qquad + \frac{1}{2}|| \sg_a^{-\al} \nabla Z \sH_2^{1/2}||_{L^2(\cC)} \nonumber \\
    &\qquad \qquad + \sqrt{\frac{5}{2\al}} ||\sg_a^{-\al}(\dd_\tau Z + \De Z)\sH_2^{1/2}||_{L^2(\cC)}. \nonumber
\end{align}
Add \eqref{equation bound norm gradient Z} and \eqref{equation bound norm Z} together and absorb the $\sqrt{\al / 5000}$ term into the $\sqrt{ \al / 80}$ term. After multiplication by a constant, this yields \eqref{equation second Carleman estimate}.
\end{proof}

\subsubsection{Carleman-type inequality for the ODE component.}

The proof of the following Carleman inequality for the ODE component from \cite{kotschwarwang2015} depends only on the uniform comparability of $g_\tau$ and $g_C$ for $\tau \in [0, \tau_0]$, which is known in our setting by Lemma \ref{lemma decay estimates}. We restate the proposition, which applies in its unmodified form in our setting.

\begin{prop}\label{proposition second carleman estimate ODE}\cite[Propostion 5.9]{kotschwarwang2015}
Suppose $a_0 \in (0, 1/8)$, $\tau_0 \in (0, 1/4)$, and $r_0 \ge 1$. Then there exist constants $N$ and $r_8 \ge r_0$, depending only on $\ka, \nu, A_0$, and $K_0$, such that, for any smooth family $Z = Z(\tau)$ of sections of $\cZ$ with compact support in $\cC^\Sg_{r_8} \times [0, \tau_0)$ which satisfies $Z(\cdot, 0) \equiv 0$, we have, for all $\al > 0$, that
\begin{multline}
||\sg_a^{-\al - \frac{1}{2}} Z \hat\sH_2^{\frac{1}{2}}||_{L^2(\cC^\Sg_{r_0}  \times [0,\tau_0])} \le N \al^{-1}||\sg_a^{-\al} \dd_\tau Z \hat\sH_2^{\frac{1}{2}}||_{L^2(\cC^\Sg_{r_0}  \times [0,\tau_0])} \\ + Na^{-\frac{1}{2}}(\rho + \sqrt{a})^{\frac{n}{2}} \bigg( \frac{1}{ae^{\frac{1}{8}}}\bigg)^\al ||Z||_{\infty, g_C}.
\end{multline}
\end{prop}

\section{Proof of Backward Uniqueness}\label{section proof of backwards uniqueness}

In this section, we prove Theorem \ref{theorem parabolic version main theorem} by proving backward uniqueness along the lines of \cite{EscauriazaSereginSverak} and \cite{kotschwarwang2015}, with some major modifications. First, we apply the second pair of Carleman estimates to prove that $\bf{X}$ and $\bf{Y}$ decay quadratically exponentially in space. Then, we apply the first pair of estimates to prove that $\bf{X}$ and $\bf{Y}$ vanish on some time interval containing $\tau = 0$. Of these two steps, the first presents a number of difficulties due to the presence of the logarithmic growth of $||Z||_2$ on the right-hand side of the PDE-type Carleman inequality. We compensate for this by adjusting our domain and our coefficients to absorb this term into the left-hand side. This introduces a growth term on the right-hand side, which we eventually determine can be absorbed into the quadratic exponential decay at the cost of increasing a universal constant. The second step, the proof of vanishing, on the other hand runs along the same lines as its counterpart in \cite{kotschwarwang2015}.

We briefly note a few choices and conventions that will hold throughout the remainder of the section. As before, the quantities $\vp$, $f$, $g$ will refer to the time-dependent versions with the subscript $\tau$ suppressed, unless necessary for the sake of clarity. We will henceforth denote the annular region given by $\cC^\Sg_{r_1} \setminus \cC^\Sg_{r_2}$ with $r_1 < r_2$ by $A(r_1, r_2)$. We set a constant $r_9 = r_9(\eps, K_0, R_0)\ge r_0$ such that
\begin{equation}\label{equation tensor fields are bounded}
     |\bf{X}| + |\bf{\na X}| + |\bf{Y}| \le N,
\end{equation}
by \eqref{equation spacetime O(r^-2) decay of all quantities}, and
\begin{equation}\label{equation ODE PDE system epsilon version}
    \bigg|\frac{\partial \bf{X}}{\partial t} + \De \bf{X}\bigg| \leq \eps \big(|\bf{X}| + |\nabla \bf{X}| + |\bf{Y}|\big),\;\;\; \bigg|\frac{\partial \bf{Y}}{\partial t}\bigg| \leq N\big(|\bf{X}| + |\bf{\na X}|\big) + \eps |\bf{Y}|,
\end{equation}
by Proposition \ref{proposition ODE PDE system}. Here $N$ depends only on $K_0$. Next, we choose the same spatial cutoff function as in \cite[Lemma 6.1]{kotschwarwang2015} 

\begin{lem}\label{lemma spatial cutoff function}
Given $\rho > 12 r_0$ and $\xi > 4\rho$, there exists $\psi_{\rho, \xi} \in C^\infty(\cC^\Sg_{r_0},[0,1])$ satisfying $\psi_{\rho, \xi} \equiv 1$ on $A(\rho/3, 2\xi)$ and $\psi_{\rho, \xi} \equiv 0$ on $A(r_0, \rho/6) \cup \cC^\Sg_{3\xi}$ whose derivatives satisfy
\begin{equation}\label{derivative estimates for spatial cutoff}
    |\na \psi_{\rho, \xi}| + |\De \psi_{\rho, \xi}| \le N \rho^{-1}
\end{equation}
for $N$ depending on $K_0$. Furthermore, \eqref{derivative estimates for spatial cutoff} holds with the same constant for norms and covariant derivatives taken with respect to $g_\tau$ and the Levi-Civita connection $\nabla^{g_\tau}$ respectively, for each $\tau \in [0,1]$.
\end{lem}

\subsection{Proof of exponential decay} In this section, we show that $\bf{X}$ and $\bf{Y}$ have quadratic exponential decay in space. We prove a statement that is essentially identical to \cite[Claim 6.2]{kotschwarwang2015}---however, the proof has significant differences.

\begin{prop}\label{proposition exponential decay}
There exists an absolute constant $s_0$ and constants $N = N(K_0)$, and $r_{10} = r_{10}(K_0)$, with $s_0 \in (0, 1]$ and $r_{10} \geq r_0$, such that
\begin{equation}\label{equation exponential decay}
|||\bf{X}| + |\nabla \bf{X}| + |\bf{Y}|||_{L^2(A((1-\sqrt{s})\rho,(1+\sqrt{s})\rho)\times [0,s])} \leq N e^{-\frac{\rho^2}{Ns}}
\end{equation}
for any $s \in (0, s_0]$ and $\rho > 13r_{10}$. 
\end{prop}

\begin{proof}
Our setup will largely follow that of the proof of \cite[Claim 6.2]{kotschwarwang2015}, with some important differences. Choose $\ga = 1/12$ and $r_{10} = \max\{r_7, r_8\}$, where $r_7$ and $r_8$ are the quantities given in Proposition \ref{proposition second Carleman estimate PDE} and Proposition \ref{proposition second carleman estimate ODE}. Note that this quantity is allowed to increase as needed throughout the proof. Let $\rho > 0$ satsify $\rho \ge 13r_10$, and let $\xi$ be a number such that $\xi \ge 4 \rho$. The constants given by a $C$ will be universal, whereas those given by $N$ will depend on $A_0$ and $K_0$.

Let $\vartheta \in C^\infty(\bR, [0,1])$ be a temporal cutoff function such that $\vartheta \equiv 1$ for $\tau \le 1/6$ and $\vartheta \equiv 0$ for $\tau \ge 1/5$. Let $\psi_{\rho, \xi}$ be the spatial cutoff function given in Lemma \ref{lemma spatial cutoff function}. Define $\bf{X}_{\rho, \xi} := \vartheta \psi_{\rho, \xi} \bf{X}$ and $\bf{Y}_{\rho, \xi} := \vartheta \psi_{\rho, \xi} \bf{Y}$. These tensor fields are compactly supported on the space-time neighborhood $A(\rho/6, 3\xi) \times [0,1/4)$ and vanish on the time slice where $\tau = 0$. Thus, we may apply Proposition \ref{proposition second Carleman estimate PDE} and Proposition \ref{proposition second carleman estimate ODE} to $\bf{X}_{\rho, \xi}$ and $\bf{Y}_{\rho, \xi}$ respectively and add the results. Note that in the inequality below, we replace the parameter $\al \ge 1$ by an integer $k \ge k_0 \ge \al_0$, we let $a \in (0,1/8)$, and we bound the $||Z||_\infty$ term in Proposition \ref{proposition second carleman estimate ODE} using \eqref{equation tensor fields are bounded}.
\begin{align*}
    &k^{\frac{1}{2}}||\sg_a^{-k -\frac{1}{2}}\bf{X}_{\rho, \xi}\hat \sH_2^{\frac{1}{2}}||_{L^2(A(\frac{\rho}{6}, 3\xi) \times [0,\frac{1}{5}])} + ||\sg_a^{-k}\na \bf{X}_{\rho, \xi} \hat \sH_2^{\frac{1}{2}}||_{L^2(A(\frac{\rho}{6}, 3\xi) \times [0,\frac{1}{5}])} \\
    &\qquad + ||\sg_a^{-k-\frac{1}{2}} \bf{Y}_{\rho, \xi} \hat\sH_2^\frac{1}{2}||_{L^2(A(\frac{\rho}{6}, 3\xi) \times [0,\frac{1}{5}])} \\
    &\le N||\sg_a^{-k}(\dd_\tau + \De)\bf{X}_{\rho, \xi} \hat \sH_2^\frac{1}{2}||_{L^2(A(\frac{\rho}{6}, 3\xi) \times [0,\frac{1}{5}])}+ N(\rho + k^\frac{1}{2})^\frac{7}{2}a^{-\frac{1}{2}}(ae^\frac{1}{8})^{-k} \\
    &+ Nk^{-1}||\sg_a^{-k} \dd_\tau \bf{Y}_{\rho, \xi} \hat \sH_2^\frac{1}{2}||_{L^2(A(\frac{\rho}{6}, 3\xi) \times [0,\frac{1}{5}])}+ N||\sg_a^{-k-\frac{1}{2}}(\log r)^\frac{1}{2} \bf{X}_{\rho, \xi} \hat \sH_2^\frac{1}{2}||_{L^2(A(\frac{\rho}{6}, 3\xi) \times [0,\frac{1}{5}])},
\end{align*}
where $\hat \sH_2 = \hat \sH_{2;a, \rho}$ as defined in \eqref{equation definition new weight second carleman}. Next, instead of sending $\xi \rightarrow \infty$ as in \cite{kotschwarwang2015}, we let $\xi = 8\rho$. Then, we choose a number $k_1$ depending on $\rho$ such that $4N^2\log(24\rho) \le k_1 \le 16N^2 \log(24\rho)$. Observing that the maximum of $\log r$ on $A(\rho/6, 24\rho)$ is $\log(24\rho)$, we obtain the following absorption for $k \ge k_1$.
\[k^\frac{1}{2} - N(\log(24\rho))^\frac{1}{2} \ge \frac{k^\frac{1}{2}}{2}.  \]
Therefore, after increasing $N$ by a factor of 2 and taking $k \geq k_1$, 
\begin{align*}
    &k^{\frac{1}{2}}||\sg_a^{-k -\frac{1}{2}}\bf{X}_{\rho, 8\rho}\hat \sH_2^{\frac{1}{2}}||_{L^2(A(\frac{\rho}{6}, 24\rho) \times [0,\frac{1}{5}])} + ||\sg_a^{-k}\na \bf{X}_{\rho, 8\rho} \hat \sH_2^{\frac{1}{2}}||_{L^2(A(\frac{\rho}{6}, 24\rho) \times [0,\frac{1}{5}])} \\
    &\qquad + ||\sg_a^{-k-\frac{1}{2}} \bf{Y}_{\rho, 8\rho} \hat\sH_2^\frac{1}{2}||_{L^2(A(\frac{\rho}{6}, 24\rho) \times [0,\frac{1}{5}])} \\
    &\le N||\sg_a^{-k}(\dd_\tau + \De)\bf{X}_{\rho, 8\rho} \hat \sH_2^\frac{1}{2}||_{L^2(A(\frac{\rho}{6}, 24\rho) \times [0,\frac{1}{5}])} \\
    &\qquad + N(\rho + k^\frac{1}{2})^\frac{7}{2}a^{-\frac{1}{2}}(ae^\frac{1}{8})^{-k}+ Nk^{-1}||\sg_a^{-k} \dd_\tau \bf{Y}_{\rho, 8\rho} \hat \sH_2^\frac{1}{2}||_{L^2(A(\frac{\rho}{6}, 24\rho) \times [0,\frac{1}{5}])}.
\end{align*}
Although we have absorbed the logarithmic growth term, we have allowed $k_1$ to depend on $\rho$ which will introduce a new growth term later on. Next, we apply \eqref{equation ODE PDE system epsilon version} on $\cC^\Sg_{r} \times [0,1/4]$ where $r \ge r_9(\eps)$ to obtain the following inequalities.
\begin{multline*}
 \bigg|\frac{\partial \bf{X}_{\rho, 8\rho} }{\partial t} + \De \bf{X}_{\rho, 8\rho} \bigg| \leq \eps \big(|\bf{X}_{\rho, 8\rho} | + |\nabla \bf{X}_{\rho, 8\rho} | + |\bf{Y}_{\rho, 8\rho} |\big) + \psi_{\rho, 8\rho}|\vartheta'||\bf{X}| \\ + \vartheta\big(2|\na \psi_{\rho, 8\rho}||\na \bf{X}| + \eps|\na \psi_{\rho, 8\rho}||\bf{X}| + |\De \psi_{\rho, 8\rho}||\bf{X}|\big)
\end{multline*}
\begin{equation*}
    \bigg|\frac{\partial \bf{Y}_{\rho, 8\rho}}{\partial t}\bigg| \leq N\big(|\bf{X}_{\rho, 8\rho}| + |\na \bf{X}_{\rho, 8\rho}|\big) + \eps |\bf{Y}_{\rho, 8\rho}| + N\vartheta|\na \psi_{\rho, 8\rho}||\bf{X}| + \psi_{\rho, 8\rho}|\vartheta'||\bf{Y}|
\end{equation*}
Making $\eps > 0$ sufficiently small, and taking $k_1$ large enough that $N/k_1 < \eps$, we can reabsorb the terms involving $|\bf{X}_{\rho, 8\rho}|, |\nabla \bf{X}_{\rho, 8\rho} |,$ and $|\bf{Y}_{\rho, 8\rho}|$ into the right hand side and obtain the following inequality.
\begin{align*}
    &||\sg_a^{-k-\frac{1}{2}}(|\bf{X}_{\rho, 8\rho}| + |\bf{Y}_{\rho, 8\rho}|)\hat \sH_2^{\frac{1}{2}}||_{L^2(A(\frac{\rho}{6}, 24\rho) \times [0,\frac{1}{5}])} + ||\sg_a^{-k} \na \bf{X}_{\rho, 8\rho} \hat \sH_2^\frac{1}{2}||_{L^2(A(\frac{\rho}{6}, 24\rho) \times [0,\frac{1}{5}])} \\
    &\qquad \le N||\sg_a^{-k-\frac{1}{2}}(|\bf{X}| + |\bf{Y}|)\hat \sH_2^\frac{1}{2}||_{L^2(A(\frac{\rho}{6}, 24\rho) \times [\frac{1}{6},\frac{1}{5}])} \\
    &\qquad \qquad  + N||\sg_a^{-k}(|\bf{X}| + |\na \bf{X}|)\hat \sH_2^\frac{1}{2}||_{L^2(A(\frac{\rho}{6}, \frac{\rho}{3}) \times [0,\frac{1}{5}])} \\
    &\qquad \qquad  \qquad + N||\sg_a^{-k}(|\bf{X}| + |\na \bf{X}|)\hat \sH_2^\frac{1}{2}||_{L^2(A(16\rho, 24\rho) \times [0,\frac{1}{5}])}\\
    &\qquad \qquad  \qquad \qquad + N(\rho + k^\frac{1}{2})^\frac{7}{2}a^{-\frac{1}{2}}(ae^\frac{1}{8})^{-k}.
\end{align*}
Note that we use the fact that $\sg_a \le 1$ and that as before, $a \in (0, 1/8)$. Since $\sigma_a \le 1$ and $1 \le e^{(\tau + a)/3} \le e$, we can further estimate:
\begin{align}\label{equation pre estimate inequality for L2 of tensor fields}
    &||(\tau + a)^{-k}(|\bf{X}_{\rho, 8\rho}| + |\bf{Y}_{\rho, 8\rho}| +  |\na \bf{X}_{\rho, 8\rho}| )\hat \sH_2^{\frac{1}{2}}||_{L^2(A(\frac{\rho}{6}, 24\rho) \times [0,\frac{1}{5}])} \\
    &\qquad \le C^k N||(\tau + a)^{-k}(|\bf{X}| + |\bf{Y}|)\hat \sH_2^\frac{1}{2}||_{L^2(A(\frac{\rho}{6}, 24\rho) \times [\frac{1}{6},\frac{1}{5}])} \nonumber \\
    &\qquad \qquad  + C^k N||(\tau + a)^{-k}(|\bf{X}| + |\na \bf{X}|)\hat \sH_2^\frac{1}{2}||_{L^2(A(\frac{\rho}{6}, \frac{\rho}{3}) \times [0,\frac{1}{5}])} \nonumber \\
    &\qquad \qquad  \qquad + C^k N||(\tau + a)^{-k}(|\bf{X}| + |\na \bf{X}|)\hat \sH_2^\frac{1}{2}||_{L^2(A(16\rho, 24\rho) \times [0,\frac{1}{5}])} \nonumber \\
    &\qquad \qquad  \qquad \qquad + N(\rho + k^\frac{1}{2})^\frac{7}{2}a^{-\frac{1}{2}}(ae^\frac{1}{8})^{-k}. \nonumber 
\end{align}
The first term on the right-hand side of the inequality can be estimated by the bounds \eqref{equation tensor fields are bounded} and $(\tau + a) \ge 1/6$.
\begin{equation*}
    ||(\tau + a)^{-k}(|\bf{X}| + |\bf{Y}|) \hat \sH_2^\frac{1}{2}||_{L^2(A(\frac{\rho}{6}, 24\rho) \times [\frac{1}{6},\frac{1}{5}])} \le 6^k || N ||_{L^2(A(\frac{\rho}{6}, 24\rho) \times [\frac{1}{6},\frac{1}{5}])} \le C^kN\rho^\frac{7}{2}
\end{equation*}
for all $k\geq k_1$ and $a\in (0, 1/8)$. 

On the domains $A(\frac{\rho}{6},\frac{\rho}{3})\times [0, \frac{1}{5}]$ and $A(16\rho, 24\rho)\times [0, \frac{1}{5}]$, we have that
$$e^{-(r-\rho)^2/(8(\tau+a))} \leq e^{-\rho^2/(18(\tau+a))}.$$
We compute using Stirling's formula that
\[
 \max_{s>0} s^{-k}e^{-\rho^2/(18 s)}  = \rho^{-2k}(18 k)^ke^{-k} \leq \rho^{-2k}C^k k!.
\]
This observation yields the following bound.
\begin{multline*}
     ||(\tau + a)^{-k}(|\bf{X}| + |\na \bf{X}|)\hat \sH_2^\frac{1}{2}||_{L^2(A(\frac{\rho}{6}, \frac{\rho}{3}) \times [0,\frac{1}{5}])} \\+ ||(\tau + a)^{-k}(|\bf{X}| + |\na \bf{X}|)\hat \sH_2^\frac{1}{2}||_{L^2(A(16\rho, 24\rho) \times [0,\frac{1}{5}])} \le N\rho^{-2k}C^kk!\rho^{\frac{7}{2}},
\end{multline*}
for $k \ge k_1$ and $a \in (0, 1/8)$. Let $l$ be such that $k = k_1 + l$. Plugging the previous two estimates back into \eqref{equation pre estimate inequality for L2 of tensor fields}, we find
\begin{align}\label{equation preliminary estimate on L2 of tensor field}
     &||(\tau + a)^{-k}(|\bf{X}_{\rho, 8\rho}| + |\bf{Y}_{\rho, 8\rho}| +  |\na \bf{X}_{\rho, 8\rho}| )\hat \sH_2^{\frac{1}{2}}||_{L^2(A(\frac{\rho}{6}, 24\rho) \times [0,\frac{1}{5}])} \\
     &\qquad \le N\rho^\frac{7}{2}C^{k_1 + l}(1 + \rho^{-2(k_1 + l)}(k_1 + l)!) + N(\rho + (k_1 +l)^\frac{1}{2})^\frac{7}{2}a^{-\frac{1}{2}}(ae^\frac{1}{8})^{-(k_1 + l)} \nonumber
\end{align}
for all $l\geq 0$ and all $a\in (0, 1/8)$. We use several inequalities which we collect here:
\[(l+k_1)! \le C^{k_1 + l}(l!)(k_1!), \;\;   l^\frac{7}{4}e^{-\frac{l}{16}} \le C, \;\;  (\tau + a)^{-k_1} \ge 1.
\]
In addition, for $r_{10}$ sufficiently large, if we recall that $k_1 \le 16N^2\log 24 \rho$ we can deduce that
\begin{align*}
    (\rho + (k_1 + l)^\frac{1}{2})^\frac{7}{2} &\le (\rho + k_1^\frac{1}{2} + l^\frac{1}{2})^\frac{7}{2} \\
    &\le (2\rho + l^\frac{1}{2})^\frac{7}{2} \\
    &\le C\rho^\frac{7}{2}( 1 + \big(\tfrac{l}{\sqrt{2\rho}}\big)^\frac{1}{2})^\frac{7}{2}\\
    &\le C\rho^\frac{7}{2}l^\frac{7}{4}
\end{align*}
Note that in each of these inequalities, $C>0$ is a universal constant. Plugging these in to \eqref{equation preliminary estimate on L2 of tensor field}, we find
\begin{align*}
     &||(\tau + a)^{-k}(|\bf{X}_{\rho, 8\rho}| + |\bf{Y}_{\rho, 8\rho}| +  |\na \bf{X}_{\rho, 8\rho}| )\hat \sH_2^{\frac{1}{2}}||_{L^2(A(\frac{\rho}{6}, 24\rho) \times [0,\frac{1}{5}])} \\
     &\qquad \le N\rho^\frac{7}{2}(C^{k_1}C^l + C^{k_1}(k_1!)C^l\rho^{-2l}(l!) + a^{-\frac{1}{2}-k_1}(l^{\frac{7}{4}}e^{-\frac{l}{16}})(ae^\frac{1}{16})^{-l})\\
     &\qquad \le N\rho^\frac{7}{2}(C^{k_1}C^l + C^{k_1}(k_1!)C^l\rho^{-2l}(l!) + a^{-\frac{1}{2}-k_1}(ae^\frac{1}{16})^{-l})\\
     &\qquad \le N(k_1!)^2\rho^\frac{7}{2}(C^l + C^l\rho^{-2l}(l!) + a^{-\frac{1}{2}-k_1}(ae^\frac{1}{16})^{-l})
\end{align*}
Note that since $C>0$ is a universal constant, we may set $r_{10}$ to be sufficiently large that $k_1! \ge C^{k_1}$: thus, we can bound all terms involving $k_1$ by $(k_1!)^2$, with the possible exception of $a^{-(k_1 + 1/2)}$. We multiply both sides by $\rho^{2l}/((2C)^ll!)$ and sum over all $l\geq 0$ yielding
\begin{align*}
     &||(|\bf{X}_{\rho, 8\rho}| + |\bf{Y}_{\rho, 8\rho}| +  |\na \bf{X}_{\rho, 8\rho}| )e^{\frac{\rho^2}{C(\tau+a)}-\frac{(r -\rho)^2}{8(\tau + a)}}||_{L^2(A(\frac{\rho}{6}, 24\rho) \times [0,\frac{1}{5}])} \\
     &\qquad \le N(k_1!)^2\rho^\frac{7}{2}\bigg(1 + e^\frac{\rho^2}{2} + a^{-(k_1 + \frac{1}{2})}e^\frac{\rho^2}{Cae^{1/6}}\bigg).
\end{align*}
Write $e^{1/16} = 1 +2\delta$ and restrict the time interval from $[0, 1/5]$  to $[0, \de a]$. Then
\begin{align*}
     &||(|\bf{X}_{\rho, 8\rho}| + |\bf{Y}_{\rho, 8\rho}| +  |\na \bf{X}_{\rho, 8\rho}| )e^{-\frac{(r -\rho)^2}{8(\tau + a)}}||_{L^2(A(\frac{\rho}{6}, 24\rho) \times [0,\de \al])} \\
     &\qquad \le N(k_1!)^2\rho^\frac{7}{2}\bigg(\bigg(1 + e^\frac{\rho^2}{2}\bigg)e^{-\frac{\rho^2}{Ca(1+\de)}} + a^{-(k_1 + \frac{1}{2})}e^\frac{-\de \rho^2}{Ca(1+2\de)(1+\de)}\bigg).
\end{align*}
All that remains is to find a constant depending on $k_1$ which bounds
$$a^{-(k_1+1/2)}e^{-\delta\rho^2/(2Ca(1+2\delta)(1+\delta))}$$
from above. The maximum of this quantity as a function of $a$ occurs at its sole critical point on the positive real axis.
\begin{align*}
0 &= \frac{d}{da}\big(a^{-(k_1+1/2)}e^{-\delta\rho^2/(2Ca(1+2\delta)(1+\delta))}\big)\\
    &= -(k_1+1/2)a^{-1}a^{-(k_1+1/2)}e^{-\delta\rho^2/(2Ca(1+2\delta)(1+\delta))}\\
    & \qquad + \frac{\delta\rho^2}{2Ca^2(1+2\delta)(1+\delta)}a^{-(k_1+1/2)}e^{-\delta\rho^2/(2Ca(1+2\delta)(1+\delta))} \\
    &= -(k_1+1/2)a^{-1} + \frac{\delta\rho^2a^{-2}}{2C(1+2\delta)(1+\delta)}
\end{align*}
Thus, the critical point occurs at
\begin{equation}
    a_0 := a = \frac{2C(1+2\delta)(1+\delta)(k_1+1/2)}{\delta\rho^2}
\end{equation}
To find an upper bound, we plug $a_0$ into $a^{-(k_1+1/2)}e^{-\delta\rho^2/(2Ca(1+2\delta)(1+\delta))}$.
\begin{align*}
a_0^{-(k_1+1/2)}e^{-\delta\rho^2/(2Ca_0(1+2\delta)(1+\delta))} &\le \bigg(\frac{C(k_1 +1/2)}{\rho^2}\bigg)^{-(k_1 +1/2)} \le (\rho^2)^{2k_1}
\end{align*}
This estimate allows us to obtain the following upper bound. 
\[
     ||(|\bf{X}_{\rho, 8\rho}| + |\bf{Y}_{\rho, 8\rho}| +  |\na \bf{X}_{\rho, 8\rho}| )e^{-\frac{(r -\rho)^2}{8(\tau + a)}}||_{L^2(A(\frac{\rho}{6}, 24\rho) \times [0,\de \al])} \le N(\rho^{4k_1})(k_1^{k_1})^2 e^\frac{-\rho^2}{Ca},
\]
for all $0 < a\leq 1/C$. Recall that $k_1 \leq 16N^2\log(24\rho) \le 17N^2\log(\rho)$. Then,
\begin{align*}
    \rho^{4k_1} \cdot k_1^{2k_1} &\le \rho^{68N^2\log(\rho)} \cdot 17N^2\log(\rho)^{17N^2\log(\rho)}\\
    &= \exp(68N^2\log(\rho)\log(\rho) + 17N^2\log(\rho)(\log(17N^2) + \log(\log(\rho))))\\
    &\le \exp(N(\log \rho)^2),
\end{align*}
where we have simply increased $N$ if necessary. This grows much more slowly than quadratic exponential. By possibly increasing $r_{10}$, we finally obtain
\[
     ||(|\bf{X}_{\rho, 8\rho}| + |\bf{Y}_{\rho, 8\rho}| +  |\na \bf{X}_{\rho, 8\rho}| )e^{-\frac{(r -\rho)^2}{8(\tau + a)}}||_{L^2(A(\frac{\rho}{6}, 24\rho) \times [0,\de \al])} \le N e^\frac{-\rho^2}{Ca},
\]
Observe that $e^{-|r -\rho|^2/(8(\tau+a))} \geq N^{-1}$ on the domain $A((1-\sqrt{\delta a})\rho, (1+\sqrt{\delta a})\rho)\times [0, \delta a]$. Furthermore, on this set, $\psi_{\rho, 8\rho} \equiv 1$ and $\vp \equiv 1$. Thus, the cut-off tensor fields agree with their unmodified counterparts on this restricted domain and we obtain the estimate 
\begin{align}
\begin{split}
    \||\bf{X}|+|\nabla \bf{X}| +|\bf{Y}|\|_{L^2(A((1-\sqrt{s})\rho,(1+\sqrt{s})\rho)\times[0, s])}
\leq Ne^{-\frac{\rho^2}{Cs}}
\end{split}  
\end{align}
for $s\in [0, 1/C]$ and arbitrary $\rho \geq 12r_{10}$, which was the desired exponential decay. 
\end{proof}

\subsection{Proof of Theorem \ref{theorem main}} 

The main theorem, Theorem \ref{theorem main}, follows immediately from the following proposition:

\begin{prop}\label{proposition vanishing}\cite[Claim 6.3]{kotschwarwang2015}
There exist $\tau' \in (0,1)$ and $r_{11} = r_{11}(K_0)$ such that $\bf{X} \equiv 0$ and $\bf{Y} \equiv 0$ on $\cC^\Sg_{r_{11}}\times [0,\tau']$.
\end{prop}

\begin{proof}
The proof is identical to that of \cite[Claim 6.3]{kotschwarwang2015}, after replacing \cite[Claim 6.2]{kotschwarwang2015} by Proposition \ref{proposition exponential decay}, \cite[(6.2)]{kotschwarwang2015} by \eqref{equation ODE PDE system epsilon version}, \cite[(4.13)]{kotschwarwang2015} by \eqref{equation comparison of h and r}, and \cite[Proposition 4.9]{kotschwarwang2015} by Proposition \ref{proposition first carleman estimate}. 
\end{proof}

This immediately implies Theorem \ref{theorem parabolic version main theorem}. By the reduction of Theorem \ref{theorem main} to Theorem \ref{theorem parabolic version main theorem} in Section \ref{section setting up the parabolic problem}, this proves Theorem \ref{theorem main}. 

\subsection{Proof of Corollary \ref{corollary main}}
We may prove Corollary \ref{corollary main} using Theorem \ref{theorem main}.

From Theorem \ref{theorem main}, we have that there exist ends $W_1 \subset V_1$ and $W_2 \subset V_2$ and a diffeomorphism $\Psi: W_1 \to W_2$ such that $\Psi^* \vp_2 = \vp_1$. 

From the proof of Theorem \ref{theorem main}, we have that $W_1$ is diffeomorphic to $\cC_{r_1}^{\Sg}$ for some $r_1>r$, but by assumption $V_1$ is diffeomorphic to $C_r^{\Sg}$ for $r\leq r_1$. Thus, the inclusion map $\io: W_1 \hookrightarrow V_1$ induces an isomorphism $\pi_1(W_1, x_0) \cong \pi_1(V_1, x_0)$. Likewise, $\pi_1(W_2, y_0) \cong \pi_1(V_2, y_0)$. Thus, the homomorphisms 
\begin{equation}\label{equation induced homomorphisms}\pi_1(W_1, x_0) \rightarrow \pi_1(M_1, x_0), \,\, \pi_1(W_2, x_0) \rightarrow \pi_1(M_2, x_0)\end{equation}
induced by inclusion are surjective.

Since $\Psi^* \vp_2 = \vp_1$, we have that $\Psi^* g_{\vp_2} = g_{\vp_1}$. Thus, $\Psi$ is an isometry from $W_1$ to $W_2$. By~\cite[Theorem 5.2]{KotschwarWangIsometries}, ~\eqref{equation induced homomorphisms} implies that $\Psi: W_1 \to W_2$ extends to an isometry $\Psi: (M_1, g_{\vp_1}) \to (M_2, g_{\vp_2})$. Note that this is in particular a diffeomorphism.

By~\cite[Cor. 1.2]{LotayWeiRealAnalyticity}, Laplacian shrinkers are real analytic. Since an isometry between real analytic manifolds is itself real analytic, we find that $\Psi$ is real analytic. Defining $h:= \Psi^* \vp_2 - \vp_1$, we see that $h$ is real analytic, and $h \equiv 0$ on $W_1$. Since our shrinkers are connected by assumption, we find that $h \equiv 0$ everywhere on $M_1$. This gives a diffeomorphism $\Psi: M_1 \to M_2$ such that $\Psi^* \vp_2 = \vp_1$ and concludes the Corollary.

\section{Automorphisms of AC \texorpdfstring{$G_2$}{G2}-Shrinkers}\label{section automorphisms of shrinkers}

For a closed $G_2$-structure $(M, \vp)$, we define $\Aut(M, \vp)$ by
$$\Aut(M, \vp) = \{F \in \Diff(M) \,|\,F^* \vp = \vp\}.$$
We will call any such $F$ satisfying $F^*\vp = \vp$ an automorphism of the $G_2$-structure $\vp$. Note that if $F \in \Aut(M, \vp)$, then $F \in \Isom(M, g_{\vp})$. 

In this section, we adapt the arguments of~\cite{KotschwarWangIsometries} to the Laplacian flow setting. We prove the following theorem, which implies Theorem \ref{theorem isometries}.

\begin{thm}\label{theorem isometries dynamically asymptotic}
Let $(\cC_{r}^{\Sg}, \vp, f)$ be a gradient Laplacian shrinker which is dynamically asymptotic to a closed $G_2$-cone $(\cC_r^{\Sg}, \vp_C)$. Then, $\Aut(\cC_r^{\Sg}, \vp) = \Aut(\cC_r^{\Sg}, \vp_C)$ as subsets of $\Diff(\cC_r^{\Sg})$.
\end{thm}

\subsection{Theorem \ref{theorem isometries dynamically asymptotic} implies Theorem \ref{theorem isometries}}

If $(M, \vp, f)$ is asymptotic to $(\cC^\Sg, \vp_C)$ along the end $V$, then by Proposition \ref{proposition backward flow reparametrization}, there is $r>0$, an end $W \subset V$, and a diffeomorphism $F: \cC_r^{\Sg} \to W$ such that $(\cC_r^{\Sg}, F^*\vp, F^*f)$ is dynamically asymptotic to $(\cC_r^{\Sg}, \vp_C)$. For $\ga \in \Aut(\cC_r^{\Sg}, F^*\vp)$, $F\circ\ga\circ F^{-1}$ is an automorphism of $(W, \vp)$, and in fact, 
$$\ga \mapsto F \circ \ga \circ F^{-1}$$
is an isomorphism from $\Aut(\cC_r^{\Sg}, F^*\vp)$ to $\Aut(W, \vp)$. Assuming Theorem \ref{theorem isometries dynamically asymptotic}, 
$$\Aut(\cC_r^{\Sg}, F^*\vp) = \Aut(\cC^\Sg, \vp_C).$$
This shows that $\Aut(\cC^\Sg, \vp_C)$ and $\Aut(W, \vp)$ are isomorphic by the desired map, and thus that Theorem \ref{theorem isometries dynamically asymptotic} implies Theorem \ref{theorem isometries}.

\subsection{Real analytic structure of the flow}

A $G_2$-structure $(M, \vp)$ is real analytic if $M$ is a real analytic manifold, and there is a subatlas of real analytic coordinate systems on $M$, such that the components $(g_{\vp})_{ij}$ of $g_{\vp}$ and $\vp_{ijk}$ of $\vp$ are real analytic functions in these coordinates. In fact, when we say that $(M, \vp)$ is real analytic, we always mean that $g_{\vp}$ and $\vp$ are real analytic relative to the subatlas of normal coordinates induced by $g_{\vp}$, i.e.\ $(g_{\vp})_{ij}$ and $\vp_{ijk}$ are real analytic functions in $g_{\vp}$-normal coordinates. See~\cite[\S 3]{LotayWeiRealAnalyticity} for additional background.

In this subsection, we follow the idea of~\cite[\S 3]{KotschwarWangIsometries} to prove the constancy in time of the real analytic structure of Laplacian flow. This will be used to prove Theorem \ref{theorem isometries dynamically asymptotic}. The next lemma follows the idea of~\cite[\S 3]{KotschwarWangIsometries}.

\begin{lem}\label{lemma Kotschwar Wang analyticity claim}
Let $\vp_t$ be a smooth, possibly incomplete solution to Laplacian flow on $\cD \subset M$ for $t \in [0,T_0]$. Let $t_0 \in (0, T_0)$, and let $U = B_r(x_0)$ be a $g_{t_0}$-normal coordinate patch in $(\cD, \vp_{t_0})$ of radius $r>0$ around $x_0$ such that in $U$-coordinates $(g_{t_0})_{ij}$ are real analytic in $U$. 

Then, for any $x_0$ in $V = B_{r/4}(x_0)$, there exists $\eps>0$, depending only on $r$, $T_0-t_0$, and $C_0 := \sup_{U \times [0,T_0]}|\Rm|_{g_t}(x,t)$, such that in $U$-coordinates,
$$\sup_{V\times [t_0, t_0 + \eps]}\Big(|\dd^{k}g|(x,t)+|\dd^{k}\vp|(x,t)\Big) \leq C L^k k!$$
for all $k \geq 0$, where $C$ and $L$ are constants depending only on $r$, $T_0-t_0$, $C_0$, and $g_{t_0}$. Moreover, if $C_0$ and $r$ are fixed, then $\eps = T_0-t_0$ for all $t_0$ close enough to $T_0$. 
\end{lem}
\begin{rmk}
The dependence of $\eps$ on $r$, $C_0$, and $T_0 - t_0$ is such that $\eps$ is uniformly bounded below in terms of a lower bound on $r$ and $T_0- t_0$ and and upper bound on $C_0$.
\end{rmk}
\begin{proof}
First, we recall that for any closed $G_2$-structure, $\na T = \Rm \ast \vp + T \ast T \ast \vp$, where $T$ is the torsion of the $G_2$-structure. Since $|\vp|_{g} =7$ and $R = -|T|^2$, there exists a universal constant $C$ such that $|\na T| \leq C|\Rm|$. If 
$$\La(x,t) := (|\Rm|^2(x,t) + |\na T|^2(x,t))^{1/2},$$
then $\La(x,t) \leq C |\Rm|(x,t)$ for some universal constant $C$. Thus, 
\begin{equation}\label{equation bound on Lotay Wei's Lambda}\sup_{V \times [0,T_0]} \La(x,t) \leq C C_0.\end{equation}
Combining~\eqref{equation bound on Lotay Wei's Lambda} with Lotay--Wei's local estimates of real analyticity~\cite[Theorem 6.3]{LotayWeiRealAnalyticity}, for any $s_0 \in (0, T_0]$, there exist $C, L, \eps_1>0$, depending only on $T_0-s_0$, $C_0$, and $r$, such that 
$$\sup_{V} \,(t-s_0)^{k/2} \Big(|\na^k \Rm|(x,t)+|\na^{k+1}T|(x,t)+|\na^{k+2} \vp|(x,t)\Big)\leq C L^{k/2} (k+1)!$$
for all $k \geq 0$ and $t \in (0, T_0]$ such that $t \in (s_0, s_0 +\eps_1]$. By the argument in~\cite[Theorem 6.3]{LotayWeiRealAnalyticity}, if $C_0$ and $r$ are bounded, then for all $T_0-s_0$ small enough, $\eps_1 = T_0-s_0$. Since $\eps_1$ only depends on $T_0-s_0$, $C_0$, and $r$, then for any $t_0 \in (0, T_0]$, there exists $s_0 \in [0,t_0)$ and $\eps>0$ such that 
$$[t_0-\eps, t_0 + \eps]\cap (0, T_0] \subset (s_0+\frac{\eps_1}{2}, s_0+\eps_1)\cap (0, T_0].$$ 
Then, there exists a relabelled constant $L$, only modified by a constant depending on $\eps_1$, such that
\begin{equation}\label{equation Lotay Wei first analyticity estimate}\sup_{V \times \cI} \Big(|\na^k \Rm|+|\na^{k+1}T|+|\na^{k+2} \vp|\Big)\leq C L^{k/2} (k+1)!\end{equation}
where $\cI = [t_0 - \eps, t_0 + \eps]\cap (0, T_0]$.

Now, recall from~\cite[\S 3]{KotschwarTimeAnalyticity} that there exists a universal constant $a>0$ such that
$$m! n! \leq (m+n)! \leq a^{m+n} m! n!$$
This implies that 
$$C L^{k/2}(k+1)! \leq 6Ca (La^2)^{k/2}(k-2)!$$
We rewrite the bound~\eqref{equation Lotay Wei first analyticity estimate} as 
\begin{equation}\label{equation Lotay Wei second analyticity estimate}\sup_{V \times \cI} \Big(|\na^k \Rm|+|\na^{k+1}T|+|\na^{k+2} \vp|\Big)\leq C L^{k/2} (k-2)!\end{equation}
where $C$ and $L$ are relabelled constants which have only been modified by universal constants. Now that we have estimates of the form~\eqref{equation Lotay Wei second analyticity estimate}, we may apply results from~\cite[\S 8]{KotschwarTimeAnalyticity}, as in~\cite[Proof of Theorem 1.3]{KotschwarWangIsometries}.

First, define 
$$B_{ij} := \Ric_{ij} + \frac{1}{3}|T|^2g_{ij} + 2T_{i}^{\;l}T_{lj},$$
so that $\frac{\dd g_{ij}}{\dd t} = -2B_{ij}$. Taking covariant derivatives of $B_{ij}$ with respect to $g(t)$,
\begin{align}\label{equation first estimate on derivatives of B}
    \sup_{V\times \cI}|\na^{k} B|- \sup_{V\times \cI}|\na^k \Ric| &\leq C\sup_{V\times \cI}\sum_{n=0}^{k} \binom{k}{n}|\na^{k-n}T||\na^n T|\nonumber\\
    &\leq a^k\sup_{V\times \cI}\sum_{n=0}^{k}|\na^{k-n}T||\na^n T| \nonumber\\
    &\leq a^k C^2 \sum_{n=0}^{k} L^{\frac{k-n-1}{2}} L^{\frac{n-1}{2}} (k-n-3)!(n-3)!\nonumber\\
    &\leq a^k C^2 \sum_{n=0}^{k} L^{\frac{k}{2}-1} (k-6)!\nonumber\\
    &\leq  a^k (k+1) C^2L^{\frac{k}{2}-1} (k-6)!
\end{align}
After modifying $C$ and $L$ by universal constants, 
\begin{equation}\label{equation second estimate on derivatives of B}
\sup_{V\times \cI} |\na^k B| \leq CL^{\frac{k}{2}} (k-2)!
\end{equation}
for all $k \geq 0$. Now, let $\na$ be the Levi-Civita connection of $g_{t}$, and for $t_0 \in (0,T_0)$, let $\bar{g} := g_{t_0}$, $\na^{\bar{g}} = \ov{\na}$, and $G:= \na - \ov{\na}$. Then, for $t_0 \in (0,T_0)$, we apply ~\cite[Proposition 27]{KotschwarTimeAnalyticity} with ~\eqref{equation second estimate on derivatives of B} on $[t_0, t_0+\eps]\cap (0,T_0]$ to find relabelled constants $\eps$, $C$, and $L$, depending only on $r$, $T_0-t_0$, and $C_0$, such that
\begin{equation}\label{equation bound on difference of connections G}
    \sup_{V\times [t_0, t_0+\eps]} |\na^k G|(x,t) \leq C L^{k/2} (k-2)!
\end{equation}
for all $k \geq 0$. Here, $[t_0, t_0+\eps]\subset (0,T_0]$ by the new choice of $\epsilon$. Indeed, by the argument in~\cite[Proposition 27]{KotschwarWangIsometries}, if $C$ and $L$ are bounded, then for all $t_0$ close enough to $T_0$, we have that $\eps = T_0-t_0$. By~\cite[Proposition 25]{KotschwarTimeAnalyticity},~\eqref{equation Lotay Wei second analyticity estimate} and~\eqref{equation bound on difference of connections G} imply that there exist relabelled constants $C$ and $L$ depending only on $r$, $T_0-t_0$, and $C_0$, such that 
\begin{equation}\label{equation phi bar derivative bounds}
    \sup_{V \times [t_0, t_0 + \eps]} |\ov{\na}^{k} \vp| \leq C L^{k/2} (k-2)!
\end{equation}
Since $|\na^k g| = 0$ for all $k \geq 0$, estimates of the form~\eqref{equation Lotay Wei second analyticity estimate} automatically hold for $|\na^k g|$ as well. Just as for~\eqref{equation phi bar derivative bounds}, Proposition 25 of~\cite{KotschwarTimeAnalyticity} goes to show that
\begin{equation}\label{equation g bar derivative bounds}
    \sup_{V \times [t_0,  t_0 + \eps]} |\ov{\na}^k g|\leq C L^{k/2} (k-2)!
\end{equation}
By assumption, $\bar{g}$ is analytic in $U$, so the Christoffel symbols $\bar{\Ga}_{ij}^k$ are also analytic in $U$. Thus, for each $x_0 \in U$, 
$$|\dd^{k}\bar{\Ga}| \leq \tilde{C} \tilde{L} k!$$ 
in some neighborhood of $x_0$, where $\tilde{C}$ and $\tilde{L}$ are constants not depending on $k$. Using a covering argument on the $\bar{g}$-closure of $V$, there exist constants $\tilde{C}$ and $\tilde{L}$ such that in a neighborhood of $x_0$,
\begin{equation}\label{equation estimates on christoffel symbols for analyticity}
    |\dd^{k}\bar{\Ga}|\leq \tilde{C}\tilde{L}k!
\end{equation}
for all $k \geq 0$. Writing $\bar{\Ga} = \bar{\na} - \dd$, we may combine~\eqref{equation estimates on christoffel symbols for analyticity} with~\eqref{equation g bar derivative bounds} and~\eqref{equation phi bar derivative bounds} to find constants $C$ and $L$ such that
\begin{equation}\label{equation final in analyticity lemma}\sup_{V\times [t_0,  t_0 + \eps]}\Big(|\dd^{k}g|(x,t)+|\dd^{k}\vp|(x,t)\Big) \leq C L^k k!\end{equation}
for all $k \geq 0$. The constants $\tilde{C}$ and $\tilde{L}$ from~\eqref{equation estimates on christoffel symbols for analyticity} depend on $\bar{g} = g_{t_0}$ in $U$, and the constants $C, L$ in~\eqref{equation g bar derivative bounds} and ~\eqref{equation phi bar derivative bounds} depend on $T_0-t_0$, $C_0$, and $r$. This gives the desired dependence for $C$, $L$ in~\eqref{equation final in analyticity lemma}. Similarly, $\eps$ depends on $\eps_1$, which itself depends only on $T_0-t_0$, $r$, and $C_0$. This gives the desired dependence for $\eps$, and the statement for $t_0$ close to $T_0$ follows from the arguments in~\cite[Theorem 6.3]{LotayWeiRealAnalyticity} and~\cite[Proposition 27]{KotschwarTimeAnalyticity}.
\end{proof}

\begin{prop}\label{proposition same analytic structure}
Suppose $(\cD, \vp_{t})$ is a smooth, possibly incomplete solution to Laplacian flow on $\cD \subset M$ for $t \in [0, T_0]$. Then,  $(\cD, \vp_{t})$ is real analytic with respect to the same analytic structure for all $t \in (0, T_0]$. In fact, $g_{t}$ and $\vp_t$ are real analytic relative to the subatlas of normal coordinates induced by $g_{t_0}$ for any $t_0 \in (0, T_0]$. 
\end{prop}
\begin{proof}
By~\cite[Theorem 1.1]{LotayWeiRealAnalyticity}, for each $t \in (0, T_0]$, the Laplacian flow $(\cD, \vp_{t})$ is real analytic with respect to the analytic structure $\cA_t$ induced by the subatlas of $g_t$-normal coordinate systems. To conclude the lemma, we will show that $(\cD, \vp_{t})$ is real analytic with respect to the subatlas of $g_{t_0}$-normal coordinate systems for any $t_0 \in (0,T_0]$. 

The key fact is the following from~\cite[\S 3]{KotschwarWangIsometries}: if a fixed metric $g$ is analytic in both of the overlapping smooth charts $(U, \phi)$ and $(V, \psi)$, then these two charts belong to the same analytic structure. The reason is that the transition function between the two charts is an isometry on $\bR^n$ of analytic metrics and is therefore real analytic itself. Since the transition function is real analytic, the two charts belong to the same analytic structure.

Fix $x_0 \in M$. By~\cite[Remark 6.4]{LotayWeiRealAnalyticity}, if $(\cD, \vp_t)$ is a smooth solution to Laplacian flow in an open set $\cD \subset M$ for $t \in [0,T_0]$, then $\sup_{\cD \times [0,T_0]}|\Rm|_{g_t}(x,t) < \infty$. Let $U$ be a normal coordinate chart for $g_{t_0}$, for $t_0 \in (0,T_0)$. By Lemma \ref{lemma Kotschwar Wang analyticity claim}, $g_t$ is analytic with respect to $U$-coordinates in $V \subset U$ for $t \in [t_0, t_0 + \eps]$. For $t_0$ close enough to $T_0$, we have that $\eps = T_0 - t_0$ so that $g_t$ is analytic with respect to $U$-coordinates in $V\subset U$ for $t \in [t_0, T_0]$.

For $t \in [t_0, t_0 + \eps]$, let $W_t$ be a $g_t$-normal coordinate chart centered at $x_0$. Since $g_t$ is analytic in $W_t$-coordinates and in $U = W_{t_0}$-coordinates, which overlap in a neighborhood of $x_0$, the chart $W_t$ belongs to the same real analytic structure as $U$. This argument applies for any choice of $t_0 \in (0,T_0)$. By the connectedness of $(0, T_0]$, for each $t \in (0,T_0]$, the chart $W_t$ belongs to the same real analytic structure as the chart $U = W_{t_0}$ for any given $t_0 \in (0, T_0]$. Moreover, both $g_t$ and $\vp_t$ are real analytic in the chart $U$ for any given $t_0 \in (0, T_0]$. This argument holds for any $x_0 \in M$, so the real analytic structure on $M$ associated to the subatlas of $g_t$-normal coordinate charts is the same for each $t\in (0, T_0]$. The conclusion of the proposition follows.
\end{proof}
The next corollary follows just as in~\cite[Corollary 3.2]{KotschwarWangIsometries}.
\begin{cor}\label{corollary common analytic structure shrinker}
Let $(\cC_r^{\Sg}, \vp, f)$ be a gradient Laplacian shrinker which is dynamically asymptotic to a closed $G_2$-cone $(\cC_r^{\Sg}, \vp_C)$. Then, $(\cC_r^{\Sg}, \vp)$ and $(\cC_r^{\Sg}, \vp_C)$ are real analytic relative to the same analytic structure. 
\end{cor}
\begin{proof}
By the definition of dynamically asymptotic, we may apply Proposition \ref{proposition same analytic structure} to conclude this corollary. 
\end{proof}

\subsection{Proof of Theorem \ref{theorem isometries dynamically asymptotic}}

From Proposition \ref{proposition backward flow reparametrization}, we let $\vp_\tau = \tau^{3/2} \Psi_{\tau}^* \vp$ be the family satisfying~\eqref{equation backwards Laplacian flow}. Let $(\cC_{r}^{\Sg}, \vp, f)$ be a gradient Laplacian shrinker which is dynamically asymptotic to a closed $G_2$-cone $(\cC_{r}^{\Sg}, \vp_C)$. 

\begin{lem}\label{lemma flat or preserved potential function}
Let $F \in \Aut(\cC_{r}^{\Sg}, \vp_\tau)$ for some $\tau \in (0, 1]$. Then, either $g_\tau$ is flat or $F^*(\na f) = \na f$ on $\cC_{r}^{\Sg}$.
\end{lem}
\begin{proof}
To prove this lemma, we follow the framework given in~\cite[Lemma 4.1]{KotschwarWangIsometries}. The proof of~\cite[Lemma 4.1]{KotschwarWangIsometries} depends on~\cite[Proposition 2.1]{KotschwarWangIsometries}. Our analogue of that proposition is Proposition \ref{proposition backward flow reparametrization}. The proof of this lemma works exactly the same as in ~\cite{KotschwarWangIsometries} using Proposition \ref{proposition backward flow reparametrization} and the fact that $F \in \Isom(\cC_{r}^{\Sg}, g_\tau)$ if $F \in \Aut(\cC_{r}^{\Sg}, \vp_\tau)$, modulo two minor changes. To find estimates on the time-dependent potential function $f_\tau$ as in~\cite[Lemma 4.1]{KotschwarWangIsometries}, we simply treat the $T \ast T$ terms in~\eqref{equation derivative identities on f 2} as decay terms at infinity using~\eqref{equation time dependent curvature decay} and the fact that $R = -|T|^2$. By the argument in~\cite[Lemma 4.1]{KotschwarWangIsometries}, we find $r_1> r$ such that $g(\tau)$ is flat on $\cC_{r_1}^{\Sg}$. By \cite[Corollary 1.2]{LotayWeiRealAnalyticity}, Laplacian solitons are real analytic and in particular, the associated metrics are also real analytic. This implies that $g_\tau$ is flat on $\cC_{r}^{\Sg}$.
\end{proof}
\begin{rmk}\label{remark on flat or preserved potential function}
The proof in \cite[Lemma 4.1]{KotschwarWangIsometries} shows that $f \circ F = f+ C$ for some constant $C$. This is equivalent to $F^*(\na f) = \na f$ using the fact that $F$ is an isometry.  
\end{rmk}

\begin{lem}\label{lemma extension of automorphism from cone}
Suppose that $(\cC_r^{\Sg}, \vp, f)$ is a gradient Laplacian shrinker which is dynamically asymptotic to a closed $G_2$-cone $(\cC_r^{\Sg}, \vp_C)$. Then, for any $F \in \Aut(\cC_r^{\Sg}, \vp_C)$, there is $r_1 \geq r$ such that $F|_{\cC_{r_1}^{\Sg}} \in \Aut(\cC_{r_1}^{\Sg}, \vp)$. 
\end{lem}
\begin{proof}
As in~\cite[Proposition 4.2]{KotschwarWangIsometries}, we have that $(\cC_r^{\Sg}, F^*\vp, F^*f)$ is dynamically asymptotic to $(\cC_r^{\Sg}, \vp_C)$. Applying the backward uniqueness found in the proof of Theorem \ref{theorem main}, $\vp = F^* \vp$ on $\cC^\Sg_a \times [0,T_0]$ for some $a>r$ and $T_0 \in (0,1)$. As in~\cite[Proposition 4.2]{KotschwarWangIsometries}, we find that $\vp_1 = F^*\vp_1$ on $\cC^\Sg_b$ for some $b\geq a$. This concludes the lemma.
\end{proof}

\begin{lem}\label{lemma extending local isometry using real analyticity}
Let $T_1 \in [0,1]$, and let $r_2 \geq r_1 \geq r$. Suppose $F: \cC_{r_1}^{\Sg} \to \cC_{r}^{\Sg}$ is an injective local diffeomorphism such that $F^* \vp_{T_1} = \vp_{T_1}$. If $F^*\vp_{T_2} = \vp_{T_2}$ on $\cC_{r_2}^{\Sg}$ for some $T_2 \in [0,1]$, then $F^* \vp_{T_2} = \vp_{T_2}$ on $\cC_{r_1}^{\Sg}$.
\end{lem}
\begin{proof}
This follows much the same idea as in \cite[Proposition 4.3]{KotschwarWangIsometries}. 

By Corollary \ref{corollary common analytic structure shrinker}, $\vp_{T_1}$ and $\vp_{T_2}$ and the metrics $g_{T_1}$ and $g_{T_2}$ are all real analytic with respect to the same analytic structure $\cA$.

Since $F^* \vp_{T_1} = \vp_{T_1}$, $F$ preserves the metric $g_{T_1}$ and thus is a local isometry of $g_{T_1}$. This implies that $F$ is real analytic on $\cC^\Sg_{r_1}$ with respect to $\cA$, and so 
$$h:= F^* \vp_{T_2} - \vp_{T_2}$$
is also real analytic on $\cC_{r_1}^{\Sg}$. By assumption, $h \equiv 0$ on $\cC_{r_2}^{\Sg}$, so the real analyticity of $h$ implies that $h\equiv 0$ on $\cC_{r_1}^{\Sg}$.  
\end{proof}
Recall that $\vp_\tau = \tau^{3/2}\Psi_\tau^* \vp$, where $\Psi_{\tau}$ is the map defined in Proposition \ref{proposition backward flow reparametrization}. Note that $\vp_1 = \vp$.
 
\begin{lem}\label{lemma isometry at time T is isometry at time 1}
Let $T_0 \in (0,1]$, and let $r_1 \geq r$. Suppose $F: \cC_{r_1}^{\Sg} \to \cC_{r_1}^{\Sg}$ is an injective local diffeomorphism such that $F^*\vp_{T_0} = \vp_{T_0}$. Then, $F^* \vp = \vp$ on $\cC_{r_1}^{\Sg}$.
\end{lem}
\begin{proof}
This follows much the same as in \cite[Proposition 4.4]{KotschwarWangIsometries}. 

If $\vp_{T_0}$ is torsion-free on $\cC_{r_1}^{\Sg}$, then $\vp$ is torsion-free on $\Psi_T(\cC_{r_1}^{\Sg})$. By the local real analyticity of Laplacian solitons~\cite[Corollary 1.2]{LotayWeiRealAnalyticity}, $\vp$ must be torsion-free on all of $\cC_{r}^{\Sg}$. This then implies that the Laplacian flow $\tau \mapsto \vp_\tau$ constructed in Proposition \ref{proposition backward flow reparametrization} is a flow of torsion-free $G_2$-structures and hence a fixed point for the Laplacian flow. Thus, $\vp = \vp_{T_0}$ on $\cC_{r_1}^{\Sg}$, and the lemma follows in this case.

Now, we may suppose that $\vp_{T_0}$ is not torsion-free. Since $R = -|T|^2$, this implies that $g_{T_0}$ is not flat. As in~\cite[Proposition 4.4]{KotschwarWangIsometries}, this allows us to apply Lemma \ref{lemma flat or preserved potential function} as well as Remark \ref{remark on flat or preserved potential function} and~\eqref{equation C0 estimates on f} to find that there is $b\geq r_1$ depending on $T_0$, $r_1$, and $N_0$ from~\eqref{equation C0 estimates on f}, such that 
$$F \circ \Psi_{\tau} = \Psi_{\tau} \circ F$$
on $\cC_{b}^{\Sg}$ for all $\tau \in (0,1]$. It follows immediately that $F^* \vp = \vp$ on $\cC^\Sg_{b/\sqrt{T}}$. An application of Lemma \ref{lemma extending local isometry using real analyticity} shows that $F^* \vp = \vp$ on $\cC_{r_1}^{\Sg}$.\end{proof}

\begin{proof}[Proof of Theorem \ref{theorem isometries dynamically asymptotic}]
This follows with the same idea as~\cite[Proof of Theorem 1.2]{KotschwarWangIsometries}.

Suppose that $F \in \Aut(\cC_{r}^{\Sg}, \vp_C)$. Lemma \ref{lemma extension of automorphism from cone} gives $r_1 \geq r$ such that $F|_{\cC_{r_1}^{\Sg}} \in \Aut(\cC_{r_1}^{\Sg}, \vp)$. Then, Lemma \ref{lemma extending local isometry using real analyticity} implies that $F \in \Aut(\cC_{r}^{\Sg}, \vp)$. Thus, 
$$\Aut(\cC_r^{\Sg}, \vp_C) \subseteq \Aut(\cC_r^{\Sg}, \vp).$$

Now, suppose that $F \in \Aut(\cC_r^{\Sg}, \vp)$. If $\vp$ is torsion-free, then by construction, $\vp_\tau$ is torsion-free and thus $\tau \mapsto \vp_\tau$ is a fixed point of the Laplacian flow. This implies that $\vp = \vp_0 = \vp_C$, and the Theorem follows in this case. 

We may assume that $\vp$ is not torsion-free. By Lemma \ref{lemma flat or preserved potential function}, $F^* \na f = \na f$. Now, for the local diffeomorphisms $\Psi_{\tau}: \cC_r^{\Sg} \to \cC_r^{\Sg}$ defined by Proposition \ref{proposition backward flow reparametrization}, both $\Psi_{\tau}$ and $F^{-1}\circ \Psi_{\tau}\circ F$ satisfy equation \eqref{equation satisfied by reparametrizations} since $F$ is an isometry. This implies that $\Psi_{\tau} = F^{-1}\circ \Psi_{\tau}\circ F$, and so 
$$F \circ \Psi_{\tau} = \Psi_{\tau} \circ F$$ 
for all $\tau \in (0,1]$. 

Thus, for any $\tau \in (0, 1]$,
$$F^* \vp_\tau = \tau^{3/2} F^* \Psi^*_{\tau} \vp = \tau^{3/2}\Psi^*_{\tau} F^* \vp = \tau^{3/2}\Psi^*_{\tau} \vp = \vp_\tau.$$
Taking the limit as $\tau \to 0$, $F^* \vp_C = \vp_C$. This shows that $F \in \Aut(\cC_r^{\Sg}, \vp_C)$, and so 
$$\Aut(\cC_r^{\Sg}, \vp)\subseteq \Aut(\cC_r^{\Sg}, \vp_C),$$
thus proving the theorem.\end{proof}

\subsection{Proof of Corollary \ref{corollary complete isometries}}
The idea of this proof follows~\cite[Theorem 1.2]{KotschwarWangIsometries}.

By Theorem \ref{theorem isometries}, there is $r_1 > 0$, an end $W\subset V$ of $(M, \vp)$, and a diffeomorphism $F: \cC_{r_1}^{\Sg} \to W$ such that $\Aut(W, \vp) \cong \Aut(C, \vp_C)$. We have that $W$ is diffeomorphic to $\cC_{r_1}^{\Sg}$, but by assumption $V$ is diffeomorphic to $C_r^{\Sg}$ for $r\leq r_1$. Thus, the inclusion map $\io: W \hookrightarrow V$ induces an isomorphism $\pi_1(W, x_0) \cong \pi_1(V, x_0)$. By assumption, the map $\pi_1(V, x_0)\rightarrow \pi_1(M, x_0)$ induced by inclusion is surjective, so the induced inclusion map $\pi_1(W, x_0) \rightarrow \pi_1(M, x_0)$ is surjective. Then, by~\cite[Theorem 5.2, Proof of Theorem 1.2]{KotschwarWangIsometries}, we have that any isometry $\phi \in \Isom(W, g_{\vp}|_W)$ can be extended to a unique isometry $\bar{\phi} \in \Isom(M, g_{\vp})$ such that this correspondence gives an isomorphism 
$$\Isom(W, g_{\vp}|_W) \cong \Isom(M, g_{\vp}).$$

Now, if $F \in \Aut(W, \vp)$, then $F$ is by definition an isometry. Thus, $F \in \Isom(W, g_{\vp}|_W)$, and this naturally extends to $\bar{F} \in \Isom(M, g_{\vp})$. Since $\bar{F}$ is an isometry, it is real analytic. Thus, by the real analyticity of $\vp$, 
$$h:= \bar{F}^* \vp - \vp$$
is also real analytic and defined everywhere on $M$. By assumption, $h\equiv 0$ on $W$. Thus, by the completeness and connectedness of $M$, $h\equiv 0$ on $M$. This implies that $\bar{F} \in \Aut(M, \vp)$. Thus, $\Aut(W, \vp)$ embeds into $\Aut(M, \vp)$, and by Theorem \ref{theorem isometries}, $\Aut(C, \vp_C)$ embeds into $\Aut(M, \vp)$.

\appendix

\section{Backward Uniqueness for Subsolutions of Parabolic Equations with Unbounded Coefficients}\label{section Appendix}

With the technique used in the proof of Proposition \ref{proposition exponential decay}, we can prove the following improvement to the backward uniqueness result of \cite{EscauriazaSereginSverak}.

\begin{thm}\label{theorem scalar backwards uniqueness}
Let $Q_{R,T}$ be the parabolic domain $(\bR^n \setminus B_R) \times [0, T]$, and let $u: Q_{R,T} \rightarrow \bR$ satisfy
\begin{equation}\label{equation subsolution and growth conditions} 
    |\De u + \dd_t u| \le M|\na u| + (N|x|^{\de} + M)|u|\;\; \text{ and }\;\; |u(x,t)| \le Me^{M|x|^2},
\end{equation}
for constants $M, T, R, C, D >0$ and $0< \de < 1/2$. Then, if $u(x,0) \equiv 0$ in $\bR^n \setminus B_R$, $u$ vanishes identically in $Q_{R,T}$. 
\end{thm}

With the rescaling $u = \tfrac{1}{M}\tilde{u}(\om x, \om^2 t)$ for some small parameter $\om$, we may assume that 
\begin{equation}\label{equation rescaled conditions for backwards uniqueness}
    |\De u + \dd_t u| \le \ve(|\na u| + |u|) + N|x|^{\de}|u|\;\; \text{ and }\;\; |u(x,t)| \le e^{\ve|x|^2}.
\end{equation}
We will not use the Carleman inequality (1.4) in \cite{EscauriazaSereginSverak}, as it discards the crucial coefficient of $\sqrt{\al}$ on the left-hand side. We instead consider the inequality on page 4 of \cite{kotschwarwang2015}, which we restate here. 
\begin{lem}
Fix $\ga \ge 1$ and let $\sg_a(t) = (t+a)e^{-\frac{(t+a)}{3}}$. Then, for all $\rho \ge 1$, $a \in (0,1)$, $\al \ge \al_0(\ga, n) \ge 0$ and $u \in C^\infty_c( \{|x| \ge \ga \rho\} \times [0,1) )$ such that $u(x,0) \equiv 0$, the following inequality holds.
\begin{multline}\label{equation L2 second carleman scalar}
    \sqrt{\al}||\sg_a^{-\al - \frac{1}{2}}e^{- \frac{(|x| - \rho)^2}{8(t+a)}} u||_{L^2(\bR^n \times (0,1))} + ||\sg_a^{-\al}e^{- \frac{(|x| - \rho)^2}{8(t+a)}} \na u||_{L^2(\bR^n \times (0,1))} \\
    \le C(\ga, n)||\sg_a^{-\al}e^{- \frac{(|x| - \rho)^2}{8(t+a)}} (\dd_t + \De) u||_{L^2(\bR^n \times (0,1))}
\end{multline}
\end{lem}
\begin{proof}
This inequality is (modulo constants) the inequality (5.23) in \cite{kotschwarwang2015} with respect to the static Ricci flow of the Euclidean metric on $\bR^n$ (where $r_c^2(x) = h(x,t) = |x|$) applied to the $0$-order tensor bundle $\cZ = C^\infty( \bR^n \times (0,1))$.
\end{proof}

We prove the $L^2$ exponential decay property analogous to Proposition \ref{proposition exponential decay}. As before, we have $A(r_1, r_2) := B_{r_2} \setminus B_{r_1}$ where $0 < r_1 < r_2$.

\begin{prop}\label{proposition exponential decay scalar linear}
There exist dimensional constants $s_0$, $C$, and $r_{0}$, with $s_0 \in (0, 1]$, such that
\begin{equation}\label{equation exponential decay scalar}
|||u| + |\nabla u| ||_{L^2(A((1-\sqrt{s})\rho,(1+\sqrt{s})\rho)\times [0,s])} \leq C e^{-\frac{\rho^2}{Cs}}
\end{equation}
for any $s \in (0, s_0]$ and $\rho > 12r_0$. 
\end{prop}

\begin{proof}
The proof is a simplified version of the proof of Proposition \ref{proposition exponential decay}. We choose $\ga = 1/12$ and set $r_0 = R$ initially. Let $\vartheta \in C^\infty(\bR, [0,1])$ be a temporal cutoff function such that $\vartheta \equiv 1$ for $\tau \le 1/6$ and $\vartheta \equiv 0$ for $\tau \ge 1/5$. Let $\psi_{\rho, \xi}$ be the spatial cutoff function given in Lemma \ref{lemma spatial cutoff function}. Define $u_{\rho, \xi} := \vartheta \psi_{\rho, \xi} u$. As before, in \eqref{equation L2 second carleman scalar} we replace the parameter $\al$ by an integer $k \ge \al_0$ and we restrict $a$ to the interval $(0,1/8)$. This yields
\begin{multline*}
    k^{\frac{1}{2}}||\sg_a^{-k - \frac{1}{2}}e^{- \frac{(|x| - \rho)^2}{8(t+a)}} u_{\rho, \xi}||_{L^2(A(\frac{\rho}{6}, 3\xi) \times [0,\frac{1}{5}])} + ||\sg_a^{-k}e^{- \frac{(|x| - \rho)^2}{8(t+a)}} \na u_{\rho, \xi}||_{L^2(A(\frac{\rho}{6}, 3\xi) \times [0,\frac{1}{5}])} \\
    \le C||\sg_a^{-k}e^{- \frac{(|x| - \rho)^2}{8(t+a)}} (\dd_t + \De) u_{\rho, \xi}||_{L^2(A(\frac{\rho}{6}, 3\xi) \times [0,\frac{1}{5}])}.
\end{multline*}
Now, set $\xi = 8\rho$. Choose $k_1$ depending on $\rho$ such that $4C^2(24\rho)^{2\de} \le k_1 \le 16C^2(24\rho)^{2\de}$. Since the maximum of $|x|^{\de}$ on $A(\rho/6, 3\xi) = A(\rho/6, 24\rho)$ is $(24\rho)^{\de}$, the following inequality holds for $k \ge k_1$.
\begin{equation}\label{equation absorption inequality} k^\frac{1}{2} - C(24\rho)^{\de} \ge \frac{k^\frac{1}{2}}{2}.
\end{equation}
We use \eqref{equation rescaled conditions for backwards uniqueness} to obtain
\begin{multline*}|(\dd_t +\De) u_{\rho, 8\rho}| \le \ve(|u_{\rho, 8\rho}| + |\na u_{\rho, 8\rho}|) + C|x|^\de|u_{\rho, 8\rho}| \\ + \psi_{\rho, 8\rho}|\vartheta'||u| + \vartheta(|u|(|\De \psi_{\rho, 8\rho}| + |\na \psi_{\rho, 8\rho}|) +2|\na \psi_{\rho, 8\rho}||\na u|).
\end{multline*}
For $\ve$ sufficiently small and $k \ge k_1$,
\begin{align*}
    &k^{\frac{1}{2}}||\sg_a^{-k -\frac{1}{2}} e^{- \frac{(|x| - \rho)^2}{8(t+a)}}u_{\rho, 8\rho}||_{L^2(A(\frac{\rho}{6}, 24\rho) \times [0,\frac{1}{5}])} + ||\sg_a^{-k}e^{- \frac{(|x| - \rho)^2}{8(t+a)}}\na u_{\rho, 8\rho}||_{L^2(A(\frac{\rho}{6}, 24\rho) \times [0,\frac{1}{5}])} \\
    &\le C(24\rho)^\de||\sg_a^{-k -\frac{1}{2}} e^{- \frac{(|x| - \rho)^2}{8(t+a)}}u_{\rho, 8\rho}||_{L^2(A(\frac{\rho}{6}, 24\rho) \times [0,\frac{1}{5}])} \\
    &\qquad + C||\sg_a^{-k-\frac{1}{2}}e^{- \frac{(|x| - \rho)^2}{8(t+a)}}u||_{L^2(A(\frac{\rho}{6}, 24\rho) \times [\frac{1}{6},\frac{1}{5}])} \\
    &\qquad \qquad  + C||\sg_a^{-k}e^{- \frac{(|x| - \rho)^2}{8(t+a)}}(|u| + |\na u|)||_{L^2(A(\frac{\rho}{6}, \frac{\rho}{3}) \times [0,\frac{1}{5}])} \\
    &\qquad \qquad  \qquad + C||\sg_a^{-k}e^{- \frac{(|x| - \rho)^2}{8(t+a)}}(|u| + |\na u|)||_{L^2(A(16\rho, 24\rho) \times [0,\frac{1}{5}])}
\end{align*}
By \eqref{equation absorption inequality}, we can absorb the first term on the right-hand side into the left-hand side to obtain
\begin{align*}
    &k^{\frac{1}{2}}||\sg_a^{-k -\frac{1}{2}} e^{- \frac{(|x| - \rho)^2}{8(t+a)}}u_{\rho, 8\rho}||_{L^2(A(\frac{\rho}{6}, 24\rho) \times [0,\frac{1}{5}])} + ||\sg_a^{-k}e^{- \frac{(|x| - \rho)^2}{8(t+a)}}\na u_{\rho, 8\rho}||_{L^2(A(\frac{\rho}{6}, 24\rho) \times [0,\frac{1}{5}])} \\
    &\le C||\sg_a^{-k-\frac{1}{2}}e^{- \frac{(|x| - \rho)^2}{8(t+a)}}u||_{L^2(A(\frac{\rho}{6}, 24\rho) \times [\frac{1}{6},\frac{1}{5}])} \\
    &\qquad \qquad  + C||\sg_a^{-k}e^{- \frac{(|x| - \rho)^2}{8(t+a)}}(|u| + |\na u|)||_{L^2(A(\frac{\rho}{6}, \frac{\rho}{3}) \times [0,\frac{1}{5}])} \\
    &\qquad \qquad  \qquad + C||\sg_a^{-k}e^{- \frac{(|x| - \rho)^2}{8(t+a)}}(|u| + |\na u|)||_{L^2(A(16\rho, 24\rho) \times [0,\frac{1}{5}])}
\end{align*}
Since $\sigma_a \le 1$ and $1 \le e^{(\tau + a)/3} \le e$, we can further estimate:
\begin{align*}
    &||(t+a)^{-k} e^{- \frac{(|x| - \rho)^2}{8(t+a)}}u_{\rho, 8\rho}||_{L^2(A(\frac{\rho}{6}, 24\rho) \times [0,\frac{1}{5}])} + ||\sg_a^{-k}e^{- \frac{(|x| - \rho)^2}{8(t+a)}}\na u_{\rho, 8\rho}||_{L^2(A(\frac{\rho}{6}, 24\rho) \times [0,\frac{1}{5}])} \\
    &\le C^k||(t+a)^{-k}e^{- \frac{(|x| - \rho)^2}{8(t+a)}}u||_{L^2(A(\frac{\rho}{6}, 24\rho) \times [\frac{1}{6},\frac{1}{5}])} \\
    &\qquad \qquad  + C^k||(t+a)^{-k}e^{- \frac{(|x| - \rho)^2}{8(t+a)}}(|u| + |\na u|)||_{L^2(A(\frac{\rho}{6}, \frac{\rho}{3}) \times [0,\frac{1}{5}])} \\
    &\qquad \qquad  \qquad + C^k||(t+a)^{-k}e^{- \frac{(|x| - \rho)^2}{8(t+a)}}(|u| + |\na u|)||_{L^2(A(16\rho, 24\rho) \times [0,\frac{1}{5}])}
\end{align*}
The first term on the right-hand side of the inequality can be estimated by the bounds \eqref{equation tensor fields are bounded} and $(\tau + a) \ge 1/6$.
\begin{equation*}
    ||(t+a)^{-k}e^{- \frac{(|x| - \rho)^2}{8(t+a)}}u||_{L^2(A(\frac{\rho}{6}, 24\rho) \times [\frac{1}{6},\frac{1}{5}])} \le 6^k || C ||_{L^2(A(\frac{\rho}{6}, 24\rho) \times [\frac{1}{6},\frac{1}{5}])} \le C^k\rho^\frac{n}{2}
\end{equation*}
for all $k\geq k_1$ and $a\in (0, 1/8)$. 

On the domains $A(\frac{\rho}{6},\frac{\rho}{3})\times [0, \frac{1}{5}]$ and $A(16\rho, 24\rho)\times [0, \frac{1}{5}]$, we have that
$$e^{-(r-\rho)^2/(8(\tau+a))} \leq e^{-\rho^2/(20(\tau+a))}.$$
We compute using Stirling's formula that
\[
 \max_{s>0} s^{-k}e^{-\rho^2/(20 s)}  = \rho^{-2k}(20 k)^ke^{-k} \leq \rho^{-2k}C^k k!.
\]
This observation yields the following bound.
\begin{multline*}
     ||(t+a)^{-k}e^{- \frac{(|x| - \rho)^2}{8(t+a)}}(|u| + |\na u|)||_{L^2(A(\frac{\rho}{6}, \frac{\rho}{3}) \times [0,\frac{1}{5}])} \\+ ||(t+a)^{-k}e^{- \frac{(|x| - \rho)^2}{8(t+a)}}(|u| + |\na u|)||_{L^2(A(16\rho, 24\rho) \times [0,\frac{1}{5}])} \le C^k\rho^{-2k}k!\rho^{\frac{n}{2}},
\end{multline*}
for $k \ge k_1$ and $a \in (0, 1/8)$. Let $l$ be such that $k = k_1 + l$. Plugging all this information back into the primary inequality, we find that 
\begin{multline*}
    ||(t+a)^{-k} e^{- \frac{(|x| - \rho)^2}{8(t+a)}}|u_{\rho, 8\rho}| + |\na u_{\rho, 8\rho}|||_{L^2(A(\frac{\rho}{6}, 24\rho) \times [0,\frac{1}{5}])} \\ \le \rho^\frac{n}{2}C^{k_1 + l}(1 + \rho^{-2(k_1 + l)}(k_1 + l)!)
\end{multline*}
for all $l\geq 0$ and all $a\in (0, 1/8)$. Noting that $(l+k_1)! \le C^{k_1 + l}(l!)(k_1!)$, 
\begin{align*}
     & ||(t+a)^{-k} e^{- \frac{(|x| - \rho)^2}{8(t+a)}}|u_{\rho, 8\rho}| + |\na u_{\rho, 8\rho}|||_{L^2(A(\frac{\rho}{6}, 24\rho) \times [0,\frac{1}{5}])} \\
     &\qquad \le \rho^\frac{n}{2}C^{k_1 + l} (1 + \rho^{-2l}C^{k_1+l} (k_1!)(l!))\\
     &\qquad \le \rho^\frac{n}{2}(k_1!)^2 (C^l + C^{l}\rho^{-2l}l!).
\end{align*}
\begin{rmk}
Note that $C^{k_1}$ can be bounded by $k_1!$ for sufficiently large $\rho$. Thus, we increase $r_0 > 0$ so that this may be true.
\end{rmk}
We multiply both sides by $\rho^{2l}/((2C)^ll!)$ and sum over all $l\geq 0$ yielding
\begin{align*}
     &||(|u_{\rho, 8\rho}| + |\na u_{\rho, 8\rho}|)e^{\frac{\rho^2}{C(\tau+a)}-\frac{(r -\rho)^2}{8(\tau + a)}}||_{L^2(A(\frac{\rho}{6}, 24\rho) \times [0,\frac{1}{5}])} \\
     &\qquad \le C(k_1!)^2\rho^\frac{7}{2}\bigg(1 + e^\frac{\rho^2}{2} \bigg).
\end{align*}
Restrict the time interval from $[0, 1/5]$  to $[0, a]$.
\begin{align*}
     &||(|u_{\rho, 8\rho}| + |\na u_{\rho, 8\rho}|)e^{-\frac{(r -\rho)^2}{8(\tau + a)}}||_{L^2(A(\frac{\rho}{6}, 24\rho) \times [0,a])} \\
     &\qquad \le C(k_1!)^2\rho^\frac{7}{2}\bigg(1 + e^\frac{\rho^2}{2} \bigg)e^{-\frac{\rho^2}{2Ca}}.
\end{align*}
For sufficiently small $a > 0$ we obtain the following upper bound
\[||(|u_{\rho, 8\rho}| + |\na u_{\rho, 8\rho}|)e^{-\frac{(r -\rho)^2}{8(\tau + a)}}||_{L^2(A(\frac{\rho}{6}, 24\rho) \times [0,a])} \le k_1^{2k_1}e^{-\frac{\rho^2}{Ca}}.
\]
Recall that $k_1 \le 16C^2(24\rho)^{2\de}$. After increasing $C > 0$, we have
\begin{align*}
    k_1^{2k_1} &\le (C\rho^{2\de})^{C\rho^{2\de}}\\
    &= \exp(C\rho^{2\de}(\log(C) + 2\de\log(\rho)))\\
    &\le \exp(C\rho^{2\de'}),
\end{align*}
where $0 < \de < \de' < 1$. This grows more slowly than quadratic exponential. By possibly increasing $r_{0}$, we finally obtain
\[||(|u_{\rho, 8\rho}| + |\na u_{\rho, 8\rho}|)e^{-\frac{(r -\rho)^2}{8(\tau + a)}}||_{L^2(A(\frac{\rho}{6}, 24\rho) \times [0,a])} \le Ce^{-\frac{\rho^2}{Ca}}.
\]
Observe that $e^{-|r -\rho|^2/(8(\tau+a))} \geq C^{-1}$ on the domain $A((1-\sqrt{a})\rho, (1+\sqrt{a})\rho)\times [0, \delta a]$. Furthermore, on this set, $\psi_{\rho, 8\rho} \equiv 1$ and $\vp \equiv 1$. Thus, the function $u_{\rho, 8\rho}$ agrees with its unmodified counterpart on this restricted domain and we obtain the estimate 
\begin{align}
\begin{split}
    \||u|+|\nabla u| \|_{L^2(A((1-\sqrt{s})\rho,(1+\sqrt{s})\rho)\times[0, s])}
\leq Ce^{-\frac{\rho^2}{Cs}}
\end{split}  
\end{align}
for $s\in [0, 1/C]$ and arbitrary $\rho > 12r_{0}$, which was the desired exponential decay. 
\end{proof}

\begin{rmk}
Note that the proof of exponential decay only requires that $\de \in (0,1)$, not the more restrictive condition $\de \in (0, 1/2)$ that is required for the proof of vanishing.
\end{rmk}

In the proof of vanishing, we will use the following Carleman inequality with weights growing quadratic exponentially.

\begin{prop}\label{proposition first carleman estimate scalar}
For each $\eta \in (0,1)$, there exists a constant $r_1 = r_1(\eta) \ge R$ such that for all $\al \ge 1$ and all $u \in C^\infty (Q_{R,T_0})$ satisfying that $u(\cdot, 0) \equiv 0$ and that $u(\cdot, t)$ is compactly supported in $\bR^n \setminus B_{r_1}$ for each $t \in [0,T_0]$, we have the estimate
\begin{multline}\label{equation first carleman estimate scalar}
\al \Big\|u e^{\frac{\al(T_0 - t)|x|^{2-\eta} + |x|^2}{2}}\Big\|^2_{L^2(Q_{r_1,T_0})} + \Big\||\na u| e^{\frac{\al(T_0 - t)|x|^{2-\eta} + |x|^2}{2}}\Big\|^2_{L^2(Q_{r_1,T_0})} \\
\le \frac{1}{2}\Big\|(\dd_t u + \De u) e^{\frac{\al(T_0 - t)|x|^{2-\eta} + |x|^2}{2}}\Big\|^2_{L^2(Q_{r_1,T_0})} \\ 
+ \Big\| |\na u| e^{\frac{\al(T_0 - t)|x|^{2-\eta} + |x|^2}{2}}\Big\|^2_{L^2((\bR^n \setminus B_{r_1}) \times \{T_0\})}.
\end{multline}
\end{prop}

\begin{proof}
This inequality is (modulo constants) the inequality in \cite[Proposition 4.9]{kotschwarwang2015} with respect to the static Ricci flow of the Euclidean metric on $\bR^n$ (where $r_c^2(x) = h(x,t) = |x|$) applied to the $0$-order tensor bundle $\cZ = C^\infty( \bR^n \times (0,1))$.
\end{proof}

\begin{proof}[Proof of Theorem \ref{theorem scalar backwards uniqueness}] We now prove that the function $u$ must vanish uniformly. We will follow the method of \cite{kotschwarwang2015}, modified by a similar technique to that in the proof of Proposition \ref{proposition exponential decay scalar linear}. We first show that the space-time $W^{1,2}$-norm, weighted by $e^{8|x|^2}$, is finite. By Proposition \ref{proposition exponential decay scalar linear}, there exists a constant $s_1(n)$ such that 
\[|||u| + |\nabla u| ||_{L^2(A(\rho - \sqrt{s}, \rho + \sqrt{s})\times [0,s])} \leq e^{-16\rho^2}
\]
for $s \in (0, s_1]$ and $\rho > 12r_0$. Now, choosing $r'' > 13 r_0$, 
\begin{align*}
    &\|(|u| + |\nabla u|) e^{4|x|^2}\|_{L^2(B_{r''} \setminus B_R \times [0,s_1])} \\
    &\qquad \le  \|(|u| + |\nabla u|) e^{4|x|^2}\|_{L^2(B_{r''} \setminus B_{13r_0} \times [0,s_1])} +  \|(|u| + |\nabla u|) e^{4|x|^2}\|_{L^2(B_{13r_0} \setminus B_R \times [0,s_1])}\\
    &\qquad \le \sum_{i=0}^{k'} \|(|u| + |\nabla u|) e^{4|x|^2}\|_{L^2(A(13r_0 + 2i\sqrt{s_1}, 13r_0 + 2(i+1)\sqrt{s_1})\times [0,s_1])}  \\
    &\qquad \qquad + C\|u\|_{W^{1,2}(B_{13r_0}\setminus B_{R} \times [0,s_1])}\\
    &\qquad \le \sum_{i=0}^{k'} C(13r_0 + 2(i+1)\sqrt{s_1})^{\frac{n}{2}} e^{-12((13r_0)^2 +i^2s_1} + C,
\end{align*}
where we have chosen $k'$ to be the least integer greater than $\tfrac{1}{2\sqrt{s_1}}(r'' - 13r_0)$. By sending $r''$ to infinity, we find that 
\begin{equation}\label{equation bounded L2 norm with exponential weight}\|(|u| + |\nabla u|) e^{4|x|^2}\|_{L^2(Q_{R,s_1})} \le C,
\end{equation}
where $C$ is a dimensional constant.

Let $\psi_{\rho, \xi}$ be the spatial cutoff function given in Lemma \ref{lemma spatial cutoff function} such that $\psi_{\rho, \xi} \equiv 1$ on $[2\rho, \xi]$ and $\equiv 0$ on $(-\infty, \rho] \cup [2\xi, \infty)$. Define the compactly supported function $u_{\rho, \xi} := \psi_{\rho, \xi} u$ and choose $\rho > 13r_0$ sufficiently large. Now apply the inequality \eqref{equation first carleman estimate scalar}.
\begin{multline*}
\al^\frac{1}{2} \Big\|u_{\rho,\xi} e^{\frac{\al(T_0 - t)|x|^{2-\eta} + |x|^2}{2}}\Big\|_{L^2(A(\rho, 2\xi) \times [0, T_0])} + \Big\||\na u_{\rho,\xi}| e^{\frac{\al(T_0 - t)|x|^{2-\eta} + |x|^2}{2}}\Big\|_{L^2(A(\rho, 2\xi) \times [0, T_0])} \\
\le 2\Big\|(\dd_t u_{\rho,\xi} + \De u_{\rho,\xi}) e^{\frac{\al(T_0 - t)|x|^{2-\eta} + |x|^2}{2}}\Big\|_{L^2(A(\rho, 2\xi) \times [0, T_0])} \\ 
+ \Big\| |\na u_{\rho,\xi}| e^{\frac{\al(T_0 - t)|x|^{2-\eta} + |x|^2}{2}}\Big\|_{L^2(A(\rho, 2\xi) \times \{T_0\})}
\end{multline*}
Observe that
\begin{multline*}
    |(\dd_t + \De) u_{\rho,\xi}| \le C|x|^\de|u_{\rho, \xi}| + \ve |\na u_{\rho, \xi}| + |u|(|\na \psi_{\rho, \xi}| + |\De \psi_{\rho, \xi}|) + 2 |\na \psi_{\rho, \xi}||\na u|.
\end{multline*}
Combining these two inequalities yields
\begin{align*}
    &\al^\frac{1}{2} \Big\|u_{\rho,\xi} e^{\frac{\al(T_0 - t)|x|^{2-\eta} + |x|^2}{2}}\Big\|_{L^2(A(\rho, 2\xi) \times [0, T_0])} + \Big\||\na u_{\rho,\xi}| e^{\frac{\al(T_0 - t)|x|^{2-\eta} + |x|^2}{2}}\Big\|_{L^2(A(\rho, 2\xi) \times [0, T_0])} \\
    &\qquad \le C \Big\||x|^\de u_{\rho,\xi} e^{\frac{\al(T_0 - t)|x|^{2-\eta} + |x|^2}{2}}\Big\|_{L^2(A(\rho, 2\xi) \times [0, T_0])} \\
    & \qquad \qquad + C\Big\|(|u| + |\na u|) e^{\frac{\al(T_0 - t)|x|^{2-\eta} + |x|^2}{2}}\Big\|_{L^2(A(\rho, 2\rho) \times [0, T_0])} \\
    & \qquad \qquad \qquad + C\Big\|(|u| + |\na u|) e^{\frac{\al(T_0 - t)|x|^{2-\eta} + |x|^2}{2}}\Big\|_{L^2(A(\xi, 2\xi) \times [0, T_0])}\\
    & \qquad + C\Big\||u| e^{\frac{\al(T_0 - t)|x|^{2-\eta} + |x|^2}{2}}\Big\|_{L^2(A(\rho, 2\rho) \times \{ T_0\})} + C\Big\||u| e^{\frac{\al(T_0 - t)|x|^{2-\eta} + |x|^2}{2}}\Big\|_{L^2(A(\xi, 2\xi) \times \{ T_0\})}\\
    & \qquad \qquad \qquad +  C\Big\||\na u| e^{\frac{\al(T_0 - t)|x|^{2-\eta} + |x|^2}{2}}\Big\|_{L^2(A(\rho, 2\xi) \times \{ T_0\})}.
\end{align*}
We now choose $\xi$ such that $C(2\xi)^\de \le \al^\frac{1}{2}/2$. We may increase $C$ by an amount depending on $\delta$ so that $\xi = C^{-1}\al^\frac{1}{2\de}$ satisfies this relation. Absorbing the first term on the right-hand side yields the simplified inequality
\begin{align*}
    &\al^\frac{1}{2} \Big\|u e^{\frac{\al(T_0 - t)|x|^{2-\eta} + |x|^2}{2}}\Big\|_{L^2(A(2\rho, C^{-1}\al^\frac{1}{2\de}) \times [0, T_0])} + \Big\||\na u| e^{\frac{\al(T_0 - t)|x|^{2-\eta} + |x|^2}{2}}\Big\|_{L^2(A(2\rho, C^{-1}\al^\frac{1}{2\de}) \times [0, T_0])} \\
    &\qquad \le  C\Big\|(|u| + |\na u|) e^{\frac{\al(T_0 - t)|x|^{2-\eta} + |x|^2}{2}}\Big\|_{L^2(A(\rho, 2\rho) \times [0, T_0])} \\
    & \qquad \qquad + C\Big\|(|u| + |\na u|) e^{\frac{\al(T_0 - t)|x|^{2-\eta} + |x|^2}{2}}\Big\|_{L^2(A(C^{-1}\al^\frac{1}{2\de}, 2C^{-1}\al^\frac{1}{2\de}) \times [0, T_0])}\\
    & \qquad \qquad \qquad+ C\Big\||u| e^{\frac{\al(T_0 - t)|x|^{2-\eta} + |x|^2}{2}}\Big\|_{L^2(A(\rho, 2\rho) \times \{ T_0\})} \\
    & \qquad \qquad \qquad + C\Big\||u| e^{\frac{\al(T_0 - t)|x|^{2-\eta} + |x|^2}{2}}\Big\|_{L^2(A(C^{-1}\al^\frac{1}{2\de}, 2C^{-1}\al^\frac{1}{2\de}) \times \{ T_0\})}\\
    & \qquad \qquad \qquad \qquad +  C\Big\||\na u| e^{\frac{\al(T_0 - t)|x|^{2-\eta} + |x|^2}{2}}\Big\|_{L^2(A(\rho, 2C^{-1}\al^\frac{1}{2\de}) \times \{ T_0\})}.
\end{align*}
Now we choose $\eta$ and $T_0$ in order to bound the left-hand side. First, we choose $T_0$. By the mean value theorem and \eqref{equation bounded L2 norm with exponential weight}, we can choose $T_0 \in (s_1/2, s_1)$ such that 
\begin{align*} \int_{(\bR^n \setminus B_R) \times \{T_0\} } (|u| + |\na u|)^2 e^{8|x|} dx & \le \frac{2}{s_1} \|(|u| + |\nabla u|) e^{4|x|^2}\|^2_{L^2(Q_{R,s_1})} \\
& \le C,
\end{align*}
where $C > 0$ is a dimensional constant.  Since $\al(T_0 - t)|x|^{2-\eta} + |x|^2 = |x|^2$ when $t=T_0$, we have 
\begin{align*}
    &C\Big\||u| e^{\frac{\al(T_0 - t)|x|^{2-\eta} + |x|^2}{2}}\Big\|_{L^2(A(\rho, 2\rho) \times \{ T_0\})} + C\Big\||u| e^{\frac{\al(T_0 - t)|x|^{2-\eta} + |x|^2}{2}}\Big\|_{L^2(A(C^{-1}\al^\frac{1}{2\de}, 2C^{-1}\al^\frac{1}{2\de}) \times \{ T_0\})}\\
    & \qquad \qquad +  C\Big\||\na u| e^{\frac{\al(T_0 - t)|x|^{2-\eta} + |x|^2}{2}}\Big\|_{L^2(A(\rho, 2C^{-1}\al^\frac{1}{2\de}) \times \{ T_0\})} \le 3C.
\end{align*}

Next, we choose $\eta$ so that 
\[ \al(T_0 - t)|x|^{2-\eta} + |x|^2 \le 4|x|^2
\]
on the domain $A(C^{-1}\al^\frac{1}{2\de}, 2C^{-1}\al^\frac{1}{2\de}) \times [0, T_0]$ for sufficiently large values of $\al$. Factoring out $|x|^2$ from this expression gives the following condition.
\begin{align*}
    \al(T_0 - t)|x|^{-\eta} + 1 &\le 4
\end{align*}
To find a sufficient condition for $\eta$, we assume that $|T_0 - t| \le 1$ and plug in $|x| = 2C^{-1} \al^\frac{1}{2\de}$. We then find $\eta \in (0,1)$ such that the following inequality is fulfilled for sufficiently large values of $\al$.
\begin{align*}
    C^\eta\al\al^{-\frac{\eta}{2\de}} &\le 3
\end{align*}
Since $1 < C^\eta < C$, this inequality is guaranteed for $\al$ large if
\begin{align*}
    1 - \frac{\eta}{2\de} &< 0\\
     \implies 2\de &< \eta.
\end{align*}
Since $\de \in (0, \frac{1}{2})$, it is always possible to find such $\eta(\de) \in (0,1)$. Choosing and fixing such a value for $\eta$ and applying \eqref{equation bounded L2 norm with exponential weight} yields the following inequality.
\begin{align*}
   &\Big\|(|u| + |\na u|) e^{\frac{\al(T_0 - t)|x|^{2-\eta} + |x|^2}{2}}\Big\|_{L^2(A(C^{-1}\al^\frac{1}{2\de}, 2C^{-1}\al^\frac{1}{2\de}) \times [0, T_0])}\\
   & \qquad \le \|(|u| + |\nabla u|) e^{4|x|^2}\|_{L^2(Q_{R,s_1})} \\
   &\qquad \le C.
\end{align*}
Combining results gives the inequality
\begin{align*}
    &\Big\|(|u| + |\na u|) e^{\frac{\al(T_0 - t)|x|^{2-\eta} + |x|^2}{2}}\Big\|_{L^2(A(2\rho, C^{-1}\al^\frac{1}{2\de}) \times [0, T_0])} \\
    &\qquad \le  C\Big\|(|u| + |\na u|) e^{\frac{\al(T_0 - t)|x|^{2-\eta} + |x|^2}{2}}\Big\|_{L^2(A(\rho, 2\rho) \times [0, T_0])} +C.
\end{align*}
Now, we can directly apply the method of \cite{EscauriazaSereginSverak} and \cite{kotschwarwang2015} to obtain the vanishing. Let $\te \ge 1$ be a constant depending on $\rho$ to be determined. On the domain $Q_{8\te \rho,\tfrac{T_0}{2}}$, the following lower bound holds.
\[ \exp(4\al T_0 (\te \rho)^{2-\eta}) \le \exp(\al(T_0 - t)|x|^{2-\eta} + |x|^2).
\]
On $A(\rho, 2 \rho) \times [0,T_0]$, the following upper bound holds.
\[ \exp(\al(T_0 - t)|x|^{2-\eta} + |x|^2) \le \exp(4\rho^2(\al T_0 + 1)).
\]
Plug these inequalities in to obtain
\begin{align*}
    &\Big\||u| + |\na u|\Big\|_{L^2(A(8 \te \rho, C^{-1}\al^\frac{1}{2\de}) \times [0, T_0/2])} \\
    &\qquad \le  Ce^{-4\al T_0 (\te \rho)^{2-\eta}}e^{4\rho^2(\al T_0 + 1)} \Big\||u| + |\na u|\Big\|_{L^2(A(\rho, 2\rho) \times [0, T_0])} + Ce^{-4\al T_0 (\te \rho)^{2-\eta}}.
\end{align*}
Choose $\te > 1$ such that $\te^{2-\eta} > 2\rho^\eta$. Then,
\[4\rho^2(1 + \al T_0 - \al T_0 \te^{2-\eta}\rho^{-\eta}) < 4\rho^2(1 - \al T_0). 
\]
Send $\al \rightarrow +\infty$ to obtain that 
\begin{equation}\label{equation vanishing outside an annulus}
    \Big\||u| + |\na u|\Big\|_{L^2(Q_{8\te \rho,T_0 / 2})} = 0.
\end{equation}
This tells us that $u$ vanishes identically outside of a large annulus $\bR^n \setminus B_{8\te \rho}$. However, since $u$ now must be supported in $B_{8\te \rho}$, it must also satisfy the following differential inequality.
\[|\De u + \dd_t u| \le M|\na u| + (N|8\te \rho|^{\de} + M)|u|.
\]
We may then apply \cite[Theorem 1]{EscauriazaSereginSverak} to obtain that $u$ vanishes on all of $Q_{R,T_0/2}$. Since $T_0$ is uniformly bounded below by $s_1/2$, iterating as in the proof of \cite[Theorem 1]{EscauriazaSereginSverak} suffices to prove that $u$ vanishes on all of $Q_{R,T}$.
\end{proof}

\bibliographystyle{amsplain}
\bibliography{bibliography}

\end{document}